\documentclass[10pt,twoside, a4paper, english, reqno]{amsart}
\usepackage[dvips]{epsfig}
\usepackage{a4wide}
\usepackage{amssymb}
\usepackage{amsthm}
\usepackage{amsmath}
\usepackage{latexsym}
\usepackage{mathtools,dsfont}

\usepackage{upref}
\usepackage[colorlinks,backref]{hyperref}
\usepackage{color}
\usepackage[usenames,dvipsnames]{xcolor}
\usepackage[toc,page]{appendix}
\usepackage{pdfsync,cancel}

\newcommand{\ep}{\varepsilon}
\newcommand{\sgn}{\mathrm{sign}}
\newcommand{\D}{{\ensuremath{\R^d}}}
\newcommand{\bP}{{\mathbb{P}}}
\newcommand{\intrd}{\int_{\mathbb{R}^{d}}}

\newcommand{\tu}{ \widetilde{u}^{\Delta t}}
\newcommand{\bu}{\widetilde{B}^{\Delta t}}

\newcommand{\frt}{\mathcal{L}_{\lambda /2}}
\newcommand{\Rd}{\mathbb{R}^d}
\newcommand{\ws}{\overset{\ast}{\rightharpoonup}}
\newcommand{\wc}{{\rightharpoonup}}
\theoremstyle{plain}
\newtheorem{thm}{Theorem}[section]
\theoremstyle{plain}
\newtheorem{lem}[thm]{Lemma}
\newtheorem{prop}[thm]{Proposition}
\newtheorem{cor}[thm]{Corollary}

\theoremstyle{definition}
\newtheorem{defi}{Definition}[section]
\newtheorem{rem}{Remark}[section]

\newtheorem*{maintheorem*}{Main Theorem}
\newtheorem*{maincorollary*}{Main Corollary}

\newenvironment{Assumptions}
{
\setcounter{enumi}{0}

\begin{enumerate}}
{\end{enumerate} }

\def\cprime{$'$}
\newcommand{\norm}[1]{\left\|#1\right\|}

\newcommand{\fr}{\mathcal{L}_{\lambda}}
\newcommand{\R}{\ensuremath{\mathbb{R}}}
\newcommand{\Div}{\mathrm{div}\,}
\newcommand{\rd}{\ensuremath{\mathbb{R}^d}}
\newcommand{\supp}{\ensuremath{\mathrm{supp}\,}}
\newcommand{\goto}{\ensuremath{\rightarrow}}
\newcommand{\grad}{\ensuremath{\nabla}}
\newcommand{\eps}{\ensuremath{\varepsilon}}

\newcommand{\E}{\mathbb{E}}

\def\f{ \vec{\mathbf{f}} }

\def\dint{\displaystyle\int}
\def\d#1{\mathrm{ d#1}}

\numberwithin{equation}{section} \allowdisplaybreaks

\title[Stochastic Fractional Conservation Laws]
{The Cauchy problem for a fractional conservation laws driven by L\'{e}vy noise}

\date{\today}


\author[Neeraj Bhauryal]{Neeraj Bhauryal}
\address[Neeraj Bhauryal]{\newline
	Centre for Applicable Mathematics,
	Tata Instiute of Fundamental Research,
	P.O.\ Box 6503, GKVK Post Office,
	Bangalore 560065, India}
\email[]{neeraj@math.tifrbng.res.in}

\author[Ujjwal Koley]{Ujjwal Koley}
\address[Ujjwal Koley]{\newline
	Centre for Applicable Mathematics,
	Tata Instiute of Fundamental Research,
	P.O.\ Box 6503, GKVK Post Office,
	Bangalore 560065, India}
\email[]{ujjwal@math.tifrbng.res.in}

\author[Guy Vallet]{Guy Vallet}
\address[Guy Vallet]{\newline 
	LMAP UMR- CNRS 5142, IPRA BP 1155, 64013 Pau Cedex, France}
\email[]{guy.vallet@univ-pau.fr}

\begin{document}
\begin{abstract}
In this article, we explore some of the main mathematical problems connected to multidimensional fractional conservation laws driven by L\'{e}vy processes. Making use of an adapted entropy formulation, a result of existence and uniqueness of a solution is established. Moreover, using bounded variation (BV) estimates for vanishing viscosity approximations, we derive an explicit continuous dependence estimate on the nonlinearities of the entropy solutions under the
assumption that L\'{e}vy noise depends only on the solution. This result is used to show the error estimate for the stochastic vanishing viscosity method. Furthermore, we establish a result on vanishing non-local regularization of scalar stochastic conservation laws. 
\end{abstract}

\maketitle

\tableofcontents


\section{Introduction}
The analytical study of almost all physical phenomenon, in areas including physics, engineering, finance and
biology, involves nonlinear partial differential equations (PDEs). It is immensely important to be able to take the
inherent uncertainties into account in one's attempt to describe these phenomenon and a systematic study of PDEs with randomness (stochastic PDEs) certainly leads to greater understanding of the actual physical phenomenon. In this paper, we are interested in the well posedness theory for stochastic fractional hyperbolic-parabolic equation driven by multiplicative L\'{e}vy noise. Nonlocal operator appears in mathematical models for viscoelastic materials, fluid flows and acoustic propagation in porous media, and pricing derivative securities in financial markets \cite{black}.

A formal description of our problem requires a filtered probability space $\big(\Omega, \mathbb{P}, \mathcal{F}, \{\mathcal{F}_t\}_{t\ge 0} \big)$, and we are interested in an $L^2$-valued predictable process $u(t,\cdot)$
which satisfies the following Cauchy problem
\begin{align}
\label{eq:stoc_con_brown}
\begin{cases} 
du(t,x)- \mbox{div} f(u(t,x)) \,dt + \mathcal{L}_{\lambda}[u(t, \cdot)](x)\,dt =
\sigma(u(t,x))\,dW(t) + \int_{E} \eta(u(t,x); z)\,\widetilde{N}(dz,dt), & \text{ in } \Pi_T, \vspace{0.1cm}\\
u(0,x) = u_0(x), &  \text{ in } \R^d,
\end{cases}
\end{align}
where $\Pi_T:= \R^d \times (0,T)$ with $T>0$ fixed, $u_0:\R^d \mapsto \R$ is the given initial
function, $f:\R \mapsto \R^d$ is a given (sufficiently smooth) vector valued flux function 
(see Section~\ref{sec:tech} for the complete list of assumptions), and $\fr[u]$ denotes the fractional Laplace operator $(-\Delta)^{\lambda}[u]$ of order $\lambda \in (0,1)$, and defined as
\begin{align*}
\fr[\varphi](x) := c_{\lambda}\, \text{P.V.}\, \int_{|z|>0} \frac{\varphi(x) -\varphi(x+z)}{|z|^{d + 2 \lambda}} \,dz,
\end{align*}
for some constants $c_{\lambda}>0$, and a sufficiently regular function $\varphi$.

Note that $W(t)$ is a cylindrical Wiener process: $W(t)= \sum_{k\ge 1} e_k \beta_k(t)$ with $(\beta_k)_{k\ge 1}$ being mutually independent real valued standard Wiener processes and $(e_k)_{k\ge 1}$ a complete orthonormal system in a separable Hilbert space $\mathbb{H}$. Furthermore, $ \widetilde{N}(dz,dt)= N(dz,dt)-m(dz)\,dt $ is an independent compensated Poisson random measure, where $N$ is a Poisson random measure on $(E,\mathcal{E})$ with intensity measure $m(dz)$, and $(E,\mathcal{E},m)$ is a $\sigma$-finite measure space.

Finally, $ u \mapsto \sigma (u)$ is an $\mathbb{H}$-valued function and $(u,z)\mapsto \eta(u,z)$ is a given real valued function signifying the multiplicative nature of the noise.
\begin{rem}
We want to make the following comments:
\begin{itemize}
\item [(a)] Note that in case of $1/2 < \lambda <1$, the non-local term $\fr[u]$ is the dominant term. Hence the equation \eqref{eq:stoc_con_brown} becomes a parabolic equation and existence of solution for such equation can be obtained by a fixed point or contraction mapping argument. So the only interesting case corresponds to $\lambda$ in the interval $(0, 1/2]$. However, our entire analysis is independent of the choice of $\lambda$, therefore, we present our results for all $\lambda \in (0,1)$.
\item [(b)] All our results can be extended to the more general explicit space dependent noise coefficients, i.e., 
$\sigma = \sigma(x,u)$ and $\eta = \eta(x,u;z)$. However for technical reasons, in view of \cite{BaVaWit_2012,
BisMajKarl_2014}, we need further assumptions on noise coefficients and we choose not to give details on that direction. 
\item [(c)] We will carry out our analysis under the structural assumption $E = \mathcal{O}\times \R^*$, where $\mathcal{O}$ is a subset of the Euclidean space. The measure $m$ on $E$ is defined as $\lambda \times \mu$ where $\lambda$ is a Radon measure on $\mathcal{O}$ and $\mu$ is so-called L\'{e}vy measure on $\R^*$. Such a noise would be called an impulsive white noise with jump position intensity $\lambda$ and jump size intensity $\mu$.  We refer to \cite{peszat} for more on L\'{e}vy sheet and related impulsive white noise. 
Moreover, for each $v\in L^2(\R^d)$, we consider the mapping $\sigma(v): \mathbb{H}\goto L^2(\R^d)$ defined by 
$\sigma(v)e_k= g_k(v(\cdot))$. In particular, we suppose that $g_k$ is Lipschitz-continuous, and $\mathbb{G}^2(r)= \sum_{k\ge 1} g_k^2(r)$.
\end{itemize}
\end{rem}

The equation \eqref{eq:stoc_con_brown} could be viewed as a stochastic perturbation of non-local hyperbolic equation. In the absence of non-local term along with the case $\sigma=\eta=0$, the equation \eqref{eq:stoc_con_brown} becomes a standard conservation law in $\R^d$. For conservation laws, the question of existence and uniqueness of solutions was first settled in the pioneer papers of Kru\v{z}kov \cite{Kruzkov} and Vol'pert \cite{Volpert}. In the case $\sigma=\eta=0$, well-posedness of Cauchy problem was studied by Alibaud \cite{Alibaud}, Cifani \& Jakobsen \cite{CifaniJakobsen}. 

\subsection{Stochastic Balance Laws}   
The study of stochastic balance laws has so far been limited mostly to equations of the type \eqref{eq:stoc_con_brown} in the absence of the non-local term $\fr[u]$. 
In fact, Kim \cite{KIm2005} extended the Kru\v{z}kov well-posedness theory to one dimensional balance laws that are driven by 
additive Brownian noise, and Vallet \& Wittbold \cite{Vallet_2009} to the multidimensional Dirichlet problem. However, when the noise 
is of multiplicative nature, one could not apply a straightforward Kru\v{z}kov's doubling method to get uniqueness. 
This issue was settled by Feng $\&$ Nualart \cite{nualart:2008}, who established uniqueness of entropy solution by 
recovering additional information from the vanishing viscosity method. The existence was proven using stochastic 
version of compensated compactness method and it was valid for \emph{one} spatial dimension.
To overcome this problem, Debussche $\&$ Vovelle \cite{Vovelle2010} introduced
kinetic formulation of such problems and as a result they were able to establish the well-posedness
of multidimensional stochastic balance law via \emph{kinetic} approach. 
A number of authors have contributed since then, and we mention the works of 
Bauzet et al. \cite{BaVaWit_2012,BaVaWit_JFA}, Biswas et al. \cite{BisMaj,BisMajKarl_2014}. For degenerate
parabolic equations, we mention works of Vallet \cite{Vallet_2005, Vallet_2008}, Debussche et al. \cite{martina},
Koley et al. \cite{Koley1, Koley2}.
We also mention works by Chen et al. \cite{Chen-karlsen_2012}, and Biswas et al. \cite{BisKoleyMaj}, where well
posedness of entropy solution is established in $L^p \cap BV$, via BV framework. 
Moreover, they were able to develop continuous dependence theory for multidimensional balance laws and as a by
product they derived an explicit \emph{convergence rate} of the approximate solutions to the underlying problem. 

We remark that our solution concept is different from the concept of \emph{random entropy solution} for fractional
conservation laws incorporating randomness in the initial data and fluxes. Several results are available in that
direction. For well-posedness theory of random entropy solution, we refer to \cite{Koley3, koley2013multilevel,
ujjwal}.

Independently of the smoothness of the initial data $u_0$, due to the presence of nonlinear flux term, degenerate diffusion term, and a nonlocal term in equation \eqref{eq:stoc_con_brown}, solutions to \eqref{eq:stoc_con_brown} are not necessarily smooth and weak solutions must be sought. Before introducing the concept of weak solutions, we first recall the notion of predictable $\sigma$-field. By a predictable $\sigma$-field on $[0,T]\times \Omega$, denoted
by $\mathcal{P}_T$, we mean that the $\sigma$-field generated by the sets of the form: $ \{0\}\times A$ and $(s,t]\times B$  for any $A \in \mathcal{F}_0; B \in \mathcal{F}_s,\,\, 0<s,t \le T$.
The notion of stochastic weak solution is defined as follows:
 
\begin{defi}[Stochastic weak solution]
\label{defi:weak-solution}
 A square integrable $ L^2(\R^d )$-valued $\{\mathcal{F}_t: t\geq 0 \}$-predictable stochastic process $u(t)= u(t,x)$ is said to be a weak solution
 to the problem \eqref{eq:stoc_con_brown} if, $ \mathbb{P}-$ a.s, for all $\psi \in \mathcal{D}([0,T) \times \R^d)$
 \begin{align}
& \int_{\R^d} \psi(0,.) u_0\,dx  + \int_{\Pi_T} \Big\{  u(t,x) \partial_t \psi(t,x) 
 - f(u(t,x)) \partial_x \psi(t,x) \Big\}\,dx\,dt - \int_{\Pi_T}  u(t,x)\mathcal{L}_\lambda[\psi(t,\cdot)](x)\,dx\, dt \notag \\
 &\qquad \qquad + \int_{\Pi_T}  \sigma(u(t,x)) \,\psi(t,x) \,dW(t) \,dx\,dt + \int_{\Pi_T} \int_{E} \eta(u(t,x);z)\, \psi(t,x) \, dx\,dt\,\widetilde{N}(dz,dt) = 0.
\end{align}
\end{defi}
However, it is well-known that weak solutions may be discontinuous and they are not uniquely determined by their initial data. Consequently, an admissibility condition, so called {\em entropy condition}, must be imposed to single out the physically correct solution. Since the notion of entropy solution is built around the so called entropy flux pair, we begin with the definition of entropy flux pair.

\begin{defi}[Entropy-entropy flux pair]
A pair $(\beta,\zeta) $ is called an entropy-entropy flux pair 
if $ \beta \in C^2(\R) $ and $\beta \ge0$, and 
$\zeta = (\zeta_1,\zeta_2,....\zeta_d):\R \mapsto\rd $ is a vector field satisfying
$\zeta'(r) = \beta'(r)f'(r)$ for all $r$. An entropy-entropy flux pair $(\beta,\zeta)$ is called 
convex if $ \beta^{\prime\prime}(\cdot) \ge 0$.  
\end{defi}
With the help of a convex entropy-entropy flux pair $(\beta,\zeta)$, we present a formal
derivation of entropy inequalities.

\subsection{Stochastic Entropy Formulation}
We introduce an entropy formulation for the initial value problem \eqref{eq:stoc_con_brown}. To this end, let us first split the non-local operator $\fr$ into two terms: for each $r>0$, we write $\fr[\varphi] := \mathcal{L}_{\lambda, r}[\varphi] + \mathcal{L}_{\lambda}^{r}[\varphi]$, where
\begin{align*}
\mathcal{L}_{\lambda, r}[\varphi](x)&:= c_{\lambda}\, \text{P.V.}\, \int_{|z|\le r} \frac{\varphi(x) -\varphi(x+z)}{|z|^{d + 2 \lambda}} \,dz, \\
\mathcal{L}_{\lambda}^{r}[\varphi](x)&:= c_{\lambda}\,\int_{|z|> r} \frac{\varphi(x) -\varphi(x+z)}{|z|^{d + 2 \lambda}} \,dz.
\end{align*}
For a small positive number $\eps>0$, assume that the parabolic perturbation 
\begin{align}
 du_\eps(t,x) -\eps \Delta u_\eps(t,x)\,dt & + \mathcal{L}_{\lambda}[u_\eps(t, \cdot)](x)\,dt  - \mbox{div}_x f(u_\eps(t,x)) \,dt \notag \\
&\qquad \qquad \qquad \qquad= \sigma(u_\eps(t,x))\,dW(t) + \int_{E} \eta(u_\eps(t,x);z)\widetilde{N}(dz,dt), \label{eq:viscous-Brown} 
\end{align}
of \eqref{eq:stoc_con_brown} has a unique weak solution $u_\eps(t,x)$ with initial data $u_{\eps}(0,x)=u_0^{\eps}(x)\in H^1(\R^d)$, where $u_0^{\eps}$ converges to $u_0$ in $L^2(\R^d)$. Note that this weak solution 
$u_{\eps} \in L^2((0,T)\times \Omega;H^2(\R^d))$ (see Appendix~\ref{viscous}), so that, in particular, $\mathcal{L}_{\lambda}[u_\epsilon]$ and $\mathcal{L}_{\lambda, r}[u_\epsilon]$ are elements of $L^2(\R^d)$. This enables one to derive a weak version of the It\^{o}-L\'{e}vy formula (as proposed in \cite{martina,fellah,BisMajVal}) for the solutions of \eqref{eq:stoc_con_brown}.
\\
Let $(\beta,\zeta)$ be an entropy flux pair. Given a non-negative test function $\psi\in C_{c}^{1,2}([0,\infty )\times\R^d)$, we apply a generalized version of It\^{o}-L\'{e}vy formula to yield, for almost every $\bar{T}>0$
\begin{align*}
& \int_{\R^d} \beta(u_{\eps}(\bar{T},x))\, \psi(\bar{T},x)\,dx -  \int_{\R^d} \beta(u_{\eps}(0,x))\, \psi(0,x)\,dx \\
 &=  \int_{\Pi_T} \beta(u_{\eps}(t,x)) \,\partial_t\psi(t,x) \,dx\,dt-  \int_{\Pi_T}   \grad \psi(t,x)\cdot \zeta(u_{\eps}(t,x)) \,dx\,dt \notag \\
 & + \sum_{k\ge 1}\int_{\Pi_T} g_k(u_\eps(t,x))\beta^\prime (u_\eps(t,x))\psi(t,x)\,d\beta_k(t)\,dx
 + \frac{1}{2}\int_{\Pi_T}\mathbb{G}^2(u_\eps(t,x))\beta^{\prime\prime} (u_\eps(t,x))\psi(t,x)\,dx\,dt \notag \\
  &  + \int_{\Pi_T} \int_{E} \int_0^1 \eta(u_\eps(t,x);z)\beta^\prime \big(u_\eps(t,x) + \lambda\,\eta(u_\eps(t,x);z)\big)\psi(t,x)\,d\lambda\,\widetilde{N}(dz,dt)\,dx  \notag \\
 & + \int_{\Pi_T} \int_{E}  \int_0^1  (1-\lambda)\eta^2(u_\eps(t,x);z)\beta^{\prime\prime} \big(u_\eps(t,x) + \lambda\,\eta(u_\eps(t,x);z)\big)
 \psi(t,x)\,d\lambda\,m(dz)\,dx\,dt \notag \\
 &  - \int_{\Pi_T} \left( \eps  \grad \psi(t,x)\cdot \grad \beta(u_\eps(t,x)) + \eps \beta^{\prime\prime}(u_\eps(t,x))\, |\grad u_\eps(t,x)|^2 \psi(t,x) \right)\,dx\,dt\\
 & -  \int_{\Pi_T} \Big< \mathcal{L}_{\lambda}[u_\eps(t, \cdot)](x), \psi(t,x)\, \beta'(u_\eps(t,x)) \Big> \,dx\,dt.
 \end{align*}
 
We only need to modify the non-local term, since rest of the terms can be manipulated in usual manner. 
To that context, note that $u_{\eps} \in H^2(\R^d)$ and that $\psi$ has compact support, so following \cite{CifaniJakobsen}, for a fixed positive $r$, by using arguments developed in Appendix~\ref{FractionalLaplace}, we have 
\begin{align}\label{the Ir}
& \int_{\Pi_T} \Big< \mathcal{L}_{\lambda}[u_\eps(t, \cdot)](x), \psi(t,x)\, \beta'(u_\eps(t,x)) \Big> \,dx\,dt 
\nonumber \\ 
& \qquad \qquad = \int_{\Pi_T\times\{|z|\leq r\}}\beta'(u_\eps(t,x)) \, \psi(t,x)\,  \frac{u_\eps(t,x) -u_\eps(t,x+z)+z.\nabla u_\eps(t,x)}{|z|^{d + 2 \lambda}} \,dz  \,dx\,dt \nonumber \\
& \qquad \qquad \qquad \qquad +  \int_{\Pi_T}  \int_{|z|> r} \beta'(u_\eps(t,x)) \, \psi(t,x)\,  \frac{u_\eps(t,x) -u_\eps(t,x+z)}{|z|^{d + 2 \lambda}} \,dz \,dx\,dt \nonumber \\ 
& \qquad \qquad = \lim_{\delta \goto 0} \Bigg[\int_{\Pi_T\times\{\delta<|z|\leq r\}}\beta'(u_\eps(t,x)) \, \psi(t,x)\,  \frac{u_\eps(t,x) -u_\eps(t,x+z)+z.\nabla u_\eps(t,x)}{|z|^{d + 2 \lambda}} \,dz  \,dx\,dt\Bigg] \nonumber \\
& \qquad \qquad \qquad \qquad +  \int_{\Pi_T}  \int_{|z|> r} \beta'(u_\eps(t,x)) \, \psi(t,x)\,  \frac{u_\eps(t,x) -u_\eps(t,x+z)}{|z|^{d + 2 \lambda}} \,dz \,dx\,dt \nonumber \\ 
& \qquad \qquad = \lim_{\delta \goto 0} \Bigg[\int_{\Pi_T \times \{\delta < |z|\leq r\}} \beta'(u_\eps(t,x)) \, \psi(t,x)\,  \frac{u_\eps(t,x) -u_\eps(t,x+z)}{|z|^{d + 2 \lambda}} \,dz \,dx\,dt \Bigg] \nonumber \\
& \qquad \qquad \qquad \qquad +  \int_{\Pi_T}  \int_{|z|> r} \beta'(u_\eps(t,x)) \, \psi(t,x)\,  \frac{u_\eps(t,x) -u_\eps(t,x+z)}{|z|^{d + 2 \lambda}} \,dz \,dx\,dt :=\mathcal{G}_r + \mathcal{G}^r
\end{align}
Next, observe that since $u_\eps(t) \in H^2(\R^d)$ and that $\psi$ has compact support, the term $\mathcal{G}^r$ is well defined. To deal with the other term $\mathcal{G}_r$, first note that for any $a$ and $b$, $(a-b) \,\beta^\prime(a) \geq \beta(a)- \beta(b)\geq (a- b)\, \beta^\prime(b)$. Therefore, we get
\begin{align*}
\big[\beta(u_\eps(t,x)) -\beta(u_\eps(t,x+z))\big] \le \big[u_\eps(t,x)- u_\eps(t,x+z)\big] \, \beta'(u_\eps(t,x))).
\end{align*}
Then a simple change of variable formula, and similar arguments to the above ones reveal that
\begin{align*}
 \mathcal{G}_r \ge & \lim_{\delta \goto 0} \Bigg[\int_{\Pi_T \times \{\delta < |z|\leq r\}} \bigg[\frac{\beta(u_\eps(t,x)) -\beta(u_\eps(t,x+z))}{|z|^{d + 2 \lambda}} \,dz \bigg]\, \psi(t,x) \,dx\,dt \Bigg] \nonumber \\
=& \int_{\Pi_T} \beta(u_\epsilon(t,x)) \,\text{P.V} \,\int_{\{|z|\leq r\}}\frac{\psi(x)-\psi(x+z)}{|z|^{d+2\lambda}}\,dz\,dx\,dt = \int_{\Pi_T} \beta(u_\epsilon(t,x)) \mathcal{L}_{\lambda, r}[\psi](x)\,dx\,dt.
\end{align*}
Since $\beta$ and $\psi$ are non-negative functions, we obtain
\begin{align}
0 & \le   \int_{\R^d} \beta(u_{\eps}(0,x))\, \psi(0,x)\,dx + 
\int_{\Pi_T} \Big\{\beta(u_{\eps}(t,x)) \,\partial_t\psi(t,x) -  \grad \psi(t,x)\cdot \zeta(u_{\eps}(t,x)) \Big\}\,dx\,dt \notag \\
& -\int_{\Pi_T} \Big[ \mathcal{L}_{\lambda}^{r}[u_\eps(t,\cdot)](x)\, \psi(t,x)\, \beta'(u_\eps(t,x)) + \beta(u_\eps(t,x)) \, \mathcal{L}_{\lambda, r}[\psi(t,\cdot)](x) \Big]\,dx\,dt+ \mathcal{O}(\eps) \notag\\
 & + \sum_{k\ge 1}\int_{\Pi_T} g_k(u_\eps(t,x))\beta^\prime (u_\eps(t,x))\psi(t,x)\,d\beta_k(t)\,dx
 + \frac{1}{2}\int_{\Pi_T}\mathbb{G}^2(u_\eps(t,x))\beta^{\prime\prime} (u_\eps(t,x))\psi(t,x)\,dx\,dt \notag \\
  &  + \int_{\Pi_T} \int_{E} \int_0^1 \eta(u_\eps(t,x);z)\beta^\prime \big(u_\eps(t,x) + \lambda\,\eta(u_\eps(t,x);z)\big)\psi(t,x)\,d\lambda\,\widetilde{N}(dz,dt)\,dx  \notag \\
 & +\int_{\Pi_T} \int_{E}  \int_0^1  (1-\lambda)\eta^2(u_\eps(t,x);z)\beta^{\prime\prime} \big(u_\eps(t,x) + \lambda\,\eta(u_\eps(t,x);z)\big) \psi(t,x)\,d\lambda\,m(dz)\,dx\,dt. \label{imp_02}
 \end{align}
Clearly, the above inequality is stable under the limit $\eps \rightarrow 0$, if the family ${\lbrace u_{\eps} \rbrace}_{\eps>0}$ has
$L^p_{\mathrm{loc}}$-type stability. Just as the deterministic equations, the above inequality provides us the entropy condition. We now  introduce the notion of stochastic entropy solution as follows: 

\begin{defi} [Stochastic Entropy Solution]
 \label{defi:stochentropsol}
A square integrable $ L^2(\R^d )$-valued $\{\mathcal{F}_t: t\geq 0 \}$-predictable stochastic process $u(t)= u(t,x)$ is called a stochastic entropy solution of \eqref{eq:stoc_con_brown} if given a non-negative test function $\psi\in C_{c}^{1,2}([0,\infty )\times\R^d) $ and a convex entropy flux pair $(\beta,\zeta)$, the following inequality holds:
\begin{align}
 &  \int_{\R^d} \beta(u_0(x))\psi(0,x)\,dx + \int_{\Pi_T} \Big\{ \beta(u(t,x)) \partial_t\psi(t,x) -  \grad \psi(t,x)\cdot \zeta(u(t,x)) \Big\}dx\,dt \notag \\
 & -\int_{\Pi_T} \Big[ \mathcal{L}_{\lambda}^{r}[u(t,\cdot)](x)\, \psi(t,x)\, \beta'(u(t,x)) + \beta(u(t,x)) \,\mathcal{L}_{\lambda, r}[\psi(t,\cdot)](x) \Big]\,dx\,dt \notag \\
 & 
 + \sum_{k\ge 1}\int_{\Pi_T} g_k(u(t,x))\beta^\prime (u(t,x))\psi(t,x)\,d\beta_k(t)\,dx
 + \frac{1}{2}\int_{\Pi_T}\mathbb{G}^2(u(t,x))\beta^{\prime\prime} (u(t,x))\psi(t,x)\,dx\,dt \notag \\
  &  
  + \int_{\Pi_T} \int_{E} \int_0^1 \eta(u(t,x);z)\beta^\prime \big(u(t,x) + \lambda\,\eta(u(t,x);z)\big)\psi(t,x)\,d\lambda\,\widetilde{N}(dz,dt)\,dx  
 \notag \\
 & 
 +\int_{\Pi_T} \int_{E}  \int_0^1  (1-\lambda)\eta^2(u(t,x);z)\beta^{\prime\prime} \big(u(t,x) + \lambda\,\eta(u(t,x);z)\big)
\psi(t,x)\,d\lambda\,m(dz)\,dx\,dt  \ge  0, \quad \mathbb{P}-\text{a.s}.\label{inq:entropy-solun}
\end{align}
\end{defi} 
In what follows, we will use explicitly the inequality \eqref{imp_02} in the sequel to establish the well posedness theory of the entropy solution in the sense of Definition \ref{defi:stochentropsol}.

\subsection{Scope and Outline of the Paper}
As we mentioned earlier, past few years have witnessed remarkable advances on the area of deterministic non-local/fractional conservation laws. An worthy reference on this subject is \cite{CifaniJakobsen}. However, very little is available on the specific problem of fractional conservation laws driven by L\'{e}vy noise, and there are number of issues waiting to be explored. To fill the gap between the stochastic theory and its deterministic counterpart, we aim to present a complete well-posedness theory for the problem \eqref{eq:stoc_con_brown}. We emphasize that the analysis presented in this manuscript differs significantly from the deterministic analysis, partly due to the technical obstacle that we can not pass to the limit in the parameter $\xi$ (related to the approximation of absolute value function) at the beginning.

To sum up, we aim at developing following results related to \eqref{eq:stoc_con_brown}:
\begin{itemize}
\item [(a)] We first propose to prove a result of existence and uniqueness of a stochastic entropy solution of \eqref{eq:stoc_con_brown}, in the sense of Definition \ref{defi:stochentropsol}, using the concept of measure-valued solutions and a variant of Kru\v{z}kov's entropy formulation. We also derive stability estimate with respect to the initial data.
\item [(b)] Drawing preliminary motivation from \cite{BisKoleyMaj,Chen-karlsen_2012,karl-resibro-2000},
we intend to develop a continuous dependence theory for stochastic entropy solution which in turn can be used to derive an error estimate for the vanishing viscosity method. However, it seems difficult to develop such a theory without securing a BV estimate for stochastic entropy solution. As a result, we first address the question of existence, uniqueness of stochastic BV entropy solution in $L^2(\R^d) \cap BV(\R^d)$ of the problem \eqref{eq:stoc_con_brown}. Making use of the crucial BV estimate, we provide a continuous dependence estimate and error estimate for the vanishing viscosity method provided initial data lies in $u_0 \in L^2(\R^d) \cap BV(\R^d)$.
\item [(c)] Finally, following \cite{Alibaud}, we also consider a non-local regularization of scalar stochastic conservation laws by adding a fractional power of the laplacian. Then, making use of the BV estimate, we derive an explicit \emph{convergence rate} of the approximate solutions to the unique entropy solution of the stochastic conservation laws. 
\end{itemize}

The rest of the paper is organized as follows:  we describe technical frameworks and state main results in
Section~\ref{sec:tech}. In Section~\ref{uniqueness}, we establish well-posedness theory for the problem under
consideration \eqref{eq:stoc_con_brown}. Next, making use of BV estimates, we derive an explicit continuous
dependence estimate on nonlinearities in Section~\ref{cont-depen-estimate} and present the error estimate for the
stochastic vanishing viscosity method in Section~\ref{sec:cor}. Section~\ref{sec:cor_01} deals with a non-local
regularization of the equation~\eqref{eq:stoc_original}, and derive an explicit rate of convergence estimate of the approximate
solutions \eqref{eq:stoc_regularization} to the unique entropy solution of \eqref{eq:stoc_original}. Furthermore, in Appendix~\ref{viscous}, we
demonstrate the existence and uniqueness results related to the viscous equation \eqref{eq:viscous-Brown}, while 
in Appendix~\ref{sec:apriori+existence}, we derive uniform spatial BV bound for viscous solutions. Using this bound,
we establish well posedness of BV entropy solution of the Cauchy problem \eqref{eq:stoc_con_brown}.
Finally Appendix~\ref{FractionalLaplace}, and Appendix~\ref{SolPositive} recapitulates some existing results on fractional operator. 
  

\section{Technical Framework and Statement of the Main Results}
\label{sec:tech}
Throughout this paper, we use the letter $C$ to denote various generic constants. There are situations where constant may change from line to line, but the notation is kept unchanged so long as it does not impact central idea. 
Moreover, for any separable Hilbert space $H$, we denote by $N_w^2(0,T,H)$, the Hilbert space of all the predictable $H$-valued processes $u$ such that $\E\Big[\int_0^T \|u\|^2_H\Big]<+\infty$.
Furthermore, we denote $BV(\R^d)$ as the set of integrable 
 functions with bounded variation on $\R^d$ endowed with the norm $|u|_{BV(\R^d)}= \|u\|_{L^1(\R^d)} + TV_{x}(u)$, where $TV_{x}$ is the total variation of $u$ defined on $\R^d$. Next, we write down some useful properties of the fractional operator which are used in the sequel, for a detailed description, consult Appendix~\ref{FractionalLaplace}. First note that 
\begin{align*}
\mathcal{L}_{\lambda}[\varphi](x)&= c_{\lambda}\, \text{P.V.}\, \int_{|z|\le r} \frac{\varphi(x) -\varphi(x+z)}{|z|^{d + 2 \lambda}} \,dz + c_{\lambda}\,\int_{|z|> r} \frac{\varphi(x) -\varphi(x+z)}{|z|^{d + 2 \lambda}} \,dz \\
& = c_{\lambda}\, \int_{|z|\le r} \frac{\varphi(x) -\varphi(x+z) + z \cdot \nabla \varphi(x)}{|z|^{d + 2 \lambda}} \,dz + c_{\lambda}\,\int_{|z|> r} \frac{\varphi(x) -\varphi(x+z)}{|z|^{d + 2 \lambda}} \,dz,
\end{align*}
for some constants $c_{\lambda}$, $\lambda \in (0,1)$, and a sufficiently regular function $\varphi$.
Moreover, for all $u,v \in H^{\lambda}(\D)$, denoting convolution operator by $\star$, we have
\begin{align*}
u \star \mathcal{L}_\lambda [v] &= v \star \mathcal{L}_\lambda [u], \\
\langle\mathcal{L}_\lambda [u],v\rangle &=\frac{c_{\lambda}}{2}\int_{\D}\int_{\D} \frac{\big(u(x)-u(y)\big) \big(v(x)-v(y)\big)}{|x-y|^{d+2\lambda}}\,dx\,dy =\int_{\D} \mathcal{L}_{\lambda/2} [u](x) \,\mathcal{L}_{\lambda/2} [v](x) \,dx.
\end{align*}
The primary aim of this paper is to settle the problem of existence and uniqueness for the Cauchy problem \eqref{eq:stoc_con_brown}, derive continuous dependence estimates for the entropy solutions of the same problem, and we do so under the following assumptions:
 \begin{Assumptions}
 	\item \label{A1} The initial function $u_0$ is a deterministic function satisfying $\|u_0\|_2 < +\infty$.
 	\item \label{A1'} For the stability analysis, we also assume that $|u_0|_{BV(\R^d)} <+\infty$.
 	\item \label{A3}  $ f=(f_1,f_2,\cdots, f_d):\R\rightarrow \R^d$ is a Lipschitz continuous function with $f_k(0)=0$, for 
 	all $1\le k\le d$.
	\item \label{A31}
	 The space $E$ is of the form $\mathcal{O}\times \R^*$ and the Borel measure $m$ on $E$ has the form $\lambda \times \mu$, where $\lambda$ is a Radon measure on $\mathcal{O}$ and $\mu$ is so-called one dimensional L\'{e}vy measure.
 	\item \label{A4} 
We assume that $g_k(0)=0$, for all $k\ge 1$. Moreover, there exists a positive constant $K > 0$  such that 
 	\begin{align*} 
 	\sum_{k\ge 1}\big| g_k(u)-g_k(v)\big|^2  \leq K |u-v|^2, \, \,\,\text{and}\,\,\, \mathbb{G}^2(u)= \sum_{k\ge 1} g_k^2(u)\le K\,|u|^2, ~\text{for all} \,\,u,v \in \R.
 	\end{align*}  
In particular, $\forall u \in L^2(\R^d),\ \sigma(u):\mathbb{H} \to L^2(\R^d)$ is Hilbert-Schmidt ($HS(L^2)$) and $\forall u \in H^1(\R^d),\ \sigma(u):\mathbb{H} \to H^1(\R^d)$ is Hilbert-Schmidt ($HS(H^1)$).
\item \label{A5} 
There exist positive constants $\lambda^* \in (0,1)$, and $h(z)\in L^2(E,m)$ with $0\le h(z)\le 1$ such that for all $u,v \in \R;~~z\in E$
 \begin{align*}
  \big| \eta(u;z)-\eta(v;z)\big|  \leq \lambda^* |u-v|h(z),\, \,\,\text{and}\,\,\, |\eta(u,z)|\le \lambda^* |u|h(z).
 \end{align*}
Moreover, we assume that $\eta(0,z)=0$, for all $z\in E$. 
 \end{Assumptions}

\begin{rem}
We remark that, one can accommodate polynomially growing flux function as a result of the requirement that the entropy solutions satisfy $L^p$ bounds for all $p\ge 2$. This in turn forces to choose initial data that are in $L^p$, for all $p$. However, we have chosen to work with the assumptions ~\ref{A1} and ~\ref{A3}. The assumption ~\ref{A5} is natural in the context of L\'{e}vy noise with the exception of $\lambda^*\in (0,1)$, which is necessary for the uniqueness.
Finally, the assumptions \ref{A1}-\ref{A5} collectively ensures existence and uniqueness of stochastic entropy solution,  and the continuous dependence estimate as well. 
\end{rem}
\begin{rem}
 In view of the assumption \ref{A4}, for any $v\in L^2(\R^d)$, $\sigma(v)$ is a Hilbert-Schmidt operator from the separable Hilbert space $\mathbb{H}$ to $L^2(\R^d)$. Therefore, for a given predictable process 
 $v\in L^2(\Omega;L^2(0,T;L^2(\R^d)))$, the stochastic integral $t\mapsto \int_0^t \sigma(v)dW(s)$ is well-defined process taking values in a Hilbert space $L^2(\R^d)$. Moreover, the trajectories 
 of $W$ are $\mathbb{P}$- a.s. continuous in $\mathbb{H}_0 \supset \mathbb{H}$, where 
 \begin{align*}
  \mathbb{H}_0:= \Big\{ v= \sum_{k\ge 1} v_k e_k:\,\, \sum_{k\ge 1} \frac{v_k^2}{k^2} < + \infty \Big\} 
 \end{align*}
 endowed with the norm $\|v\|_{\mathbb{H}_0}^2= \sum_{k\ge 1} \frac{v_k^2}{k^2} $ with $v= \sum_{k\ge 1} v_k e_k$. Furthermore, the embedding $\mathbb{H}\hookrightarrow \mathbb{H}_0$ is Hilbert-Schmidt (see \cite{daprato}).
\end{rem}


Like its deterministic counterpart, existence of entropy solution largely related to the study of associated viscous problem. We first propose a result of existence of weak solutions to the regularized problem \eqref{eq:viscous-Brown} based on an implicit time discretization, adapted from the work of Bauzet et. al. \cite{BaVaWit_2014}.

\begin{thm}[Existence of Viscous Solution]
\label{prop:vanishing viscosity-solution}
Let the assumptions \ref{A1},\,\ref{A3},\,\ref{A4}, and \ref{A5} hold. Then, for any $\eps>0$, there exists a unique weak solution $u_\eps \in N_w^2(0,T,H^1(\R^d))$  with $\partial_t \big(u_\eps - \int_0^t \sigma(u_\eps(s,\cdot))\, dW(s) - \int_0^t \int_{E} \eta(u_\eps(s,\cdot);z)\widetilde{N}(dz,ds)\big)
\in L^2(\Omega\times(0,T),H^{-1}(\R^d))$, to the problem \eqref{eq:viscous-Brown}. Moreover, the solution 
$u_\eps \in L^\infty(0,T;L^2(\Omega \times\R^d))$ with $\Delta u_{\eps} \in L^2(\Omega \times \Pi_T)$, and there exists a constant $C>0$, independent of $\eps$, such that
\begin{align}
\sup_{0\le t\le T} \E\Big[\big\|u_\eps(t)\big\|_{L^2(\R^d)}^2\Big]  + \eps \int_0^T \E\Big[\big\|\grad u_\eps(s)\big\|_{L^2(\R^d)}^2\Big]\,ds + \int_0^T \E\Big[\big\| u_\eps(s)\|_{H^{\lambda}(\R^d)}^2\Big]\,ds \le C.\label{bounds:a-priori-viscous-solution}
\end{align} 
\end{thm}
 
With the above results at hand, we are now in a position to state the main results of this paper. 
\begin{thm}[Existence and Uniqueness]
\label{uniqueness_new}
Let the assumptions \ref{A1}, \ref{A3}, \ref{A31}, \ref{A4},and \ref{A5} are true. Then there exists a stochastic entropy solution for the Cauchy problem \eqref{eq:stoc_con_brown} in the sense of Definition~\ref{defi:stochentropsol}. Moreover, let $u$ and $v$ be two stochastic entropy solutions of \eqref{eq:stoc_con_brown} with same initial condition $u_0 = v_0$. Then almost surely $u(t) =v(t)$, for almost every $t\ge 0$. Furthermore, $u \in L^2(\Omega\times(0,T),H^\lambda(\R^d))$, and assuming that $(u_0 -v_0) \in L^1(\R^d)$, we have for almost every $t$ in $(0,T)$
\begin{align*}
\E \Big[\int_{\D}  |u(t,x) -v(t,x)| \,dx\Big] \leq& \int_\D |u_0-v_0|\,dx.
\end{align*}
\end{thm}

\begin{thm}[Continuous Dependence Estimate]
\label{continuous-dependence}
 Let the assumptions \ref{A1}-\ref{A5} hold for two sets of given data $(u_0, f, \sigma, \eta, \lambda)$ and 
 $(v_0,g, \widetilde{\sigma}, \widetilde{\eta}, \kappa)$. 
 Let $u(t,x)$ be any BV entropy solution of \eqref{eq:stoc_con_brown} with initial data $u_0(x)$ and $v(s,y)$ be another BV entropy solution with initial
 data $v_0(x)$ and satisfies 
\begin{equation}
\label{eq:stoc_con_brown-1}
\begin{aligned}
dv(s,y)- \mathrm{div} g(v(s,y)) \,ds + \mathcal{L}_{\kappa}[v(s, \cdot)](y)\,ds =
\widetilde{\sigma}(v(s,y))\,dW(s) + \int_{E} \widetilde{\eta} (v(s,y); z)\,\widetilde{N}(dz,ds)
\end{aligned}
\end{equation}
Moreover, define 
\begin{align*}
 \mathcal{E}_k(\sigma, \widetilde{\sigma}):=& \sup_{\xi \neq0} \frac{|g_k(\xi)-\widetilde{g}_k(\xi)|}{|\xi|}; \quad  \mathcal{E}(\sigma, \widetilde{\sigma})^2:= \underset{k\ge 1} \sum 
 \mathcal{E}_k(\sigma, \widetilde{\sigma})^2, \\
 \mathcal{D}(\eta,\widetilde{\eta}):=& \displaystyle{ \sup_{u \neq0} \int_{E} \frac{\big| \eta(u;z)-\widetilde{\eta}(u;z)\big|^2
}{|u|^2}\, m(dz)},
\end{align*}
and, in addition, assume that $f^{\prime\prime}\in L^\infty$.
Then, there exists a constant $C_T$, only depending on  $T$, $ |u_0|_{BV(\R^d)}$, $|v_0|_{BV(\R^d)}$, $\|f^{\prime\prime}\|_{\infty}$, 
$\|f^\prime\|_{\infty}$, and $\|\varphi\|_{1}$ such that for a.e. $0<t<T<+\infty$, 
  \begin{align*}
     & \E \Big[\int_{\R^d}\big| u(t,x)-v(t,x)\big|\varphi(x)\,dx \Big] \notag \\
     & \hspace{0.5cm} \le C_Te^{Ct} \Bigg\{ \E\Big[\int_{\R^d}\big| u_0(x) -v_0(x)\big| \,dx\Big] 
     + \max \bigg\{ \mathcal{E}(\sigma, \widetilde{\sigma}),\sqrt{\mathcal{D}(\eta,\widetilde{\eta})}\bigg\} \sqrt{t} + t\,||f^\prime-g^\prime||_{L^{\infty}(\D)}  \\
&\hspace{1.5cm} + \sqrt{\int_{0<|z|\le r} |z|^2 \,d |\mu_{\lambda} - \mu_{\kappa}|(z)} 
+ \int_{|z|> r} d |\mu_{\lambda} - \mu_{\kappa}|(z)\Bigg\},
\end{align*}
where $\varphi \in L^1(\R^d)$ such that $ 0\le \varphi(x)\le 1$, for all $x\in \R^d$, and $d\mu_{\lambda}(z)= \frac{dz}{|z|^{d + 2 \lambda}}$.
\end{thm}

As a by product of the above theorem, we have the following corollary:
\begin{cor}[Rate of Convergence]
\label{rate-of-convergence}
Let the assumptions \ref{A1}-\ref{A5} hold and $f^{\prime\prime}\in L^\infty$. Let $u(t,x)$ be any BV entropy solution of \eqref{eq:stoc_con_brown}
with $\E\Big[|u(t,\cdot)|_{BV(\R^d)}\Big] \le \E\Big[|u_0(\cdot)|_{BV(\R^d)}\Big]$ and 
  $u_\eps(s,y)$ be a weak solution to the problem \eqref{eq:viscous-Brown}. 
  Then there exists a constant $C$ depending only on 
  $|u_0|_{BV(\R^d)}, \|f^{\prime\prime}\|_{\infty}$, and $\|f^\prime\|_{\infty}$ such that for a.e. $t>0$, 
\begin{align*}
 \E\Big[\|u_\eps(t,\cdot)-u(t,\cdot)\|_{L^1(\R^d)}\Big] \approx \mathcal{O}(\eps^\frac{1}{2}),
\end{align*}
provided the initial error $\|u_0^\eps-u_0\|_{L^1(\R^d)} \approx \mathcal{O}(\eps^\frac{1}{2})$.
\end{cor}
\noindent Finally, inspired by the work of \cite{Alibaud} on the deterministic counterpart of 
\begin{align}
\label{eq:stoc_original}
\begin{cases} 
dw(t,x)- \mbox{div} f(w(t,x)) \,dt =
\sigma(w(t,x))\,dW(t) + \int_{E} \eta(w(t,x); z)\,\widetilde{N}(dz,dt),&(x,t) \in \Pi_T, \vspace{0.1cm}\\
w(0,x) = w_0(x), &  x\in \R^d,
\end{cases}
\end{align}
we consider the following non-local regularization of \eqref{eq:stoc_original}
\begin{align}
\label{eq:stoc_regularization}
dw_{\eps}(t,x)- \mbox{div} f(w_{\eps}(t,x)) \,dt +\eps \mathcal{L}_{\lambda}[w_{\eps}(t, \cdot)](x)\,dt=
\sigma(w_{\eps}(t,x))\,dW(t) + \int_{E} \eta(w_{\eps}(t,x); z)\,\widetilde{N}(dz,dt)
\end{align}
with initial data $w_{\eps}(0,x) = w^{\eps}_0(x)$. Note that existence of weak solutions to \eqref{eq:stoc_regularization} are guaranteed by virtue of the previous Theorem~\ref{uniqueness_new}. However, we are interested to prove the following result:
\begin{thm}[Vanishing Non-Local Regularization]
\label{vanishing-non-local}
Let the assumptions \ref{A1}-\ref{A5} hold. Let $w(t,x)$ be any BV entropy solution of \eqref{eq:stoc_original} with $\E\Big[|w(t,\cdot)|_{BV(\R^d)}\Big] \le \E\Big[|w_0(\cdot)|_{BV(\R^d)}\Big]$ and 
$w_\eps(s,y)$ be a weak solution to the problem \eqref{eq:stoc_regularization}. Then there exists a constant $C$ such that for a.e. $t>0$
\begin{align*}
\E\Big[\|w_\eps(t,\cdot)-w(t,\cdot)\|_{L^1(\R^d)}\Big] =Ce^{Ct}
\begin{cases}
\mathcal{O}(\eps), & \quad \text{when}\,\, 0 < \lambda <1/2, \\
\mathcal{O}(\eps \,|\log(\eps)|), & \quad \text{when}\,\, \lambda =1/2, \\
\mathcal{O}(\eps^{1/2\lambda}), & \quad \text{when}\,\, 1/2 < \lambda <1,
\end{cases}
\end{align*}
provided the initial error is also same.
\end{thm}


\begin{rem}
We remark that all the results presented in Theorem~\ref{uniqueness_new}, Theorem~\ref{continuous-dependence}, Corollary~\ref{rate-of-convergence}, and Theorem~\ref{vanishing-non-local} are indeed true for \textit{all} $t>0$.

In fact we know that the solution 
\begin{align*}
u \in L^2((0,T); H^{\lambda}(\Omega \times \R^d)), \,\, \text{and}\\
\partial_t\Big(u - \int^t_0 \sigma(u) \,dW -\int^t_0 \int_E \eta( u;z) \widetilde{N}(\,dz,\,dt)\Big) \in L^2((0,T); L^2(\Omega; H^{-1}(\Rd))).
\end{align*}
Therefore we have $u \in C([0,T]; L^2(\Omega; H^{-1}(\Rd)))$. Since we also know (cf. \cite[Remark 2.4]{BaVaWit_2012}) that 
$u \in L^{\infty}((0,T); L^2(\Omega \times \R^d))$, we conclude that $u$ is weakly continuous in time with values in $L^2(\Omega \times \R^d)$. Moreover, we know (cf. Theorem~\ref{thm:existence-bv}) that for a.e. $t>0$, $\norm{u(t)}_{L^1(\Omega\times\R^d)}$ is finite. In particular, this implies that for any $t \in [0,T]$ and $h\neq 0$ such that $t+h \in [0,T]$, we have
\begin{align*}
\int_{\Omega \times \R^d} u(t) w \,dx\,d\mathbb{P}&= \lim_{h\rightarrow0} \int_t^{t+h} \int_{\Omega \times \R^d} u(s) w \,dx\,d\mathbb{P}\,\,ds, \quad \forall w \in L^2(\R^d), \\
\bigg| \int_t^{t+h} \int_{\Omega \times \R^d} u(s) w \,dx\,d\mathbb{P}\,\,ds \bigg| & \le \frac{1}{|h|} \bigg|\int_t^{t+h} \int_{\Omega \times \R^d} u(s) w \,dx\,d\mathbb{P}\,\,ds \bigg| \le C, \quad \forall w \in L^2(\R^d) \,\, \text{with}\,\, |w|\le 1.
\end{align*}
Thus for any $t>0$,
\begin{align*}
\bigg| \int_{\Omega \times \R^d} u(t) w \,dx\,d\mathbb{P} \bigg| \le C, \quad \forall w \in L^2(\R^d) \,\, \text{with}\,\, |w|\le 1,
\end{align*}
and for any $M>0$
\begin{align*}
\int_{\Omega \times B(0,M)} |u(t)| \,dx\,d\mathbb{P} = \int_{\Omega \times \R^d} u(t) \,\sgn(u(t)) {\bf 1}_{B(0,M)}\,dx\,d\mathbb{P} \le C
\end{align*}
Therefore, Fatou's lemma yields $u(t) \in L^1(\Omega \times \R^d)$ and $\norm{u(t)}_{L^1(\Omega\times\R^d)} \le C$, for all $t>0$.
\end{rem}


Before concluding this section, we introduce a  special class of entropy functions, 
called convex approximation of absolute value function. To do so,  let $\beta:\R \rightarrow \R$ be a $C^\infty$ function satisfying 
\begin{align*}
\beta(0) = 0,\quad \beta(-r)= \beta(r),\quad 
\beta^\prime(-r) = -\beta^\prime(r),\quad \beta^{\prime\prime} \ge 0,
\end{align*} 
and 
\begin{align*}
\beta^\prime(r)=\begin{cases} -1\quad \text{when} ~ r\le -1,\\
\in [-1,1] \quad\text{when}~ |r|<1,\\
+1 \quad \text{when} ~ r\ge 1.
\end{cases}
\end{align*} 
For any $\xi> 0$, define $\beta_\xi:\R \rightarrow \R$ by 
$\beta_\xi(r) = \xi \beta(\frac{r}{\xi})$. 
Then
\begin{align}\label{eq:approx to abosx}
|r|-M_1\xi \le \beta_\xi(r) \le |r|\quad 
\text{and} \quad |\beta_\xi^{\prime\prime}(r)| 
\le \frac{M_2}{\xi} {\bf 1}_{|r|\le \xi},
\end{align} 
where $M_1 := \sup_{|r|\le 1}\big | |r|-\beta(r)\big |$ and 
$M_2 := \sup_{|r|\le 1}|\beta^{\prime\prime} (r)|$. 

Moreover, for $\beta=\beta_\xi$, we define 
\begin{align*}
\begin{cases}
f_k^\beta(a,b)= \int_b^a \beta_\xi ^\prime(r-b) f_k^\prime(r)\,dr, \\
f^\beta(a,b)=\big(f_1^\beta(a,b),f_2^\beta(a,b),\cdots,f_d^\beta(a,b)\big), \\
F(a,b)= \sgn(a-b)(f(a)-f(b)).
\end{cases}
\end{align*}


\section{Proof of Theorem~\ref{uniqueness_new}: Uniqueness of Entropy Solutions}
 \label{uniqueness}
 
This section is the main part of the manuscript. To prove the uniqueness for entropy solutions, we will use the following strategy. First note that, thanks to the \textit{a priori} estimates, the compactness given by the theory of Young measures yields the existence of a limit point, a measure-valued (mild) solution (``mild'' because we haven't proved that this limit satisfies the entropy formulation). Then, by using a uniqueness method for entropy solutions, a classical argument will ensure some strong convergence for the sequence of approximation and thus: a result of existence of an entropy solution, the uniqueness of the entropy solution and the fact that any measure-valued (mild) solution is indeed the entropy solution.

\subsection{Kato's inequality}
\label{Kato}

To prove the result of uniqueness of (measure-valued (mild)) entropy solution, we follow the usual strategy of adapting a variant of classical Kruzkov's ``doubling of variable'' approach. Note that, the main difficulty lies in ``doubling'' the time variable which gives rise to stochastic integrands that are anticipative and hence can not be interpreted in the usual It\^{o}-L\'{e}vy sense. 
To get around this problem, we will have to compare a weakly converging sequence of viscous approximation to an entropy one, assuming that it exists.  
\\[0.1cm]
Usually, one proposes (as in \cite{BaVaWit_2012,BaVaWit_JFA,BisMajKarl_2014}) to compare a viscous solution to a measure-valued entropy one. Since we have not proposed such a notion of solution, we will have to, like in \cite{BaVaWit_2014,BisMajVal}, compare two weakly converging sequences of viscous approximations and propose, first, a Kato inequality for limit points in the sense of Young measures associated with these above mentioned sequences. The method of uniqueness will yield the existence of an entropy solution and some strong convergence properties, justifying thereby our initial aim: the existence of an entropy solution.

Let $\rho$ and $\varrho$ be the standard nonnegative 
mollifiers on $\R$ and $\R^d$ respectively such that 
$\supp(\rho) \subset [-1,0]$ and $\supp(\varrho) = \overline{B_1}(0)$.  
We define  $\rho_{\delta_0}(r) = \frac{1}{\delta_0}\rho(\frac{r}{\delta_0})$ 
and $\varrho_{\delta}(x) = \frac{1}{\delta^d}\varrho(\frac{x}{\delta})$, 
where $\delta$ and $\delta_0$ are two positive constants.  Given a  nonnegative test 
function $\psi\in C_c^{1,2}([0,\infty)\times \rd)$ and two 
positive constants $\delta$ and $ \delta_0 $, we define 
\begin{align}
\label{test_function}
\varphi_{\delta,\delta_0}(t,x, s,y) = \rho_{\delta_0}(t-s) 
\,\varrho_{\delta}(x-y) \,\psi(t,x).
\end{align}
Clearly $ \rho_{\delta_0}(t-s) \neq 0$ only 
if $s-\delta_0 \le t\le s$ and hence $\varphi_{\delta,\delta_0}(t,x; s,y)= 0$, 
outside  $s-\delta_0 \le t\le s$. 
 
Let us assume two possible situations: 
\\
$u(t,x,\alpha)$ and $v(s,y,\beta)$, $\alpha, \beta \in (0,1)$, are Young measure-valued limit processes solutions associated with the sequences  $u_\theta(t,x)$ and $u_\eps(s,y)$ of weak solutions of \eqref{eq:viscous-Brown}, with regular initial data $u^{\theta}_0(x)$ and $v^{\eps}_0(y)$ respectively, and estimates given in Theorem \ref{prop:vanishing viscosity-solution};
\\
or,  $u(t,x)$ is an entropy solution with initial data $u_0(x)$ in the sense of Definition \ref{defi:stochentropsol} (\textit{i.e.} $\theta=0$) and $v(s,y,\beta)$ is as above. 
\\
For convenience, whatever the situation is, $u_\theta(t,x)$ and $u(t,x,\alpha)$ are considered is the sequel, keeping in mind that $u_\theta(t,x) = u(t,x) = u(t,x,\alpha)$ for any $\alpha$ if $u$ is assumed to be an entropy solution.

Moreover, let $\varsigma$ be the standard symmetric 
nonnegative mollifier on $\R$ with support in $[-1,1]$ 
and $\varsigma_l(r)= \frac{1}{l} \varsigma(\frac{r}{l})$, 
for $l > 0$. We use the generic $\beta$ for the 
functions $\beta_{\xi}$ introduced in Section \ref{sec:tech}.
Given $k\in \R$, the function 
$\beta(\cdot-k)$ is a smooth convex function 
and $(\beta(\cdot-k), F^\beta(\cdot, k))$ is a 
convex entropy pair. 

We now write the It\^{o}-L\'{e}vy formula for $u_\theta(t,x)$, based on the entropy pair $(\beta(\cdot-k), F^\beta(\cdot, k))$, and then multiply by $\varsigma_l(u_\eps(s,y)-k)$, take the expectation and integrate with 
respect to $ s, y, k$. A simple application of Fubini's theorem leads to
\begin{align}
& 0\le \E \bigg[\int_{\Pi_T}\int_{\R^d}\int_{\R} \beta(u^{\theta}_0(x)-k)\,
\varphi_{\delta,\delta_0}(0,x,s,y) \,\varsigma_l(u_\eps(s,y)-k)\,dk \,dx\,dy\,ds\bigg] \notag 
\\&  
+ \E \bigg[\int_{\Pi_T} \int_{\Pi_T} 
\int_{\R} \beta(u_\theta(t,x)-k)\partial_t \varphi_{\delta,\delta_0}(t,x,s,y)\,
\varsigma_l(u_\eps(s,y)-k)\,dk\,dx\,dt\,dy\,ds \bigg]\notag 
\\ & 
+ \E \bigg[\sum_{k\ge 1}\int_{\Pi_T} \int_{\R}\int_{\Pi_T} 
  g_k(u_\theta(t,x))\beta^\prime (u_\theta(t,x)-k)\, \varphi_{\delta,\delta_0}(t,x,s,y)\,d\beta_k(t) \,dx\, \varsigma_l(u_\eps(s,y)-k)\,dk\,dy\,ds \bigg] \notag \\
 &+ \frac{1}{2}\, \E \bigg[ \int_{\Pi_T} \int_{\Pi_T} 
\int_{\R} \mathbb{G}^2(u_\theta(t,x))\beta^{\prime\prime} (u_\theta(t,x) -k)\, \varphi_{\delta,\delta_0}(t,x,s,y)\, \varsigma_l(u_\eps(s,y)-k)\,dk\,dx\,dt\,dy\,ds \bigg] \notag \\
& + \E \bigg[ \int_{\Pi_T}\int_{\R} \int_{t=0}^T\int_{E} \int_{\R^d}\Big(\beta \big(u_\theta(t,x) 
+\eta(u_\theta(t,x);z)-k\big)-\beta(u_\theta(t,x)-k)\Big) \notag \\
& \hspace{6cm} 
\times \varphi_{\delta,\delta_0}(t,x,s,y) \, \tilde{N}(dz,dt)
\varsigma_l(u_\eps(s,y)-k)\,dk \,dx\,dy\,ds \bigg] \notag\\
& + \E \bigg[\int_{\Pi_T} \int_{t=0}^T\int_{E}\int_{\R^d} 
\int_{\R} \Big(\beta \big(u_\theta(t,x)
+\eta(u_\theta(t,x);z)-k\big)-\beta(u_\theta(t,x)-k) \notag \\
 & \hspace{4.5cm}-\eta(u_\theta(t,x);z) \beta^{\prime}(u_\theta(t,x)-k)\Big)
 \varphi_{\delta,\delta_0}(t,x;s,y) \notag \\
&\hspace{6cm}\times \varsigma_l(u_\eps(s,y)-k)\,dk\,dx\,m(dz)\,dt\,dy\,ds\bigg]\notag \\
& + \E \bigg[\int_{\Pi_T}\int_{\Pi_T} \int_{\R} 
 F^\beta(u_\theta(t,x),k) \cdot \grad_x  \varphi_{\delta,\delta_0}(t,x;s,y)\, \varsigma_l(u_\eps(s,y)-k)\,dk\,dx\,dt\,dy\,ds\bigg] \notag \\
 & -\theta \, \E \bigg[\int_{\Pi_T}\int_{\Pi_T} \int_{\R} \beta^{\prime}(u_{\theta}(t,x)-k) \nabla_x u_{\theta} (t,x) \cdot \nabla_x \varphi_{\delta,\delta_0}(t,x;s,y)\, \varsigma_l(u_\eps(s,y)-k)\,dk\,dx\,dt\,dy\,ds\bigg] \notag \\
& - \E \bigg[\int_{\Pi_T} \int_{\Pi_T}  
\int_{\R} \mathcal{L}_{\lambda}^r[u_\theta(t,\cdot)](x)\, \varphi_{\delta,\delta_0}(t,x,s,y)\, \beta'(u_\theta(t,x) -k) \,dx\,dt 
\, \varsigma_l(u_\eps(s,y)-k)\,dk\,dy\,ds \bigg] \notag \\
& -  \E \bigg[\int_{\Pi_T} \int_{\Pi_T}  
\int_{\R} \beta(u_\theta(t,x) -k ) \,\mathcal{L}_{\lambda,r} [\varphi_{\delta,\delta_0}(t,\cdot,s,y)](x)
\, \varsigma_l(u_\eps(s,y)-k)\,dk\,dx \,dt\,dy\,ds \bigg]\notag  \\[2mm]
& =:  I_1 + I_2 + I_3 +I_4 + I_5 + I_6 + I_7 + I_8 + I_9+I_{10}. \label{stochas_entropy_1-levy-d}
\end{align}
We now apply the It\^{o}-L\'{e}vy formula 
to \eqref{eq:viscous-Brown}, to get
\begin{align}
& 0\le \E \bigg[\int_{\Pi_T}\int_{\R^d} \int_{\R} 
\beta(v^{\eps}_0(y)-k)\varphi_{\delta,\delta_0}(t,x,0,y) 
\varsigma_l(u_\theta(t,x)-k)\,dk\,dx\,dy\,dt\bigg] \notag \\
 &  + \E \bigg[\int_{\Pi_T} \int_{\Pi_T} \int_{\R} 
 \beta(u_\eps(s,y)-k)\partial_s \varphi_{\delta,\delta_0}(t,x,s,y)\,
 \varsigma_l(u_\theta(t,x)-k)\,dk\,dy\,ds\,dx\,dt\bigg]\notag \\ 
 & + \E \bigg[\sum_{k\ge 1}\int_{\Pi_T} \int_{\Pi_T} 
\int_{\R}  g_k(u_\eps(s,y))\,\beta^\prime (u_\eps(s,y)-k)\, \varphi_{\delta,\delta_0}(t,x,s,y)\,\varsigma_l(u_\theta(t,x)-k)\,dk \,d\beta_k(s) \,dy\,dx\,dt \bigg] \notag \\
 &+ \frac{1}{2}\, \E \bigg[ \int_{\Pi_T} \int_{\Pi_T} 
\int_{\R} \mathbb{G}^2(u_\eps(s,y)) \,\beta^{\prime\prime} (u_\eps(s,y) -k)\, \varphi_{\delta,\delta_0}(t,x,s,y) \,\varsigma_l(u_\theta(t,x)-k)\,dk\,dy\,ds\,dx\,dt \bigg] \notag \\
 & + \E \bigg[\int_{\Pi_T} \int_{s=0}^T\int_{E} \int_{\R} 
 \int_{\R^d}\Big(\beta \big(u_\eps(s,y) +\eta_\eps(u_\eps(s,y);z)-k\big)
 -\beta(u_\eps(s,y)-k)\Big) \notag \\
 & \hspace{6cm} 
 \times \varphi_{\delta,\delta_0}(t,x,s,y)\, \varsigma_l(u_\theta(t,x)-k)
 \,dy\,dk\,\tilde{N}(dz,ds)\,dx\,dt \bigg]\notag\\
 & + \E \bigg[\int_{\Pi_T} \int_{s=0}^T\int_{E}\int_{\R^d} \int_{\R}\Big(\beta \big(u_\eps(s,y) +\eta_{\eps}((u_\eps(s,y);z)-k\big)
 -\beta(u_\eps(s,y)-k) \notag \\
 & \hspace{6cm}-\eta_{\eps}(u_\eps(s,y);z)  \beta^{\prime}(u_\eps(s,y)-k)\Big)
 \varphi_{\delta,\delta_0}(t,x;s,y) \notag \\
 &\hspace{8cm}
 \times \varsigma_l(u_\theta(t,x)-k)\,dk\,dy\,m(dz)\,ds\,dx\,dt \bigg]\notag \\
 & + \E \bigg[\int_{\Pi_T}\int_{\Pi_T} 
 \int_{\R}  F^\beta(u_\eps(s,y),k)  \cdot \grad_y\varrho_\delta(x-y) \psi(t,x)
 \rho_{\delta_0}(t-s ) \varsigma_l(u_\theta(t,x)-k)\,dk\,d\alpha\,dy\,ds\,dx\,dt \bigg] \notag \\
  &- \eps \, \E\bigg[ \int_{\Pi_T} \int_{\Pi_T} 
  \int_{\R} \beta^\prime(u_\eps(s,y)-k)\grad_y u_\eps(s,y)\cdot \grad_y  \varphi_{\delta,\delta_0}(t,x;s,y)   \varsigma_l(u_\theta(t,x)-k)\,dk\, \,dy\,ds\,dx\,dt\bigg]  \notag \\ 
 & - \E \bigg[\int_{\Pi_T} \int_{\Pi_T} 
\int_{\R} \mathcal{L}^r_{\lambda}[u_\eps(s, \cdot)](y) \, \varphi_{\delta,\delta_0}(t,x,s,y)\, \beta^\prime (u_\eps(s,y)-k) \,dy\,ds \, \varsigma_l(u_\theta(t,x)-k)\,dk\,dx\,dt \bigg]  \notag \\
& -\E \bigg[\int_{\Pi_T} \int_{\Pi_T} 
\int_{\R} \beta( u_\eps(s,y)-k)\mathcal{L}_{\lambda, r}[\varphi_{\delta,\delta_0}(t,x,s,.)](y) \, \varsigma_l(u_\theta(t,x)-k)\,dk\,dy\,ds\,dx\,dt \bigg]  \notag \\[2mm]
& =:  J_1 + J_2 + J_3 +J_4 + J_5 + J_6 + J_7 + J_8 +J_9+J_{10}.
\label{stochas_entropy_3-levy-d}
\end{align}

We now add \eqref{stochas_entropy_1-levy-d} and \eqref{stochas_entropy_3-levy-d}, 
and compute limits with respect to the various parameters involved. To deal with most of the terms, other than terms involving the fractional operator, we follow as in \cite{BaVaWit_2014,BisMajKarl_2014}. We do this by claiming a series of lemmas: \ref{lem:initial+time-terms}, \ref{lem:stochastic-terms} and \ref{lem:flux-terms},
whose proofs follow the ones in \cite{BaVaWit_2014,BisMajKarl_2014} modulo cosmetic changes. 

 
\begin{lem}\label{lem:initial+time-terms}
It holds that $I_1=0$ and since $v_0^{\eps} \underset{\eps\to0} \longrightarrow v_0$ in $L^2(\D)$, and $u_0^{\theta} \underset{\theta\to0} \longrightarrow u_0$ in $L^2(\D)$, we have
\begin{align*}
\lim_{\delta \goto 0}\,\lim_{\xi\goto 0} \,\lim_{\theta \goto 0}\,\lim_{\eps \goto 0}\,\lim_{l\goto 0} \,\lim_{\delta_0\goto 0} \big(I_1 + J_1\big)& =  \E \Big[\int_{\R^d} |u_0(x)-v_0(x)| \,\psi(0,x) \,dx \Big]\notag.\\
\lim_{\delta \goto 0}\,\lim_{\xi\goto 0} \,\lim_{\theta \goto 0}\,\lim_{\eps \goto 0} \,\lim_{l\goto 0} \,\lim_{\delta_0\goto 0} \big(I_2 + J_2\big) 
& =\E \Big[\int_{t=0}^T \int_{\R^d}\int_{\alpha=0}^1 \int_{\beta=0}^1 
  |u(t,x,\alpha)-v(t,x,\beta)| \partial_t\psi(t,x)
  d\alpha \,d\beta\,dx\,dt\Big].
\end{align*}
\end{lem}
 
\begin{lem}\label{lem:stochastic-terms}
We have $J_3=J_5=0$, and The following hold:
\begin{align*}
&\lim_{\delta \goto 0}\,\lim_{\xi\goto 0}\,\lim_{\theta \goto 0} \, \lim_{\eps \goto 0}\,\lim_{l\goto 0}\,\lim_{\delta_0\goto 0} \big(I_4 + J_4 +I_3\big) =0, \\[2mm] 
&\lim_{\delta \goto 0}\,\lim_{\xi\goto 0} \,\lim_{\theta \goto 0} \,\lim_{\eps \goto 0}\,\lim_{l\goto 0}\,\lim_{\delta_0\goto 0}  \big(I_6 + J_6+I_5\big) =0.
\end{align*}
\end{lem}
 
\begin{lem}\label{lem:flux-terms}
The following hold:
\begin{equation*}
\begin{aligned}
& \lim_{\delta \goto 0}\,\lim_{\xi\goto 0}\,\lim_{\theta \goto 0} \, \lim_{\eps \goto 0}\,\lim_{l\goto 0}\,\lim_{\delta_0\goto 0} \big(I_7 + J_7 \big)  
= -\E \Big[\int_{\Pi_T} \int_{\alpha=0}^1 \int_{\beta=0}^1 
F(u(t,x,\alpha),v(t,x,\beta)) \cdot \grad_x \psi(t,x)
 \,d\alpha\,d\beta\,dx\,dt\Big], \\[2mm]  
& \lim_{\delta \goto 0}\,\lim_{\xi\goto 0} \,\limsup_{\theta \goto 0} \, \limsup_{\eps \goto 0}\,\lim_{l\goto 0}\,\lim_{\delta_0\goto 0} (I_8 +J_8) =0. 
\end{aligned}
\end{equation*}
\end{lem}

We now add terms coming from the fractional operators, 
and compute limits with respect to the various parameters involved.   
\begin{lem}
\label{fractional_lemma_1}
It holds that 
\begin{align*}
  I_9 + J_9   & \underset{\delta_0 \goto 0} \longrightarrow - \E \Big[\int_{\R^d} \int_{\Pi_T} 
\int_{\R}  \mathcal{L}_{\lambda}^r [u_\theta(t,\cdot)](x)\, \psi(t,x)\varrho_{\delta} (x-y)\, \beta'(u_\theta(t, x) -k)
\, \varsigma_l(u_\eps(t,y)-k)\,dx\,dt \,dk \,dy \Big] \\
& - \E \Big[\int_{\Rd} \int_{\Pi_T} 
\int_{\R} \mathcal{L}^r_{\lambda}[u_\eps(t, \cdot)](y) \, \psi(t,x)\varrho_{\delta} (x-y)\, \beta^\prime (u_\eps(t,y)-k)  \varsigma_l(u_\theta(t,x)-k) \,dx\,dt\,dk \,dy\Big] \\
&\underset{l \goto 0} \longrightarrow  -\E \Big[\int_{\R^d} \int_{\Pi_T}
 \mathcal{L}_{\lambda}^r[u_\theta(t, \cdot](x) \, \psi(t,x)\varrho_{\delta} (x-y)\, \beta'(u_\theta(t, x) - u_\eps(t,y)) \,dx\,dt 
\,dy \Big]\\
&- \E \Big[\int_{\R^d} \int_{\Pi_T}  
\mathcal{L}^r_{\lambda} [u_{\eps}(t,\cdot)](y)\, \psi(t,x)\,\varrho_{\delta} (x-y)\, \beta'(u_\eps(t,y) - u_\theta(t, x)\,dx\,dt 
\,dy \Big].
\end{align*}
Moreover, we have 
\begin{align*}
\limsup_{\theta \to 0}\,&\limsup_{\ep \to 0} \, \lim_{l\to 0}\,\lim_{\delta_0 \to 0} \big(I_9+J_9\big)\\ &\le  - \E \Big[\int_{\R^d} \int_{\Pi_T} \int_{\alpha=0}^1\int^1_{\beta=0}\fr^r[\psi(t,\cdot)](x) \varrho_{\delta}(x-y)\beta(v(t,x,\alpha)-u(t,y,\beta)) \,d\alpha\, d\beta \,dy\,dx\,dt \Big]
\end{align*}
\end{lem}

\begin{proof}
We propose to prove this lemma in several steps.

\noindent {\bf Step 1}(Passing to the limit as $ \delta_0 \to 0$): Since $\psi$ is compactly supported in space and $u_\theta \in L^2(\Omega \times \Pi_T)$, by using properties of convolution in $L^1(0,T;L^1(\Omega \times \R^d \times \R^d \times \R))$, we conclude that
\begin{align*}
\E \Big[\int_{\Pi_T} \int_{\Pi_T}
&\int_{\R} \mathcal{L}^r_{\lambda}[u_\theta(t, \cdot)](x) \, \psi(t,x) \varrho_\delta(x-y)\, \beta^\prime (u_\theta(t,x)-k)\varsigma_l(u_\ep(s, y)-k) \varrho_{\delta_0}(t-s)\,dk  \,dx\,dt \,dy\,ds \Big] \\
&\underset{\delta \to 0}\longrightarrow  \E \Big[\int_{\Rd} \int_{\Pi_T}
\int_{\R} \mathcal{L}^r_{\lambda}[u_\theta(t, \cdot)](x) \, \psi(t,x)\varrho_{\delta} (x-y)\, \beta^\prime (u_\theta(t,x)-k)  \varsigma_l(u_\ep(t,y)-k) \,dk \,dy\,dx\,dt\Big].
\end{align*}
A similar argument also confirms that
\begin{align*}
&\E \Big[\int_{\Pi_T} \int_{\Pi_T}
\int_{\R} \mathcal{L}^r_{\lambda}[u_\ep(s, \cdot)](y) \, \psi(t,x) \varrho_\delta(x-y)\, \beta^\prime (u_\eps(s,y)-k)\varsigma_l(u_\theta(t, x)-k) \varrho_{\delta_0}(t-s)\,dk  \,dx\,dt \,dy\,ds \Big] \\
& \underset{\delta \to 0}\longrightarrow  \E \Big[\int_{\Rd} \int_{\Pi_T}
\int_{\R} \mathcal{L}^r_{\lambda}[u_\ep(t, \cdot)](y) \, \psi(t,x) \varrho_\delta(x-y)\, \beta^\prime (u_\eps(t,y)-k)\varsigma_l(u_\theta(t, x)-k) \,dk  \,dx\,dt \,dy \Big]. 
\end{align*}

\noindent {\bf Step 2} (Passing to the limit as $l \to 0$): 
Observe that
\begin{align*}
& \Bigg| \E \bigg[\int_{\Rd} \int_{\Pi_T}  \int_{\R}  ( \beta'(u_\theta(t,x)-k)-\beta'(u_\theta(t,x)-u_\ep(t,y))) \varsigma_l(u_\ep(t,y)-k)\\
&\hspace{5cm}\times \mathcal{L}_{\lambda}^r[u_\theta(t,\cdot)](x)\, \psi(t,x)\varrho_{\delta} (x-y) \,dk\,dx\,dt\,dy \bigg] \Bigg|\\
&\le C(\beta^{\prime\prime}) \E \bigg[\intrd \int_{\Pi_T}\int_{\R}| u_\ep(t,y)-k| \varsigma_l(u_\ep(t,y)-k)|\mathcal{L}_{\lambda}^r[u_\theta(t,\cdot)](x)|\, \psi(t,x)\varrho_{\delta} (x-y)\,dk\,dx\,dt\,dy\bigg] \\
& \le C(\beta^{\prime \prime})\, l\, \E \bigg[\int_{\Rd} \int_{\Pi_T}|\mathcal{L}_{\lambda}^r[u_\theta(t,\cdot)](x)| \psi(t,x) \varrho_\delta(x-y) \,dx\,dt\,dy\bigg]\underset{l\to 0} \longrightarrow 0,
\end{align*}
where we have used the fact that $| u_\ep(t,y)-k|\varsigma_l(u_\ep(t,y)-k)\leq \varsigma_l(u_\ep(t,y)-k)l$. A similar argument reveals that 
\begin{align*}
&\E \Big[\int_{\Rd} \int_{\Pi_T} 
\int_{\R} \mathcal{L}^r_{\lambda}[u_\eps(t, \cdot)](y) \, \psi(t,x)\varrho_{\delta} (x-y)\, \beta^\prime (u_\eps(t,y)-k) \varsigma_l(u_\theta(t,x)-k) \,dk \,dx\,dt\,dy\Big] \\
&\quad \quad \underset{l \to 0}\longrightarrow \E \Big[\int_{\R^d} \int_{\Pi_T} 
\mathcal{L}^r_{\lambda} [u_{\eps}(t,\cdot)](y)\, \psi(t,x)\,\varrho_{\delta} (x-y)\, \beta'(u_\eps(t,y) - u_\theta(t,x))  \,dx\,dt 
\,dy \Big].
\end{align*}

\noindent {\bf Step 3} (Passing to the limit as $\ep, \delta \to 0$): 
First observe that
\begin{align*}
 & -\E \bigg[\int_{\R^d} \int_{\Pi_T}
 \mathcal{L}_{\lambda}^r[u_\theta(t,\cdot)](x)\, \psi(t,x)\varrho_{\delta} (x-y)\, \beta'(u_\theta(t,x) - u_\eps(t,y))  \,dx\,dt 
\,dy \bigg]\\
&\qquad \qquad -\E \bigg[\int_{\R^d} \int_{\Pi_T} 
\mathcal{L}^r_{\lambda} [u_{\eps}(t,\cdot)](y)\, \psi(t,x)\,\varrho_{\delta} (x-y)\, \beta'(u_\eps(t,y) - u_\theta(t,x))  \,dx\,dt 
\,dy \bigg] \\
& =\E \Bigg[\intrd \int_{\Pi_T} \bigg[ \int_{|z| \ge r} \frac{u_\theta(t,x+z)-u_\theta(t,x)}{|z|^{d+2\lambda}}\,dz- \int_{|z| \ge r} \frac{u_\ep(t,y+z)-u_\ep(t,y)}{|z|^{d+2\lambda}} \,dz  \bigg]\\
& \hspace{5cm}\times \psi(t,x) \varrho_{\delta}(x-y) \beta'( u_\theta(t,x)-u_\ep(t,y))\,dx\,dt\,dy \Bigg]\\
&=\E \Bigg[\intrd \int_{\Pi_T} \bigg[ \int_{|z| \ge r} \frac{u_\theta(t,x+z)-u_\ep(t,y+z)}{|z|^{d+2\lambda}} \,dz -\int_{|z|\ge r} \frac{u_\theta(t,x)-u_\ep(t,y)}{|z|^{d+2\lambda}}\,dz\bigg]\\
& \hspace{5cm} \times \psi(t,x) \varrho_{\delta}(x-y)\beta'( u_\theta(t,x)-u_\ep(t,y))\,dx\,dt\,dy \Bigg]\\
&\le \E \Bigg[\intrd \int_{\Pi_T}\bigg[\int_{|z| \ge r} \frac{\beta(u_\theta(t,x+z)-u_\ep(t,y+z))}{|z|^{d+2\lambda}}\,dz- \int_{|z| \ge r}\frac{\beta(u_\theta(t,x)-u_\ep(t,y))}{|z|^{d+2\lambda}} \,dz\bigg]\\
& \hspace{8cm} \times \varrho_\delta(x-y) \psi(t,x) \,dx\,dt \,dy \Bigg]\\
&=- \E \bigg[\intrd \int_{\Pi_T}  \beta( u_\theta(t,x)-u_\ep(t,y)) \fr^r[ \psi(t, \cdot)](x) \varrho_\delta(x-y) \,dy\,dx\,dt \bigg],
\end{align*}
where to derive the penultimate inequality, we have used the fact that $\beta(b)-\beta(a) \ge \beta^\prime (a) (b-a)$ with
$a=u_\theta(t,x)-u_\ep(t,y)$ and $b= u_\theta(t,x+z)-u_\ep(t,y+z)$. For the last equality, we have performed a change of coordinates for the first integral $x \to x+z, y\to y+z, z\to -z$. 

At this point, we first fix $\theta$ and pass to the limit in $\ep$ in the sense of Young measures. For that purpose, let 
us define 
$$G(t,y,\omega; \mu):= \int_{\D} \beta( u_\theta(t,x)-\mu) \fr^r[ \psi(t, \cdot)](x) \varrho_\delta(x-y)\,dx.$$
Since $\fr^r[ \psi(t, \cdot)] \in L^p(\R^d)$ for any $t$ and any $p\in[1,+\infty]$, $G$ is a Carath\'eodory function on $\Pi_T\times\Omega\times\R$.
\\
Note that $G(\cdot,u_\epsilon)$ is a uniformly bounded sequence in $L^2(\Pi_T\times \Omega)$. Indeed, 
\begin{align*}
\E \Big[\int_{\Pi_T} & |G(\cdot;u_\ep)|^2 \,dy\,dt\Big]= \E \bigg[ \intrd \int^T_{t=0} \Big[\intrd \beta(u_\theta(t,x)-u_\ep(t,y)) \fr^r[ \psi(t, \cdot)](x) \varrho_\delta(x-y)\,dx\Big]^2 \,dy \,dt \bigg]\\
&\le c(r,\|\psi\|_\infty) \E \bigg[\intrd \int^T_{t=0} \intrd | u_\theta(t,x)-u_\ep(t,y)|^2 \varrho^2_\delta(x-y)\,dx\,dt \,dy \bigg] \leq C.
\end{align*}
To ensure that the family $G(\cdot, u_\ep(\cdot))$ is uniformly integrable, we need to check the equi-smallness property at infinity. For that purpose, set $\tilde \eps>0$ and note that for a given $R>r$ 
\begin{align*}
G(t,y,\omega;u_\ep) =& \int_{|z|>r} \frac{1}{|z|^{d+2\lambda}} \intrd \beta(u_\theta(t,x)-u_\ep(t,y)) [\psi(t, x)-\psi(t,x+z)] \varrho_\delta(x-y)\,dx
\\
=& \int_{|z|>R} \frac{1}{|z|^{d+2\lambda}} \intrd \beta(u_\theta(t,x)-u_\ep(t,y)) [\psi(t, x)-\psi(t,x+z)] \varrho_\delta(x-y)\,dx 
\\&+ \int_{R>|z|>r} \frac{1}{|z|^{d+2\lambda}} \intrd \beta(u_\theta(t,x)-u_\ep(t,y)) [\psi(t, x)-\psi(t,x+z)] \varrho_\delta(x-y)\,dx.
\end{align*}
Firstly, set $R$ such that 
\begin{align*}
&\E \int_{\Pi_T} \int_{|z|>R} \frac{1}{|z|^{d+2\lambda}} \intrd \beta(u_\theta(t,x)-u_\ep(t,y)) |\psi(t, x)-\psi(t,x+z)| \varrho_\delta(x-y)\,dx\,dy\,dt \\
\leq & 
C\int_{|z|>R} \frac{1}{|z|^{d+2\lambda}} \E \int_{\Pi_T}  \intrd [|u_\theta(t,x)|+|u_\ep(t,y)|] [|\psi(t, x)|+|\psi(t,x+z)|] \varrho_\delta(x-y)\,dx\,dy\,dt \\
\leq & C(\|\psi\|_{L^2(\Pi_T)},\|u_\theta\|_{L^2(\Omega\times\Pi_T)},\|u_\epsilon\|_{L^2(\Omega\times\Pi_T)})\frac{1}{R^{2\lambda}} < \tilde \eps.
\end{align*}
Then, considering $M$ such that $M > K+\delta + R$, where supp$\psi(t,.) \subset \bar B(0,K)$ for any $t$,
\begin{align*}
&\E \int_{|y|>M} \int_{R>|z|>r} \frac{1}{|z|^{d+2\lambda}} \intrd \beta(u_\theta(t,x)-u_\ep(t,y)) [\psi(t, x)-\psi(t,x+z)] \varrho_\delta(x-y)\,dx\,dy\,dt\\
=& 
\int_{R>|z|>r} \frac{1}{|z|^{d+2\lambda}} \E \int_{|y|>M}  \int_{|x-y|<\delta} \beta(u_\theta(t,x)-u_\ep(t,y)) [\psi(t, x)-\psi(t,x+z)] \varrho_\delta(x-y)\,dx\,dy\,dt =0,
\end{align*}
since then $|x|>K$ and $|x+z|>K$. 
\\
Hence $G(\cdot;u_\ep)$ is uniformly integrable, and taking advantage of Young measure theory, we conclude that
\begin{align*}
&\lim_{\ep \to 0} \,\E\bigg[ \intrd \int_{\Pi_T}\beta( u_\theta(t,x)-u_\ep(t,y)) \fr^r[ \psi(t, \cdot)](x) \varrho_\delta(x-y) \,dy\,dx\,dt \bigg]\\
& \qquad \qquad = \E\bigg[ \intrd \int_{\Pi_T} \int^1_0\beta( u_\theta(t,x)-v(t,y,\beta)) \fr^r[ \psi(t, \cdot)](x) \varrho_\delta(x-y) \,d\alpha \,dy\,dx\,dt \bigg].
\end{align*}
A verbatim copy of the above arguments with the Carath\'eodory function $G$ on $\Pi_T\times\Omega\times\R$ defined by 
\begin{align*}
G(t,x,\omega; \mu):= \fr^r[ \psi(t, \cdot)](x) \int_{\R^d}\beta( \mu - v(t,y,\beta)  \varrho_\delta(x-y)dy
\end{align*}
yields
\begin{align*}
&\lim_{\theta \to 0}  \,\E\bigg[ \intrd \int_{\Pi_T} \int^1_0\beta( u_\theta(t,x)-v(t,y,\beta)) \fr^r[ \psi(t, \cdot)](x) \varrho_\delta(x-y) \,d\alpha \,dy\,dx\,dt \bigg]\\
& \qquad \qquad = \E\bigg[ \intrd \int_{\Pi_T} \int^1_0\int^1_0\beta( u(t,x,\alpha)-v(t,y,\beta)) \fr^r[ \psi(t, \cdot)](x) \varrho_\delta(x-y) \,d\alpha\,d\beta \,dy\,dx\,dt\bigg].
\end{align*}
This finishes the proof.

\end{proof}

\begin{lem}
\label{fractional_lemma_one_1}
It holds that 
\begin{align}
I_{10}+J_{10} & \underset{\delta_0 \goto 0} \longrightarrow  -  \E \Big[\int_{\R^d} \int_{\Pi_T}
\int_{\R} \beta(u_\theta(t,x) -k ) \,\mathcal{L}_{\lambda,r} \big[ \psi(t, \cdot)\varrho_{\delta} (\cdot -y)\big](x)
\, \varsigma_l(u_\eps(t,y)-k)\,dk\,dx\,dt\,dy \Big]\notag \\
& -\E \Big[\int_{\Rd} \int_{\Pi_T} \int_{\alpha=0}^1 
\int_{\R} \beta( u_\eps(t,y)-k)\mathcal{L}_{\lambda, r}[\varrho_{\delta}(x-\cdot)](y) \psi(t,x) \, \varsigma_l(u_\theta(t,x)-k)\,dk\,dx\,dt\,dy \Big]  \notag \\
&  \underset{l \to 0}\longrightarrow -  \E \Big[\int_{\R^d} \int_{\Pi_T}  \beta(u_\theta(t,x) -u_\eps(t,y) ) \,\mathcal{L}_{\lambda,r} \big[ \psi(t, \cdot)\varrho_{\delta} (\cdot -y)\big](x) \,dx\,dt \,dy \Big]\notag \\
& -\E \Big[\int_{\Rd} \int_{\Pi_T} \beta( u_\eps(t,y)-u_\theta(t,x))\mathcal{L}_{\lambda, r}[\varrho_{\delta}(x-\cdot)](y) \psi(t,x) \,\, \,dx\,dt\,\,dy \Big]  \notag \\
&\underset{\eps \goto 0} \longrightarrow  -  \E \Big[\int_{\R^d} \int_{\Pi_T} \int^1_0  \beta(u_\theta(t,x) -u(t,y, \alpha) ) \,\mathcal{L}_{\lambda,r} \big[ \psi(t, \cdot)\varrho_{\delta} (\cdot -y)\big](x)  \, d\alpha \,dx\,dt \,dy \Big] \notag\\ 
& - \E \Big[\int_{\R^d} \int_{\Pi_T}  \int^1_0 \beta(u(t,y,\alpha) - u_\theta(t,x) ) \,\mathcal{L}_{\lambda,r} \big[ \varrho_{\delta} (x -\cdot)\big](y)\,\psi(t,x)  \, d\alpha \,dx\,dt \,dy \Big]\notag\\
&\underset{\theta \goto 0} \longrightarrow -  \E \Big[\int_{\R^d} \int_{\Pi_T} \int^1_0 \int^1_0 \beta(v(t,x,\beta) -u(t,y, \alpha) ) \,\mathcal{L}_{\lambda,r} \big[ \psi(t, \cdot)\varrho_{\delta} (\cdot -y)\big](x)  \,d\alpha\, d\beta \,dx\,dt \,dy \Big] \notag\\ 
& - \E \Big[\int_{\R^d} \int_{\Pi_T}  \int^1_0 \int^1_0 \beta(u(t,y,\alpha) - v(t,x,\beta) ) \,\mathcal{L}_{\lambda,r} \big[ \varrho_{\delta} (x -\cdot)\big](y)\,\psi(t,x) \, d\alpha \, d\beta \,dx\,dt \,dy \Big]\notag\\
\end{align}
\end{lem}

\begin{proof} 
We prove this lemma in several steps.\\
\noindent {\bf Step 1} (Passing to the limit as $\delta_0 \to 0$):\\
Consider $K \subset \R^d$ a compact set such that supp$\psi(t,\cdot) \subset K$ for any $t$. Then, thanks to a change of variable in $k$, 
\begin{align*}
&\Bigg| \E \bigg[\int_{\Pi_T} \int_{\Pi_T} 
\int_{\R} \beta(u_\theta(t,x) -k ) \,\mathcal{L}_{\lambda,r} [\varphi_{\delta,\delta_0}(t,\cdot,s,y)](x)
\, \varsigma_l(u_\eps(s,y)-k)\,dk\,dy\,ds \,dx\,dt \bigg]\\
&- \E \bigg[\int_{\Pi_T} \int_{\R^d}  
\int_{\R} \beta(u_\theta(t,x)) -k ) \,\mathcal{L}_{\lambda,r} \big[ \psi(t, \cdot)\varrho_{\delta} (\cdot -y)\big](x) 
\, \varsigma_l(u_\eps(t,y)-k)\,dk\,dy \,dx\,dt \bigg]\Bigg|\\
&\leq c_{\lambda}\Bigg| \E \bigg[\int_{\Pi_T} \int_{\Pi_T}  
\int_{\R}(\beta(u_\theta(t,x) -u_\ep(s,y)+k )-\beta(u_\theta(t,x)-u_\ep(t,y)+k)) \varsigma_l(k)\rho_{\delta_0}(t-s)\\
&\hspace{4cm}\times 
\text{P.V.}\, \int_{|z|\le r} \frac{\psi(t, x)\varrho_{\delta} (x -y) -\psi(t, x+z)\varrho_{\delta} (x+z -y)}{|z|^{d + 2 \lambda}} \,dz
 \,dk \,dx\,dt \,dy\,ds\bigg]\Bigg|\\
& + \Bigg| \E \bigg[ \int\limits_{\Pi_T} \int\limits_{\R^d} 
\int\limits_{\R} \Big(1-\int_t^{\min(T,t+\delta_0)}\hspace*{-1.5cm}\rho_{\delta_0}(t-s)ds\Big)\,\beta(u_\theta(t,x) -k ) \,\mathcal{L}_{\lambda,r} \big[ \psi(t, \cdot)\varrho_{\delta} (\cdot -y)\big](x) 
\, \varsigma_l(u_\eps(t,y)-k)\,dk\,dy \,dx\,dt \bigg] \Bigg|\\
&\leq c(\beta',\psi,\rho_\delta)\Bigg| \E \bigg[\int_{(0,T)^2}\int_{K+\bar B(0,r+\delta)} \int_{K+\bar B(0,r)}  
|u_\ep(s,y)-u_\ep(t,y)|\rho_{\delta_0}(t-s)  \,dx\,dt \,dy\,ds\bigg]\Bigg|\\
& + c(\psi,\rho_\delta) \Bigg| \E \bigg[\int\limits_{K+\bar B(0,r)} \int\limits_0^T\int\limits_{K+\bar B(0,r+\delta)} 
\int\limits_{\R} \Big(1-\int_t^{\min(T,t+\delta_0)}\hspace*{-1.5cm}\rho_{\delta_0}(t-s)ds\Big)\beta(u_\theta(t,x) -k ) \,
\, \varsigma_l(u_\eps(t,y)-k)\,dk\,dy\,dt \,dx \bigg]\Bigg|
\\&\le C(\beta^\prime,\psi,\rho_\delta,K,r,\delta) \E \bigg[\int^T_{s=0}\int^T_{t=0} \int_{K+\bar B(0,r+\delta)} | u_\ep(s,y)-u_\ep(t,y)| \rho_{\delta_0}(t-s) \,dy \,dt \,ds \bigg]\\
& \hspace{5cm} + C \Big(1- \int_0^T\int_t^{\min(T,t+\delta_0)}\hspace*{-1.5cm}\rho_{\delta_0}(t-s)\,ds\,dt\Big)  \underset{\delta_0\to 0} \longrightarrow 0,
\end{align*}
where we have used that
\begin{align*}
& \E \Big[\int\limits_{K+\bar B(0,r)}  \int\limits_{K+\bar B(0,r+\delta)}  \int\limits_{\R} \beta(u_\theta(t,x) -k ) \, \varsigma_l(u_\eps(t,y)-k)\,dk\,dy \,dx \Big]
\\ \leq & 
C(\beta') \E \Bigg[\int\limits_{K+\bar B(0,r)}  \int\limits_{K+\bar B(0,r+\delta)}  \int\limits_{\R} \Big[|u_\theta(t,x)| +|u_\eps(t,y)|+ |u_\eps(t,y)-k| \Big] \, \varsigma_l(u_\eps(t,y)-k)\,dk\,dy \,dx \Bigg]
\\ \leq & 
C(\beta') \E \Big[\int\limits_{K+\bar B(0,r)}  \int\limits_{K+\bar B(0,r+\delta)}   \Big[|u_\theta(t,x)| +|u_\eps(t,y)|+ l \Big] \, dy \,dx \Big]
\\ \leq & C(\beta',K,r,\delta)\Big[l+\sup_t \|u_\theta\|_{L^2(\Omega\times\R^d)}\sup_t \|u_\eps\|_{L^2(\Omega\times\R^d)}\Big] \leq C.
\end{align*}

The same argument yields
\begin{align*}
& \Bigg| \E \bigg[\int_{\Pi_T} \int_{\Pi_T} 
\int_{\R} \beta( u_\eps(s,y)-k)\mathcal{L}_{\lambda, r}[\varphi_{\delta,\delta_0}(t,x,s,.)](y) \varsigma_l(u_\theta(t,x)-k)\,dk\,dx\,dt\,dy\,ds \bigg] \\
&\qquad - \E \bigg[\int_{\Rd} \int_{\Pi_T} 
\int_{\R} \beta( u_\eps(t,y)-k)\mathcal{L}_{\lambda, r}[\varrho_{\delta}(x-\cdot)](y) \psi(t,x) \,\varsigma_l(u_\theta(t,x)-k)\,dk \,dx\,dt\,dy \bigg]\Bigg| \underset{\delta_0\to 0} \longrightarrow 0.
\end{align*}

\noindent {\bf Step 2} (Passing to the limit as $l \to 0$):
Observe that 
\begin{align*}
&\Bigg| \E \bigg[\int_{\R^d} \int_{\Pi_T}
\int_{\R} \beta(u_\theta(t,x) -k ) \,\mathcal{L}_{\lambda,r} \big[ \psi(t, \cdot)\varrho_{\delta} (\cdot -y)\big](x)
\, \varsigma_l(u_\eps(t,y)-k)\,dk  \,dx\,dt\,dy \bigg]\\
&\qquad \qquad -\E \bigg[\int_{\R^d} \int_{\Pi_T}   \beta(u_\theta(t,x) -u_\eps(t,y) ) \,\mathcal{L}_{\lambda,r} \big[ \psi(t, \cdot)\varrho_{\delta} (\cdot -y)\big](x)  \,dx\,dt \,dy \bigg]\Bigg|\\
&= \Bigg|  \E \bigg[\int_{\R^d} \int_{\Pi_T}  
\int_{\R} (\beta(u_\theta(t,x) -k )- \beta(u_\theta(t,x)-u_\ep(t,y)) \mathcal{L}_{\lambda,r} \big[ \psi(t, \cdot)\varrho_{\delta} (\cdot -y)\big](x)\\
&\hspace{8cm}\times \varsigma_l(u_\eps(t,y)-k) \,dk \,dx\,dt\,dy\bigg] \Bigg|\\
&\le C(\beta')\, \E\bigg[\int_{\R^d} \int_{\Pi_T} 
\int_{\R} | u_\ep(t,y)-k| |\mathcal{L}_{\lambda,r} \big[ \psi(t, \cdot)\varrho_{\delta} (\cdot -y)\big](x)| \varsigma_l(u_\eps(t,y)-k) \,dk \,dx\,dt\,dy\bigg]\\
&\le C\,l \,\E\bigg[\int_{\R^d} \int_{\Pi_T} |\mathcal{L}_{\lambda,r} \big[ \psi(t, \cdot)\varrho_{\delta} (\cdot -y)\big](x)| \,dx\,dt\,dy\bigg] 
\underset{l\to 0} \longrightarrow 0.
\end{align*}
Similarly, we have 
\begin{align*}
&\E \bigg[\int_{\Rd} \int_{\Pi_T} 
\int_{\R} \beta( u_\eps(t,y)-k)\mathcal{L}_{\lambda, r}[\varrho_{\delta}(x-\cdot)](y) \psi(t,x) \, \varsigma_l(u_\theta(t,x)-k)\,dk \,dx\,dt\,dy \bigg] \\
& \qquad \qquad \underset{l \to 0}\longrightarrow \E \bigg[\int_{\Rd} \int_{\Pi_T} \beta( u_\eps(t,y)-(u_\theta(t,x))\mathcal{L}_{\lambda, r}[\varrho_{\delta}(x-\cdot)](y) \psi(t,x) \, \,dx\,dt\,\,dy \bigg].
\end{align*}

\noindent {\bf Step 3} (Passing to the limit as $\ep, \theta \to 0$): 
Consider
\begin{align*}
\mathcal{A}_3:=\E \bigg[\int_{\R^d} \int_{\Pi_T}  \beta(u_\theta(t,x) -u_\eps(t,y) ) \,\mathcal{L}_{\lambda,r} \big[ \psi(t, \cdot)\varrho_{\delta} (\cdot -y)\big](x) \,dx\,dt \,dy \bigg]
\end{align*}
As before, let us define
$$G(t,y,\omega,\mu) =\int_{\R^d}\beta(u_\theta(t,x) -\mu)\mathcal{L}_{\lambda,r} \big[ \psi(t, \cdot)\varrho_{\delta} (\cdot -y)\big](x)\,dx.$$
Note that the above integration holds in the compact set $K+\bar B(0,r)$ and $G$ is a Carath\'eodory function.  
Thanks to the compact support of $\psi$,
\begin{align*}
&\E \bigg[\int_{\Pi_T} \Big|\int_{\R^d} \beta( u_\theta(t,x) -u_\ep(t,y))  \mathcal{L}_{\lambda,r} \big[ \psi(t, \cdot)\varrho_{\delta} (\cdot -y)\big](x)\,dx\Big|^2 \,dy\,dt \bigg]\\
& \le C(\beta^\prime, \varrho,\psi) \E \bigg[\int_0^T \int_{K +\bar{B}(0, r +\delta)}\int_{K +\bar{B}(0, r)} | u_\theta(t,x)-u_\ep(t,y)|^2 \,dx\,dy \,dt \bigg] \le C,
\end{align*}
$G(\cdot , u_\ep(\cdot))$ is uniformly integrable and the Young measure theory gives
\begin{align*}
\mathcal{A}_3 \underset{\ep \to 0} \longrightarrow  \E \bigg[\int_{\R^d} \int_{\Pi_T} \int_{\beta=0}^1  \beta(u_\theta(t,x) -v(t,y, \beta) ) \,\mathcal{L}_{\lambda,r} \big[ \psi(t, \cdot)\varrho_{\delta} (\cdot -y)\big](x)  \, d\beta \,dx\,dt \,dy \bigg]
\end{align*}
Similarly, it can be shown that 
\begin{align*}
& \E \bigg[\int_{\R^d} \int_{\Pi_T} \int_{\beta=0}^1  \beta(u_\theta(t,x) -v(t,y,\beta) ) \,\mathcal{L}_{\lambda,r} \big[ \psi(t, \cdot)\varrho_{\delta} (\cdot -y)\big](x)  \, d\beta \,dx\,dt \,dy \bigg]\\ &\underset{\theta\to 0}\longrightarrow \E \bigg[\int_{\R^d} \int_{\Pi_T}\int^1_{\alpha=0} \int_{\beta=0}^1  \beta(u(t,x,\alpha) -v(t,y, \beta) ) \,\mathcal{L}_{\lambda,r} \big[ \psi(t, \cdot)\varrho_{\delta} (\cdot -y)\big](x)\,d\alpha  \, d\beta \,dx\,dt \,dy \bigg]
\end{align*}
For the other term, we consider
$$\mathcal{B}_3:= \E \bigg[\int_{\Rd} \int_{\Pi_T}\beta( u_\eps(t,y)-u_\theta(t,x))\mathcal{L}_{\lambda, r}[\varrho_{\delta}(x-\cdot)](y) \psi(t,x) \, \,dx\,dt\,\,dy \bigg]$$
Following the same analysis as for the term $\mathcal{A}_3$, we conclude
\begin{align*}
\mathcal{B}_3 &\underset{\ep \to 0} \longrightarrow  \E \bigg[\int_{\R^d} \int_{\Pi_T}  \int_{\beta=0}^1 \beta(v(t,y,\beta) - u_\theta(t,x) ) \,\mathcal{L}_{\lambda,r} \big[ \varrho_{\delta} (x -\cdot)\big](y)\,\psi(t,x) \, d\beta \,dx\,dt \,dy \bigg]\\
&\underset{\theta \to 0}\longrightarrow \E \bigg[\int_{\R^d} \int_{\Pi_T} \int_{\alpha=0}^1 \int_{\beta=0}^1 \beta(v(t,y,\beta) - u(t,x,\alpha) ) \,\mathcal{L}_{\lambda,r} \big[ \varrho_{\delta} (x -\cdot)\big](y)\,\psi(t,x) \,d\alpha \, d\beta \,dx\,dt \,dy \bigg].
\end{align*}
This concludes the proof.
\end{proof}


\begin{lem}
\label{fractional_lemma_two_1}
It holds that 
\begin{align*}
&\limsup_{\xi\goto 0} \limsup_{\theta \goto 0}\limsup_{\eps \goto 0}\lim_{l\goto 0}\lim_{\delta_0\goto 0} \big(I_9 + J_9 + I_{10}+J_{10} \big)\\
& \le  -\E \bigg[\int_{\R^d} \int_{\Pi_T} \int_{\alpha=0}^1\int^1_{\beta=0}\fr^r[\psi(t,\cdot)](x) \varrho_{\delta}(x-y)|v(t,x,\alpha)-u(t,y,\beta)| \,d\alpha\, d\beta \,dy\,dx\,dt \bigg]\\
& -  \E \bigg[\int_{\R^d} \int_{\Pi_T} \int_{\alpha=0}^1\int_{\beta=0}^1  |u(t,x, \alpha) -v(t,y, \beta) | \,\mathcal{L}_{\lambda,r} \big[ \psi(t, \cdot)\varrho_{\delta} (\cdot -y)\big](x) \,d\alpha \, d\beta \,dx\,dt \,dy \bigg] \notag\\ 
& - \E \bigg[\int_{\R^d} \int_{\Pi_T} \int_{\alpha=0}^1 \int_{\beta=0}^1 |v(t,y,\beta) - u(t,x,\alpha)| \,\mathcal{L}_{\lambda,r} \big[ \varrho_{\delta} (x -\cdot)\big](y)\,\psi(t,x) \,d\alpha \, d\beta \,dx\,dt \,dy \bigg]\\
& \underset{r \goto 0} \longrightarrow - \E \bigg[\int_{\R^d} \int_{\Pi_T} \int_{\alpha=0}^1\int_{\beta=0}^1 
|u(t,x, \alpha) - v(t,y, \beta)| \, \fr[\psi(t,\cdot)](x)\, \varrho_{\delta} (x-y)  \,d\alpha \,d\beta \,dx\,dt \,dy \bigg] \notag \\
& \underset{\delta \goto 0} \longrightarrow - \E \bigg[\int_{\Pi_T} \int_{\alpha=0}^1\int_{\beta=0}^1 
|u(t,x, \alpha) - v(t,x, \beta)| \, \fr[\psi(t, \cdot)](x)\,d\alpha \,d\beta \,dx\,dt  \bigg]
\end{align*}
\end{lem}

\begin{proof}
We will prove this lemma in several steps. First note that, from the two previous lemmas, one has that
\begin{align*}
&\limsup_{\theta \goto 0}\limsup_{\eps \goto 0}\lim_{l\goto 0}\lim_{\delta_0\goto 0} \big(I_9 + J_9 + I_{10}+J_{10} \big) 
\\ \leq& -  \E \bigg[\int_{\R^d} \int_{\Pi_T} \int_{\alpha=0}^1\int_{\beta=0}^1  \beta(v(t,x, \alpha) -u(t,y, \beta) ) \,\mathcal{L}_{\lambda,r} \big[ \psi(t, \cdot)\varrho_{\delta} (\cdot -y)\big](x) \,d\alpha \, d\beta \,dx\,dt \,dy \bigg] \notag\\ 
& - \E \bigg[\int_{\R^d} \int_{\Pi_T} \int_{\alpha=0}^1 \int_{\beta=0}^1 \beta(u(t,y,\beta) - v(t,x, \alpha) ) \,\mathcal{L}_{\lambda,r} \big[ \varrho_{\delta} (x -\cdot)\big](y)\,\psi(t,x) \,d\alpha \, d\beta \,dx\,dt \,dy \bigg]\notag
\\
&- \E \bigg[\int_{\R^d} \int_{\Pi_T} \int_{\alpha=0}^1\int^1_{\beta=0}\fr^r[\psi(t,\cdot)](x) \varrho_{\delta}(x-y)\beta(v(t,x,\alpha)-u(t,y,\beta)) \,d\alpha\, d\beta \,dy\,dx\,dt \bigg]
\end{align*}

\noindent {\bf Step 1} (Passing to the limit as $\xi \to 0$):
Recall that $|\beta_\xi(r) -|r|| \le M_1\xi$ (cf. \eqref{eq:approx to abosx}), so we have
\begin{align*}
&\Bigg| \E \bigg[\int_{\R^d} \int_{\Pi_T} \int_{\alpha=0}^1\int^1_{\beta=0}\fr^r[\psi(t,\cdot)](x) \varrho_{\delta}(x-y)\beta_\xi(v(t,x,\alpha)-u(t,y,\beta)) \,d\alpha\, d\beta \,dy\,dx\,dt \bigg]\\
&- \E \bigg[\int_{\R^d} \int_{\Pi_T} \int_{\alpha=0}^1\int^1_{\beta=0}\fr^r[\psi(t,\cdot)](x) \varrho_{\delta}(x-y)|v(t,x,\alpha)-u(t,y,\beta)| \,d\alpha\, d\beta \,dy\,dx\,dt \bigg]\Bigg|\\
&\le c_{\lambda}M_1 \xi \E \bigg[\int_{\R^d} \int_{\Pi_T}  \varrho_{\delta}(x-y) \int_{|z|> r} \frac{|\psi(t,x) -\psi(t,x+z)|}{|z|^{d + 2 \lambda}} \,dz \,dy\,dx\,dt\bigg] \underset{\xi \goto 0} \longrightarrow 0.
\end{align*}
Similarly and taking into account the compact support of $\psi$ as in the proof of Lemma \ref{fractional_lemma_one_1}, we have for the other two terms: 
\begin{align*}
& \E \bigg[\int_{\R^d} \int_{\Pi_T} \int_{\alpha=0}^1\int_{\beta=0}^1  \beta(u(t,x, \alpha) -v(t,y, \beta) ) \,\mathcal{L}_{\lambda,r} \big[ \psi(t, \cdot)\varrho_{\delta} (\cdot -y)\big](x) \,d\alpha \, d\beta \,dx\,dt \,dy \bigg] \\ 
& \quad \underset{\xi \to 0} \longrightarrow \E \bigg[\int_{\R^d} \int_{\Pi_T} \int_{\alpha=0}^1\int_{\beta=0}^1  |u(t,x, \alpha) -v(t,y, \beta) | \,\mathcal{L}_{\lambda,r} \big[ \psi(t, \cdot)\varrho_{\delta} (\cdot -y)\big](x) \,d\alpha \, d\beta \,dx\,dt \,dy \bigg] \notag\\ 
& \E \bigg[\int_{\R^d} \int_{\Pi_T} \int_{\alpha=0}^1 \int_{\beta=0}^1 \beta(v(t,y,\beta) - u(t,x, \alpha) ) \,\mathcal{L}_{\lambda,r} \big[ \varrho_{\delta} (x -\cdot)\big](y)\,\psi(t,x) \,d\alpha \, d\beta \,dx\,dt \,dy \bigg]\\
& \quad \underset{\xi \to 0} \longrightarrow \E \bigg[\int_{\R^d} \int_{\Pi_T} \int_{\alpha=0}^1 \int_{\beta=0}^1 |v(t,y,\beta) - u(t,x,\alpha)| \,\mathcal{L}_{\lambda,r} \big[ \varrho_{\delta} (x -\cdot)\big](y)\,\psi(t,x) \,d\alpha \, d\beta \,dx\,dt \,dy \bigg]
\end{align*}

\noindent {\bf Step 2} (Passing to the limit as $r \to 0$):
First note that (cf. \cite{CifaniJakobsen}) for regular $\phi$ being $\psi$, $\psi\varrho_{\delta}(x-\cdot)$ or $\varrho_{\delta}(x-\cdot)$ in the sequel,
\[
|\mathcal{L}_{\lambda,r}[\phi](x)|\le 
\begin{cases}
c_{\lambda}\|D\phi\|_{L^\infty} \int_{|z|\le r} \frac{|z|}{|z|^{d+2\lambda}}\,dz, & \lambda \in (0,1/2) \\[2mm]
\frac{c_{\lambda}}{2}\|D^2\phi\|_{L^\infty}\int_{|z|\le r} \frac{|z|^2}{|z|^{d+2\lambda}}\,dz, & \lambda \in [1/2,1).
\end{cases}
\]
Thus we see that in both cases $|\mathcal{L}_{\lambda,r}[\phi](x)| \le c(\phi) r^s$ for some $s>0$ and  $\lim_{r\to 0}|\mathcal{L}_{\lambda,r}[\phi](x)|=0$.
On the other hand, since supp$\psi(t,\cdot) \subset K$ for any $t$, assuming $r+\delta<1$ one gets 
\begin{align*}
\bigg|\mathcal{L}_{\lambda,r} \big[ \psi(t, \cdot)\varrho_{\delta} (\cdot -y)\big](x)\bigg|=&
\bigg|\mathcal{L}_{\lambda,r} \big[ \psi(t, \cdot)\varrho_{\delta} (\cdot -y)\big](x)\bigg|1_{K + \bar B(0,1)}(x)1_{K + \bar B(0,1)}(y)
\\[1mm]
\leq& C(\psi,\rho_\delta) r^s 1_{K + \bar B(0,1)}(x)1_{K + \bar B(0,1)}(y),
\\[2mm]
\bigg| \mathcal{L}_{\lambda,r} \big[ \varrho_{\delta} (x -\cdot)\big](y)\,\psi(t,x) \bigg|
=& 
\bigg| \mathcal{L}_{\lambda,r} \big[ \varrho_{\delta} (x -\cdot)\big](y)\,\psi(t,x) \bigg| 1_{K}(x)1_{K + \bar B(0,1)}(y)
\\[1mm]
\leq& C(\psi,\rho_\delta) r^s 1_{K}(x)1_{K + \bar B(0,1)}(y).
\end{align*}
Therefore
\begin{align*}
& \E \bigg[\int_{\R^d} \int_{\Pi_T} \int_{\alpha=0}^1\int_{\beta=0}^1  |(u(t,x, \alpha) -v(t,y, \beta) | \,\mathcal{L}_{\lambda,r} \big[ \psi(t, \cdot)\varrho_{\delta} (\cdot -y)\big](x) \,d\alpha \, d\beta \,dx\,dt \,dy \bigg] \notag\\ 
&\quad +\E \bigg[\int_{\R^d} \int_{\Pi_T} \int_{\alpha=0}^1 \int_{\beta=0}^1 |v(t,y,\beta) - u(t,x,\alpha)| \,\mathcal{L}_{\lambda,r} \big[ \varrho_{\delta} (x -\cdot)\big](y)\,\psi(t,x) \,d\alpha \, d\beta \,dx\,dt \,dy \bigg]\underset{ r \to 0} \longrightarrow 0\\
\end{align*}
On the other hand, using $\fr[\varphi] := \mathcal{L}_{\lambda, r}[\varphi] + \mathcal{L}_{\lambda}^{r}[\varphi]$, we see that using similar arguments,
\begin{align*}
&\E \bigg[\int_{\R^d} \int_{\Pi_T} \int_{\alpha=0}^1\int^1_{\beta=0}\fr^r[\psi(t,\cdot)](x) \varrho_{\delta}(x-y)|v(t,x,\alpha)-u(t,y,\beta)| \,d\alpha\, d\beta \,dy\,dx\,dt \bigg]\\
&=\E \bigg[\int_{\R^d} \int_{\Pi_T} \int_{\alpha=0}^1\int^1_{\beta=0}(\fr[\psi(t,\cdot)](x)- \mathcal{L}_{\lambda, r}[\psi(t,\cdot)](x)) \varrho_{\delta}(x-y)|v(t,x,\alpha)-u(t,y,\beta)| \,d\alpha\, d\beta \,dy\,dx\,dt \bigg]\\
&\qquad \underset{ r\to 0} \longrightarrow \E \bigg[\int_{\R^d} \int_{\Pi_T} \int_{\alpha=0}^1\int_{\beta=0}^1 
|u(t,x, \alpha) - v(t,y, \beta)| \, \fr[\psi(t,\cdot)](x)\, \varrho_{\delta} (x-y)  \,d\alpha \,d\beta \,dx\,dt \,dy \bigg]
\end{align*}

\noindent {\bf Step 3} (Passing to the limit as $\delta \to 0$):
Let us consider  
\begin{align*}
&\Bigg| \E \bigg[\int_{\R^d} \int_{\Pi_T} \int_{\alpha=0}^1\int_{\beta=0}^1 
|v(t,x, \alpha) - u(t,y, \beta)| \, \fr[\psi(t,\cdot)](x)\, \varrho_{\delta} (x-y)  \,d\alpha \,d\beta \,dx\,dt \,dy \bigg]\\
& - \E \bigg[\int_{\Pi_T} \int_{\alpha=0}^1\int_{\beta=0}^1 
|v(t,x, \alpha) - u(t,x, \beta)| \, \fr[\psi(t, \cdot)](x)\,d\alpha \,d\beta \,dx\,dt  \bigg]\Bigg|\\
& \le \E \bigg[\int_{\R^d} \int_{\Pi_T} \int_{\alpha=0}^1\int_{\beta=0}^1\Big| |v(t,x,\alpha)- u(t,y,\beta)|-|v(t,x,\alpha)-u(t,x,\beta)|\Big|\\
&\hspace{7cm}\times |\fr[\psi(t, \cdot)](x)| \varrho_\delta(x-y) \,d\alpha \, d\beta \,dx \,dt \,dy \bigg]\\
& \le \E\bigg[\int_{\R^d} \int_{\Pi_T} \int_{\alpha=0}^1| u(t,y,\beta)-u(t,x,\beta)| |\fr[\psi(t, \cdot)](x)| \varrho_\delta(x-y) \,d\beta \,dx\,dy\,dt\bigg]\\
&\le C(\psi) \,\E\bigg[\int_{\R^d} \int_{\Pi_T} \int_{\alpha=0}^1| u(t,y,\beta)-u(t,x,\beta)|^2 \varrho_\delta(x-y) \,d\beta \,dx\,dy\,dt\bigg]^{1/2}\underset{\delta \to 0} \longrightarrow 0,
\end{align*}
where in the last step, we have used the continuity of translations in $L^2$. Indeed
\begin{align*}
&\E\bigg[\int_{\R^d} \int_{\Pi_T} \int_{\alpha=0}^1| u(t,y,\beta)-u(t,x,\beta)|^2 \varrho_\delta(x-y) \,d\beta \,dx\,dy\,dt\bigg]^{1/2} \\
& \qquad = 
\bigg[\int_{\R^d}\int_{\R^d} \| u(y)-u(x)\|_{L^2(\Omega\times(0,T)\times(0,1))}^2 \varrho_\delta(x-y)  \,dx\,dy\bigg]^{1/2} \\
& \qquad \leq
\bigg[\int_{B(0,\delta)}\ \int_{\R^d} \| u(x+z)-u(x)\|_{L^2(\Omega\times(0,T)\times(0,1))}^2\,dx \varrho_\delta(z)  \,dz\bigg]^{1/2} \\
& \qquad \leq
\bigg[\sup_{|z|<\delta}\int_{\R^d} \| u(x+z)-u(x)\|_{L^2(\Omega\times(0,T)\times(0,1))}^2\,dx \bigg]^{1/2} \underset{\delta \to 0} \longrightarrow 0.
\end{align*}
This essentially finishes the proof of the lemma.
\end{proof}

Finally, making use of Lemma~\ref{lem:initial+time-terms} to \ref{fractional_lemma_two_1}, we get the expected Kato inequality 
\begin{align*}
0 \leq \int_\D |u_0-v_0|\psi(0)\,dx & + \E\bigg[ \int_{\Pi_T}\int_{\alpha=0}^1\int_{\beta=0}^1  |u(\alpha,t,x) -v(\beta,t,x)| \,\partial_t \psi(t,x) \,d\alpha \,d\beta \,dx \,dt \bigg] \\
& \qquad- \E\bigg[ \int_{\Pi_T}\int_{\alpha=0}^1\int_{\beta=0}^1 F(u(\alpha,t,x),v(\beta,t,x)) \, \nabla \psi(t,x)\,d\alpha \, d\beta \,dx \,dt \bigg] \\
& \qquad \qquad -\E\bigg[\int_{\Pi_T} \int_{\alpha=0}^1\int_{\beta=0}^1 |u(\alpha,t,x)-v(\beta,t,x)| \, \mathcal{L}_\lambda[\psi(t, \cdot)](x) \,d\alpha \, d\beta\,dx\,dt \bigg],
\end{align*}
\textit{a priori} for any non-negative $\psi \in \mathcal{D}([0,T[\times\D)$, but for any non-negative $\psi \in L^2(0,T,H^2(\D))\cap H^1(Q)$ by a density argument.


\subsection{Uniqueness of (Measure-Valued (mild)) Solution}
\label{Uniqueness of (Measure-Valued) Entropy Solution}
To begin with, observe that 
\begin{align}
&0 \leq \int\limits_\D  |u_0-v_0|\psi(0)\,dx \label{important_01}\\
&\qquad + \E\bigg[ \int\limits_{\Pi_T}\int\limits_{\alpha=0}^1\int\limits_{\beta=0}^1  |u(\alpha,t,x) -v(\beta,t,x)|\Big[ \partial_t \psi(t,x) - \, \mathcal{L}_\lambda[\psi(t, \cdot)](x) + \|\f^\prime\|_\infty|\nabla \psi|(t,x) \Big]d\alpha \, d\beta\,dx\,dt \bigg].\notag
\end{align}
To proceed further, we make a special choice for the function $\psi$. The method consists in applying a transposition method of Holmgreens type to \eqref{important_01}, thus, one has to test this above mentioned inequality with a regular solution to a certain backward nonlocal problem. For that purpose first we define $\theta(t) =1 -\frac{t}{T}$, for $t \in [0,T]$. Next, we define a function $\psi_R$ as follows: For $R\geq1$ and $a=d/2 +\eps$, with $\eps>0$
\begin{align}
\label{psi}
\psi_R(x) = \min \Bigg(1, \frac{R^a}{|x|^a} \Bigg).
\end{align}
Note that our choice of $a$ guarantees that $\psi_R$  in  $L^2(\D)$. Moreover, notice that
\begin{itemize}
\item [(a)] $\nabla \psi_R(x) = -aR^a |x|^{-a-1} \frac{x}{|x|}\chi_{\{|x|>R\}} = -a\frac{\psi_R(x)}{|x|}\frac{x}{|x|} \chi_{\{|x|>R\}} \in L^2(\D)$, \ $ |\nabla \psi_R|(x) \leq \frac{a}{R} \psi_R(x)$ in $\R^d$. 
\\[0.1mm]
\item [(b)] $\Delta \psi_R(x) = a(2 + 2\eps -a)\, \frac{\psi_R(x)}{R^2}\chi_{\{|x|>R\}} \in L^2(\lbrace {|x|>R} \rbrace)$, \quad $ |\Delta \psi_R|(x) \leq \frac{C}{R^2} \psi_R(x)$ in $\{|x|>R\}$.
\end{itemize}
Then, following \cite{Alibaud}, set $\mathcal{K}$ the kernel of $\mathcal{L}_{\lambda}$, \textit{i.e.} $\mathcal{K}(t,\cdot)\underset{x}\star u_0$ is solution to $\partial_tu+\mathcal{L}_{\lambda}u=0$, for the initial condition $u_0$. $\mathcal{K}(t,\cdot) \geq 0$ is integrable in $\D$ and $\mathcal{K}(t,\cdot)\underset{x}\star \mathcal{K}(s,\cdot)=\mathcal{K}(t+s,\cdot)$.
For any positive $\delta$, $\mathcal{K} \underset{x}\star \rho_{\delta} \in C^\infty_b([0,+\infty[\times\D)$ is a classical solution to  $\partial_tu+\mathcal{L}_{\lambda}u=0$ and integrable on $[0,T]\times\D$.
\\
With the help of the above functions, we define $\psi(t,x) =  \theta(t) \, \big(\psi_R (\cdot) \underset{x} \star \mathcal{K}_{\delta} (\cdot,t)\big)(x)$, where $\mathcal{K}_{\delta} (x,t)=\big( \mathcal{K} \underset{x}\star \rho_{\delta} \big)(x, \tau-t)$ for a given $\tau > T$.
\\
A simple observation reveals that $\mathcal{K}_{\delta} (x,t)$ satisfies $\partial_t \mathcal{K}_{\delta} - \mathcal{L}_{\lambda}[\mathcal{K}_{\delta}] =0$, and that $\partial_t \psi -  \mathcal{L}_{\lambda}[\psi] = \theta'(t) \, \psi_R \star \mathcal{K}_{\delta} (x,t)$.
Therefore, with this choice of test function $\psi$ in \eqref{important_01}, along with the assumption that $u_0 =v_0$, we have
\begin{align*}
0 \le & \E \bigg[\int_0^T \dint_{\D\times(0,1)^2} |u(\alpha,t,x) -v(\beta,t,x)| \, \psi_R \star \mathcal{K}_{\delta} (x,t) \,dx\,d\alpha \, d\beta \bigg]\,dt \\
&\quad \le  \frac{aT}{R}\|\f^\prime\|_\infty\, \int_0^T \E \bigg[\dint_{\{|x|>R\}\times(0,1)^2} |u(\alpha,s,x) -v(\beta,s,x)| \, \mathcal{K}_{\delta} (t, \cdot) \star  \psi_R(x)  \,dx\,d\alpha \, d\beta \bigg]\,ds. 
\end{align*}
This implies that 
\begin{align*}
0 \le &\frac12 \E \bigg[\int_0^T \dint_{\D\times(0,1)^2} |u(\alpha,t,x) -v(\beta,t,x)| \,  \psi_R \star \mathcal{K}_{\delta} (x,t) \,dx\,d\alpha \, d\beta \bigg]\,dt \\
& \le \int_0^T \Big( \frac{aT}{R}\|\f^\prime\|_\infty-\frac12 \Big) \E \bigg[\dint_{\D\times(0,1)^2} |u(\alpha,s,x) -v(\beta,s,x)| \, \mathcal{K}_{\delta} (t, \cdot) \star  \psi_R(x)  \,dx\,d\alpha \, d\beta \bigg]\,ds. 
\end{align*}
Observe that for large $R$, right hand side of the above inequality is non-positive. 
So, $R$ big enough yields
\begin{align*}
\E \bigg[\int_{\Pi_T\times(0,1)^2}  |u(\alpha,t,x) -v(\beta,t,x)|\,  \mathcal{K}_{\delta} (t, \cdot) \star  \psi_R(x)  \,dx\,d\alpha d\beta\,dt \bigg] = 0.
\end{align*}
Note that $\mathcal{K}_{\delta} (t, \cdot) \star  \psi_R(x) \to \|\mathcal{K}_\delta(t)\|_{L^1(\D)}$ and 
\begin{align*}
\|\mathcal{K}_\delta(t)\|_{L^1(\D)} &=\int_{\R^{2d}}\mathcal{K}(y,\tau-t)\,\rho_\delta(y-x)dydx=\int_\D \|\mathcal{K}(\tau-t)\|_{L^1(\D)} \, \rho_\delta(y-x)\,dx =1,
\end{align*}
so that, by Fatou's lemma
\begin{align*}
\E \bigg[\dint_{\Pi_T \times(0,1)^2} |u(\alpha,t,x) -v(\beta,t,x)|  \,dx\,d\alpha d\beta \bigg]=0.
\end{align*}
This ensures the uniqueness of the measure valued (mild) solution coming from a viscous regularization. Moreover, the above equality also implies that this unique measure valued (mild) solution is independent of its additional variables $\alpha$ or $\beta$. On the other hand, we conclude that the whole sequence of viscous approximation converges weakly in $L^2(\Omega \times \Pi_T)$. Since the limit process is independent of the additional variable, the viscous approximation converges strongly in $L^p(\Omega \times (0,T)\times B(0,M))$, for any $M>0$ and any $1\le p <2 $.

\begin{rem}[measurability of viscous solution]
\label{measurability}
Note that since $u_{\eps}$ is bounded in the Hilbert space $\mathcal{N}^2_w(0,T,L^2(\D))$, by identification, it is easy to show that $u_\eps$ converges weakly to $\mathfrak{u}:= \int_0^1u(\cdot, \alpha) d\alpha$ in the same space, so $\int_0^1u(\cdot, \alpha) d\alpha$ is a predictable process. The interesting point is the measurability of $\mathfrak{u}$ with respect to all the variables $(t,x,\omega,\alpha)$. In fact, one can follow the work of Panov \cite{Panov} and achieve the desired measurability of $\mathfrak{u}$. 
\end{rem}


\subsection{Uniqueness (Alternative Method for $d>1$):}
Let $f,\psi \in \mathcal{D}^+(\D)$, and $\Phi$ be the (variational) solution to $(-\Delta)^\lambda [\Phi] = f$. Moreover, we choose a specific test function $\varphi(t,x)=\theta(t)\psi*\Phi$. 
If $d>1$ (see Section \ref{SolPositive}), $\Phi(x)$ is given by the Riesz potential of $f$, i.e., $\Phi(x) = I_{2\lambda}(f)(x)$. Note that it is positive (strictly) as soon as $f> 0$. Then the above choice of test function in \eqref{important} yields
\begin{align*}
&0 \leq \theta(0)\int_\D |u_0-v_0|\psi*\Phi(x) \,dx -\E \bigg[ \theta(t)  |u(\alpha,t,x)-v(\beta,t,x)|  \psi*f \,dx\,d\alpha d\beta\,dt \bigg]
\\ &+ \E \bigg[ \int_{\Pi_T}\int_0^1\int_0^1 |u(\alpha,t,x) -v(\beta,t,x)| \theta^\prime(t) \psi*\Phi(x) - \theta(t) F(u(\alpha,t,x),v(\beta,t,x))(\nabla \psi)*\Phi (x)\,dx\,d\alpha d\beta\,dt\bigg].
\end{align*}
By a density argument, the above inequality also holds for any non-negative $\psi$ in $H^1(\R^d)$ and $\theta \in H^{1}(0,T)$ with $\theta(T)=0$.
Assume that $\psi=\psi_R$ is the function already defined in the previous section. Then 
\begin{align*}
&\E \bigg[\int_{\Pi_T}\int_0^1\int_0^1 \theta(t)  |u(\alpha,t,x)-v(\beta,t,x)|  \psi_R*f \,dx\,d\alpha d\beta\,dt \bigg] \\ & \qquad \leq \theta(0)\int_\D |u_0-v_0|\psi_R*\Phi(x) \,dx \\
& \qquad \qquad + \int_0^T \Big[\theta^\prime(t)  + \frac{a}{R}\|\f^\prime\|_\infty \theta(t)\Big] \E \bigg[\int_{\D} \int_0^1\int_0^1  \psi_R*\Phi (x) |u(\alpha,t,x) -v(\beta,t,x)|  \,dx\,d\alpha d\beta\,dt\bigg].
\end{align*}
Set $u_0=v_0$ and $\theta(t)=1-\frac{t}{T}$. Since for big enough values of $R$,  the right hand side of the above inequality is non positive, passing to the limit $R \to \infty$ as in the previous section yields 
\begin{align*}
\E \bigg[\int_{\Pi_T}\int_0^1\int_0^1 \theta(t)  |u(\alpha,t,x)-v(\beta,t,x)|  \|f\|_{L^1(\D)} \,dx\,d\alpha d\beta\,dt \bigg]  \leq 0,
\end{align*}
and the conclusion is the same.


\subsection{Existence of Entropy Solution}
\label{entropy_existence}

The aim of this subsection is to prove the result of existence of an entropy solution, in the sense of Definition~\ref{defi:stochentropsol}, using the strong convergence of the sequence of viscous solutions we had in a previous section and \textit{a priori} bounds (cf. \textit{a priori} estimates \ref{regularity}). 
 
In what follows, following the calculations leading up to Definition~\ref{defi:stochentropsol}, for any $\psi \in \mathcal{D}^{+}(\Pi_T)$, any pair of entropy-entropy flux pair $(\beta, \zeta)$, with $\beta$ convex, and for any $\mathbb{P}$-measurable set $A$, the viscous solution $u_{\eps}$ satisfies the following inequality:
\begin{align*}
0 & \le  \E \Big[\mathds{1}_{A}\, \int_{\R^d} \beta(u_{\eps}(0,x))\, \psi(0,x)\,dx \Big]+ 
 \E \Big[\mathds{1}_{A}\, \int_{\Pi_T} \Big\{\beta(u_{\eps}(t,x)) \,\partial_t\psi(t,x) -  \grad \psi(t,x)\cdot \zeta(u_{\eps}(t,x)) \Big\}\,dx\,dt \Big] \\
 & \quad + \E \Big[\mathds{1}_{A}\,  \sum_{k\ge 1}\int_{\Pi_T} g_k(u_\eps(t,x))\beta^\prime (u_\eps(t,x))\psi(t,x)\,d\beta_k(t)\,dx \Big] \notag \\
 & \quad \quad + \frac{1}{2} \E \Big[\mathds{1}_{A}\, \int_{\Pi_T}\mathbb{G}^2(u_\eps(t,x))\beta^{\prime\prime} (u_\eps(t,x))\psi(t,x)\,dx\,dt \Big] + \mathcal{O}(\eps) \notag \\
  &  \quad + \E \Big[\mathds{1}_{A}\, \int_{\Pi_T} \int_{E} \int_0^1 \eta(u_\eps(t,x);z)\beta^\prime \big(u_\eps(t,x) + \lambda\,\eta(u_\eps(t,x);z)\big)\psi(t,x)\,d\lambda\,\widetilde{N}(dz,dt)\,dx \Big]  \notag \\
 & \quad + \E \Big[\mathds{1}_{A}\,\int_{\Pi_T} \int_{E}  \int_0^1  (1-\lambda)\eta^2(u_\eps(t,x);z)\beta^{\prime\prime} \big(u_\eps(t,x) + \lambda\,\eta(u_\eps(t,x);z)\big) 
 \psi(t,x)\,d\lambda\,m(dz)\,dx\,dt \Big] \notag \\
 & \quad -  \E \Big[\mathds{1}_{A}\,\int_{\Pi_T} \Big\{ \mathcal{L}_{\lambda}^{r}[u_\eps(t,\cdot)](x)\, \psi(t,x)\, \beta'(u_\eps(t,x)) + \beta(u_\eps(t,x)) \, \mathcal{L}_{\lambda, r}[\psi(t,\cdot)](x) \Big\}\,dx\,dt \Big]. 
\end{align*} 
For convenience, we will denote the above inequality by  
\begin{align}
\label{final}
0 \leq \int_{A} \mu^r_{\lambda, \beta}[u_{\eps}](\psi)\,d\bP - \int_{A}\int_{\Pi_T} \Big[ \psi\, \mathcal{L}_{\lambda}^{r}[u_\eps] \,\beta'(u_{\eps}) + \beta(u_{\eps}) \, \mathcal{L}_{\lambda, r}[\psi] \Big]\,dx\,dt \,d\bP.
\end{align}
Note that thanks to Section \ref{Uniqueness of (Measure-Valued) Entropy Solution}, for any Carath\'eodory function $\Theta$ from $\Omega\times\Pi_T\times\R$ to $\R$ such that $\Theta(\cdot,u_\epsilon)$ is uniformly integrable, then $\E \int_{\Pi_T}\Theta(\cdot,u_\epsilon)dxdt \to \E \int_{\Pi_T}\Theta(\cdot,u)dxdt$. Thus, using the same strategy as depicted in \cite{BaVaWit_2012, BisMajKarl_2014}, it is possible to pass to the limit in the first term on the right hand side of the above relation \eqref{final}. Therefore, $u \in N^2_w(0,T,L^2(\D))$ and $\int_A \mu^r_{\eta,k}[u_\epsilon](\psi) \, d\bP$ converges to $\int_A \mu^r_{\eta,k}[u](\psi) \,d\bP$ for any measurable subset $A \subset \Omega$. Hence, we only need to check the passage to the limit in the second term on the right hand side of \eqref{final}.

Observe that, since $ \mathcal{L}_{\lambda, r}[\psi]$ is a bounded measurable function with a compact support (depending on the one of $\psi$ and $r$), one gets that 
\begin{align*}
\int_A\int_{\Pi_T} \beta(u_{\eps}) \, \mathcal{L}_{\lambda, r}[\psi] \,dx\,dt \,d\bP 
 \to \int_A\int_{\Pi_T} \beta(u) \, \mathcal{L}_{\lambda, r}[\psi] \,dx\,dt \,d\bP.
\end{align*}
Finally, concerning the last remaining term, note that $\frac{1}{|z|^{d+2\lambda}}1_{|z|>r}$ is in all $L^p(\D)$, $p\ge1$
and that $\mathcal{L}^r_{\lambda}$ is a linear and continuous operator from $L^2(\Omega\times \Pi_T)$ into itself. Therefore, knowing that $u_\eps$ converges weakly to $u$ in $L^2(\Omega\times \Pi_T)$, $\mathcal{L}^r_{\lambda}(u_\eps)$ converges also weakly to $\mathcal{L}^r_{\lambda}(u)$ in the same space. 
\\
Since $u_\eps$ converges strongly in $L^p(\Omega \times (0,T)\times B(0,M))$, for any $M>0$ and any $1\le p <2 $, thanks the boundedness of $\beta'$, of $\psi$ and of its support, $\beta'(u_\eps)\psi$ converges to $\beta'(u)\psi$ in $L^2(\Omega\times \Pi_T)$ and 
 we conclude that 
 $$\int_A\int_{\Pi_T}  \mathcal{L}^r_{\lambda}[u_\ep(t,\cdot)](x) \psi(t,x) \beta'(u_\ep(t,x))\,dx\,dt\,dP \to \int_A\int_{\Pi_T} \mathcal{L}^r_{\lambda}[u(t,\cdot)](x) \psi(t,x) \beta'(u(t,x))\,dx\,dt\,dP.$$
This proves that $u$ is an entropy solution of \eqref{eq:stoc_con_brown}, in the sense of Definition~\ref{defi:stochentropsol}.


\subsection{Uniqueness of Entropy Solution}
\label{Unique}
As alluded to before, to ensure the uniqueness of entropy solutions, we compare any entropy solution to a weakly converging sequence of viscous solutions, as depicted in subsection~\ref{Kato}, and subsection~\ref{Uniqueness of (Measure-Valued) Entropy Solution}, to arrive at the following equality
\begin{align*}
\E \bigg[\dint_{\Pi_T} |u(t,x) -v(t,x)|  \,dx\,dt \bigg]=0,
\end{align*}
where $u$ represents any entropy solution and $v$ is the limit entropy solution associated with the sequence $u_\eps$ of weak solutions of \eqref{eq:viscous-Brown}. The above equality confirms the uniqueness of the entropy solution.



\subsection{Stability of the Entropy Solution}
\label{imp}
First note that, by virtue of the above uniqueness result, any entropy solution $u$ is stemmed from the sequence of viscosity solutions $u_\eps$. Hence, thanks to the \textit{a priori} estimates (cf. Theorem~\ref{prop:vanishing viscosity-solution}), we conclude that $u \in L^2(\Omega\times(0,T),H^\lambda(\D))$. Thus, we can recast Kato's inequality as
\begin{align*}
0 \leq& \int_\D |u_0-v_0|\psi(0)\,dx + \E \bigg[\int_{\Pi_T} |u(t,x) -v(t,x)| \partial_t \psi(t,x) - F(u(t,x),v(t,x))\nabla \psi(t,x)\,dx\,dt \bigg]
\\& -\E \bigg[\int_{\Pi_T}   \mathcal{L}_{\lambda/2} \Big[|u(t,\cdot)-v(t,\cdot)| \Big](x) \, \mathcal{L}_{\lambda/2}[\psi (t,\cdot)](x) \,dx\,dt\bigg].
\end{align*}
At this point, we recall the test function $\psi_R$ given in \eqref{psi}. Note that, for a space mollifier $\rho$, 
$\psi_R \star \rho(x)\to 1$ for any $x \in \R^d$, so that $\psi_R \star \rho(x)-\psi_R \star \rho(y)\to 0$ for any $x,y \in \R^d$. 
Moreover, observe that
\begin{align*}
(a) \quad &|\psi_R\star\rho(x)-\psi_R\star\rho(y)| \leq  \int_{\D} \psi_R(z) |\int_y^x \rho^\prime(\sigma-z) d\sigma| dz \leq |x-y| \|\rho^\prime\|_{L^1} \text{(thanks to Fubini's theorem),}
\\
& \hspace{4cm} \text{that implies} \quad \frac{|\psi_R\star\rho(x)-\psi_R\star\rho(y)|^2}{|x-y|^{d+2\lambda}} \leq \frac{c}{|x-y|^{d+2\lambda-2}},
\\
(b) \quad&|\psi_R\star\rho(x)-\psi_R\star\rho(y)| \leq 2 \hspace{1cm} \text{that implies} \quad \frac{|\psi_R\star\rho(x)-\psi_R\star\rho(y)|^2}{|x-y|^{d+2\lambda}} \leq \frac{c}{|x-y|^{d+2\lambda}}.
\end{align*}
Thus, we have
\begin{align*}
\frac{|\psi_R\star\rho(x)-\psi_R\star\rho(y)|^2}{|x-y|^{d+2\lambda}} \leq \frac{c}{|x-y|^{d+2\lambda-2}}\chi_{|x-y|<1} + \frac{c}{|x-y|^{d+2\lambda}}\chi_{|x-y|>1},
\end{align*}
and we conclude that
$\psi_R\star\rho \to 0$ in $H^{\lambda}(\R^d)$ and $\mathcal{L}_{\lambda/2}(\psi_R\star\rho) \to 0$ in $L^2(\Omega\times(0,T),L^2(\D))$.

Then, for any non-negative $\theta \in \mathcal{D}([0,T))$,
\begin{align*}
0 \leq& \int_\D |u_0-v_0|\theta(0)\psi_R\star \rho\,dx  -\E \bigg[\int_{\Pi_T} \theta(t)  \mathcal{L}_{\lambda/2}(|u(t,x)-v(t,x)|)  \mathcal{L}_{\lambda/2}(\psi_R\star \rho)(x) \,dx\,dt \bigg]\\
& \qquad + \E \bigg[\int_{\Pi_T} \theta^\prime(t) |u(t,x) -v(t,x)| \psi_R\star \rho(x) - \theta(t) F(u(t,x),v(t,x))\nabla (\psi_R\star \rho) (x) \,dx\,dt \bigg]
\end{align*}
and, since $|\nabla (\psi_R\star \rho) (x)| \leq |\nabla \psi_R|\star \rho(x) \leq \frac{c}{R} \psi_R\star \rho (x)$, 
\begin{align*}
0 \leq& \int_\D |u_0-v_0|\theta(0)\psi_R\star \rho\,dx + \E \bigg[\int_{\Pi_T}  |u(t,x) -v(t,x)|  \psi_R\star \rho (x) \Big[ \theta^\prime(t)  +\frac{c}{R} \theta(t)\Big] \,dx\,dt \bigg]
\\&
-\E \bigg[\int_{\Pi_T} \theta(t)  \mathcal{L}_{\lambda/2}(|u(t,x)-v(t,x)|)  \mathcal{L}_{\lambda/2}(\psi_R\star \rho)(x) \,dx\,dt \bigg]
\end{align*}
Replacing $\theta(t)$ by $\theta(t)e^{-\frac{ct}{R}}$, one has that 
\begin{align*}
0 \leq& \int_\D |u_0-v_0|\theta(0)\psi_R\star \rho\,dx + \E \bigg[\int_{Q}  |u(t,x) -v(t,x)|  \psi_R\star \rho (x) \theta^\prime(t) e^{-\frac{ct}{R}} \,dx\,dt \bigg]
\\&
+\|\theta\|_\infty \| \mathcal{L}_{\lambda/2}(|u(t,x)-v(t,x)|)\|_{L^2(\Omega\times(0,T),L^2(\D))}  \|\mathcal{L}_{\lambda/2}(\psi_R\star \rho)\|_{L^2(\Omega\times(0,T),L^2(\D))}.
\end{align*}
Assume that $u_0-v_0 \in L^1(\D)$ and chose $\theta$ in such a way that it is a non-increasing function with $\theta(0)=1$, then one gets, passing to the limit when $R \to \infty$, 
\begin{align*}
0 \leq& \int_\D |u_0-v_0|\,dx + \E \bigg[\int_{\Pi_T}  |u(t,x) -v(t,x)|  \theta^\prime(t) \,dx\,dt \bigg].
\end{align*}
Thus for $t$ a.e. in $(0,T)$,
\begin{align*}
\E \bigg[\int_{\D}  |u(t,x) -v(t,x)| \,dx \bigg] \leq& \int_\D |u_0-v_0|\,dx.
\end{align*}
This essentially demonstrates the stability of entropy solution of Theorem~\ref{uniqueness_new} with respect to its initial data. 


\section{Proof of Theorem~\ref{continuous-dependence}: Continuous Dependence Estimates}
\label{cont-depen-estimate}

Note that, the average $L^1$-contraction principle (cf. Subsection~\ref{imp}) gives the continuous dependence on the initial data. However, we intend to establish continuous dependence on the fractional exponent $\lambda$, and
on the nonlinearities, i.e., on the flux function and the noise coefficients. 
To achieve that, we proceed as follows: 
For $\eps>0$, let $v_\eps$ be the weak solution to the problem 
\begin{align}
dv_\eps(s,y) - \eps \Delta v_\eps(s,y)\,ds + \mathcal{L}_{\kappa}[v_\eps(s, \cdot)](y)\,ds  & - \mbox{div}_y g(v_\eps(s,y)) \,ds  \label{eq:viscous} \\
& =\widetilde{\sigma}(v_\eps(s,y))\,dW(s) + \int_{E} \widetilde{\eta}(v_\eps(s,y);z)\,\widetilde{N}(dz,ds), \notag \\
v_\eps(0,y)&=v_0^{\eps}(y) \notag.
\end{align}
In view of Theorem \ref{thm:existence-bv}, we conclude that $v_\eps(s,y)$ converges to the unique BV-entropy solution $v(s, y)$ of \eqref{eq:stoc_con_brown} with initial data $v_0(y)$. Let $u(t,x)$ be the unique BV-entropy solution of  \eqref{eq:stoc_con_brown} with initial data $u_0(x)$. Moreover, we assume that Assumptions~\ref{A1}, \ref{A1'}, \ref{A3}, \ref{A31}, \ref{A4}, and \ref{A5} hold for both sets of given functions $(u_0, f, \sigma, \eta, \lambda)$ and $(v_0, g, \widetilde{\sigma}, \widetilde{\eta}, \kappa)$.

In what follows, we shall estimate the average $L^1$-difference between two entropy solutions $u$ and $v$.  To achieve this, we shall make use of the ``{\it doubling of variables}" technique. However, we can not directly compare two entropy solutions $u$ and $v$, but instead we first compare the entropy solution $u(t,x)$ with the solution of the viscous approximation \eqref{eq:viscous}, i.e., $v_{\eps}(s,y)$. This approach is somewhat different from the deterministic approach, where one can directly compare two entropy solutions. 

To improve the readability of the presentation, we make use of the following notation:
$$
\mathcal{L}_{\lambda}[\varphi](x)= c_{\lambda}\, \text{P.V.}\, \int_{\R^d} \big( \varphi(x) -\varphi(x+z)\big)\,d\mu_{\lambda}(z),  
$$
where $d\mu_{\lambda}(z):= \frac{dz}{|z|^{d + 2 \lambda}}$. Observe that $\mu_{\lambda}$ is a nonnegative Radon measure on $\D\setminus \{0\}$ satisfying
\begin{align}
\label{important}
\int_{\D\setminus \{0\}} \big(|z|^2 \wedge 1\big)\,d\mu_{\lambda}(z) < +\infty.
\end{align}

For technical purposes (see \cite{Alibaud_one}), we need to split the Radon measures $\mu_{\lambda}, \mu_{\kappa}$ as follows: Let $K^{\pm}$ be the sets such that
\begin{equation}
\label{one}
\begin{cases}
K^{\pm} \subseteq \R^d \setminus \{0\} \text{ are Borel sets. }\\
\cup_{\pm} K^{\pm} = \R^d \setminus \{0\}, \text{ and } \cap_{\pm} K^{\pm} = \emptyset.\\
\R^d \setminus \{0\} \setminus \mathrm{supp}(\mu_{\lambda} - \mu_{\kappa})^{\mp} \subseteq K^{\pm},
\end{cases}
\end{equation}
and we denote $\mu_{\lambda_{\pm}}$ and $\mu_{\kappa_{\pm}}$ as the restrictions of $\mu_{\lambda}$ and $\mu_{\kappa}$ to $K^{\pm}$. Then it is easy to see that
\begin{equation}
\label{two}
\begin{cases}
\mu_{\lambda} = \sum_{\pm}\mu_{\lambda_{\pm}},  \text{ and } \mu_{\kappa} = \sum_{\pm}\mu_{\kappa_{\pm}} \\
\pm (\mu_{\lambda_{\pm}} -\mu_{\kappa_{\pm}}) = (\mu_{\lambda} - \mu_{\kappa})^{\pm}.\\
\mu_{\lambda_{\pm}}, \mu_{\kappa_{\pm}}, \text{ and } \pm (\mu_{\lambda_{\pm}} - \mu_{\kappa_{\pm}}) \text{ all are nonnegative Radon measures satisfying \eqref{important}. }
\end{cases}
\end{equation}

Next, for a nonnegative test function $\psi\in C_c^{1,2}([0,\infty)\times \rd)$, and two positive constants $\delta, \delta_0 $, we define the same test function as in \eqref{test_function}          
\begin{align*}
\varphi_{\delta,\delta_0}(t,x, s,y) = \rho_{\delta_0}(t-s) \varrho_{\delta}(x-y) \psi(t,x). 
\end{align*} 
We now write the entropy inequality for $u(t,x)$, based on the 
entropy pair $(\beta(\cdot-k), f^\beta(\cdot, k))$, and 
then multiply by $\varsigma_l(v_\eps(s,y)-k)$, integrate with 
respect to $ s, y, k$ and take the expectation. The result is
\begin{align}
0\le  & \E \Big[\int_{\Pi_T}\int_{\R^d}\int_{\R} \beta(u(0,x)-k)
\varphi_{\delta,\delta_0}(0,x,s,y) \varsigma_l(v_\eps(s,y)-k)\,dk \,dx\,dy\,ds\Big] \notag \\
 &\qquad + \E \Big[\int_{\Pi_T} \int_{\Pi_T} \int_{\R} \beta(u(t,x)-k)\partial_t \varphi_{\delta,\delta_0}(t,x,s,y)
\varsigma_l(v_\eps(s,y)-k)\,dk \,dx\,dt\,dy\,ds \Big]\notag \\ 
& + \E \Big[\sum_{k\ge 1}\int_{\Pi_T} \int_{\Pi_T} 
\int_{\R}  g_k(u(t,x)) \, \beta^\prime (u(t,x)-k)\, \varphi_{\delta,\delta_0}\, dx \,d\beta_k(t) \varsigma_l(u_\eps(s,y)-k)\,dk\,dy\,ds \Big] \notag \\
 &+ \frac{1}{2}\, \E \Big[ \int_{\Pi_T} \int_{\Pi_T}  
\int_{\R} \mathbb{G}^2(u(t,x))\, \beta^{\prime\prime} (u(t,x) -k)\, \varphi_{\delta,\delta_0} \,dx\,dt 
\, \varsigma_l(u_\eps(s,y)-k)\,dk\,dy\,ds \Big] \notag \\
 & \qquad +  \E \Big[ \int_{\Pi_T} \int_{\R}\int_{\Pi_T}\int_{|z|>0}\Big(\beta \big(u(t,x) +\eta(u(t,x);z)-k\big)-\beta(u(t,x)-k)\Big) \notag \\
& \hspace{6cm} \times \varphi_{\delta,\delta_0}(t,x,s,y)\,\varsigma_l(v_\eps(s,y)-k) \,\tilde{N}(dz,dt) \,dx \,dk \,dy\,ds \Big] \notag\\
&\qquad +  \E \Big[\int_{\Pi_T} \int_{t=0}^T\int_{|z|>0}\int_{\R^d} 
\int_{\R} \Big(\beta \big(u(t,x) +\eta(u(t,x);z)-k\big)-\beta(u(t,x)-k) \notag \\
 & \hspace{6.5cm}-\eta(u(t,x);z) \beta^{\prime}(u(t,x)-k)\Big)
 \varphi_{\delta,\delta_0}(t,x;s,y) \notag \\
&\hspace{8cm}\times \varsigma_l(v_\eps(s,y)-k)\,dk\,dx\,\nu(dz)\,dt\,dy\,ds\Big]\notag \\
& \qquad +  \E \Big[\int_{\Pi_T}\int_{\Pi_T} \int_{\R} 
 f^\beta(u(t,x),k) \cdot \grad_x \varphi_{\delta,\delta_0}(t,x,s,y) \, \varsigma_l(v_\eps(s,y)-k)\,dk\,dx\,dt\,dy\,ds\Big] \notag \\
 & - \E \Big[\int_{\Pi_T} \int_{\Pi_T} 
\int_{\R} \mathcal{L}^r_{\lambda}[u(t,\cdot)](x)\, \varphi_{\delta,\delta_0}(t,x,s,y)\, \beta'(u(t,x) -k) \,dx\,dt 
\, \varsigma_l(u_\eps(s,y)-k)\,dk\,dy\,ds \Big] \notag \\
& -  \E \Big[\int_{\Pi_T} \int_{\Pi_T} 
\int_{\R} \beta(u(t,x) -k ) \,\mathcal{L}_{\lambda,r} [\varphi_{\delta,\delta_0}(t,\cdot,s,y)](x) \,dx\,dt
\, \varsigma_l(u_\eps(s,y)-k)\,dk\,dy\,ds \Big]\notag \\[2mm]
& =:  I_1 + I_2 + I_3 +I_4 + I_5 + I_6 + I_7 + I_8 + I_9. \label{stochas_entropy_1-levy}
\end{align}
 
We now apply the It\^{o}-L\'{e}vy formula to \eqref{eq:viscous} and multiply with test 
function $\varphi_{\delta, \delta_0}$ and $\varsigma_l(u(t,x)-k)$ and integrate. The result is
\begin{align}
 0\le  &\, \E \Big[\int_{\Pi_T}\int_{\R^d}\int_{\R} 
 \beta(v_\eps(0,y)-k)\varphi_{\delta,\delta_0}(t,x,0,y) \varsigma_l(u(t,x)-k)\,dk \,dx\,dy\,dt\Big] \notag \\
   & \qquad \qquad \qquad \quad +   \E \Big[\int_{\Pi_T} \int_{\Pi_T} \int_{\R} 
 \beta(v_\eps(s,y)-k)\partial_s \varphi_{\delta,\delta_0}(t,x,s,y)
 \varsigma_l(u(t,x)-k)\,dk \,dy\,ds\,dx\,dt\Big] \notag \\ 
 & + \E \Big[\sum_{k\ge 1}\int_{\Pi_T} \int_{\Pi_T}
\int_{\R}  \widetilde{g}_k(v_\eps(s,y))\,\beta^\prime (v_\eps(s,y)-k)\, \varphi_{\delta,\delta_0}\, dx \,d\beta_k(t) \varsigma_l(u(t,x)-k)\,dk\,dy\,ds \Big] \notag \\
 &+ \frac{1}{2}\, \E \Big[ \int_{\Pi_T} \int_{\Pi_T} 
\int_{\R} \mathbb{\widetilde{G}}^2(v_\eps(s,y)) \,\beta^{\prime\prime} (v_\eps(s,y) -k)\, \varphi_{\delta,\delta_0} \,dx\,dt 
\, \varsigma_l(u(t,x)-k)\,dk\,dy\,ds \Big] \notag \\
  + &  \E \Big[\int_{\Pi_T} \int_{\Pi_T}\int_{|z|>0} \int_{\R} 
 \Big(\beta \big(v_\eps(s,y) +\widetilde{\eta}(v_\eps(s,y);z)-k\big)
 -\beta(v_\eps(s,y)-k)\Big) \notag \\
 & \hspace{6.5cm} \times \varphi_{\delta,\delta_0}(t,x,s,y)\varsigma_l(u(t,x)-k)\,dk \,\tilde{N}(dz,ds)\,dy\,dx\,dt \Big]\notag\\
  + &  \E \Big[\int_{\Pi_T} \int_{s=0}^T\int_{|z|>0}\int_{\R^d} 
 \int_{\R} \Big(\beta \big(v_\eps(s,y) +\widetilde{\eta}(v_\eps(s,y);z)-k\big)
 -\beta(v_\eps(s,y)-k) \notag \\
  & \hspace{6.0cm}-\widetilde{\eta}(v_\eps(s,y);z) \beta^{\prime}(v_\eps(s,y)-k)\Big)
  \varphi_{\delta,\delta_0}(t,x;s,y) \notag \\
 &\hspace{8.7cm}\times \varsigma_l(u(t,x)-k)\,dk\,dy\,\nu(dz)\,ds\,dx\,dt\Big]\notag \\
  & \quad +  \E\Big[\int_{\Pi_T}\int_{\Pi_T} \int_{\R}  
 g^\beta(v_\eps(s,y),k)\cdot \grad_y \varphi_{\delta,\delta_0}(t,x;s,y)\, \varsigma_l(u(t,x)-k)\,dk\,dx\,dt\,dy\,ds\Big] \notag \\
 & \quad  - \eps  \,\E \Big[\int_{\Pi_T} \int_{\Pi_T} \int_{\R} 
 \beta^\prime(v_\eps(s,y)-k)\grad_y v_\eps(s,y) \cdot \grad_y  \varphi_{\delta,\delta_0}
  \varsigma_l(u(t,x)-k)\,dk \,dy\,ds\,dx\,dt\Big]  \notag \\
  & - \E \bigg[\int_{\Pi_T} \int_{\Pi_T} 
\int_{\R} \mathcal{L}^r_{\kappa}[v_\eps(s, \cdot)](y) \, \varphi_{\delta,\delta_0}(t,x,s,y)\, \beta^\prime (v_\eps(s,y)-k) \,dx\,dt \, \varsigma_l(u(t,x)-k)\,dk\,dy\,ds \bigg]  \notag \\
& -\E \bigg[\int_{\Pi_T} \int_{\Pi_T} 
\int_{\R} \beta( v_\eps(s,y)-k)\mathcal{L}_{\kappa, r}[\varphi_{\delta,\delta_0}(t,x,s,.)](y) \, \varsigma_l(u(t,x)-k)\,dk\,dy\,ds\,dx\,dt \bigg] \notag \\[2mm]
& =: J_1 + J_2 + J_3 + J_4 + J_5 + J_6 + J_7 + J_8 + J_9 + J_{10}. \label{stochas_entropy_3-levy}
\end{align}

Our aim is to add \eqref{stochas_entropy_1-levy} and \eqref{stochas_entropy_3-levy}, 
and pass to the  limits with respect to the various parameters involved. We do this by claiming
a series of lemma's and proofs of these lemmas follow from \cite{BaVaWit_2014,BisMajKarl_2014} modulo cosmetic changes. 
 
To begin with, note that the particular choice of the test function \eqref{test_function} implies that $I_1=0$. 

\begin{lem}
\label{stochastic_lemma_1}
It holds that 
\begin{align}
I_1 + J_1 & \underset{\delta_0 \goto 0} \longrightarrow \E \Big[\int_{\R^d}\int_{\R^d}\int_{\R} 
  \beta(u(0,x)-k) \psi(0,x)\varrho_{\delta} (x-y) \varsigma_l(v_\eps(0,y)-k)\,dk\,dx\,dy \Big]\notag\\
  & \underset{l \goto 0} \longrightarrow  \E \Big[\int_{\R^d}\int_{\R^d}
  \beta(u(0,x)-v_\eps(0,y)) \psi(0,x)\varrho_{\delta} (x-y)\,dx\,dy\Big]:= \mathcal{A}_1,\notag \\
  &\hspace{-2cm} \mathcal{A}_1 \le  \E \Big[\int_{\R^d}\int_{ \R^d}\big| v_\eps(0,y) -u(0,x)\big| \psi(0,x)\,\varrho_\delta(x-y)
  \,dx\, dy\Big].\notag
 \end{align}
\end{lem}
 
\begin{lem}\label{stochastic_lemma_2}
It holds that
\begin{align*}
I_2 + J_2  &\underset{\delta_0 \goto 0}\longrightarrow   \E\Big[ \int_{\Pi_T}\int_{\R^d}\int_{\R}
    \beta(v_\eps(t,y)-k) \partial_t \psi(t,x)\varrho_\delta(x-y)\varsigma_l(u(t,x)-k)\,dk\,dy\,dx\,dt\Big]\\
&\underset{l \goto 0}\longrightarrow  \E \Big[\int_{\Pi_T}\int_{\R^d} \beta(v_\eps(t,y)-u(t,x)) \partial_t \psi(t,x)
 \, \varrho_\delta(x-y)\,dy\,dx\,dt\Big].
\end{align*}
\end{lem}
\noindent Next we consider stochastic terms. Regarding that we have the following result
\begin{lem}
\label{lem:stochastic-terms_01}
We have $J_3=0 =J_5$ and The following hold:
\begin{align*}
 	\mathcal{A}_2:= &\lim_{l\goto 0}\lim_{\delta_0\goto 0}\Big( \big(I_3 + J_3\big) + \big(I_4 + J_4 \big)\Big) \notag \\
 	 &= \frac{1}{2} \E\Big[\sum_{k\ge 1}\int_{\Pi_T}\int_{\R^d} \beta^{\prime\prime}
 	 \big(u(t,x)-v_{\eps}(t,y)\big) \big(g_k(u(t,x))-\widetilde{g}_k(v_\eps(t,y))\big)^2 
 	 \psi(t,x)\varrho_{\delta}(x-y)\,dx\,dy\,dt \Big], \\
 	 \mathcal{A}_3:=&\lim_{l\goto 0}\lim_{\delta_0\goto 0}\Big(I_5 + J_5 + I_6 + J_{6}\Big) \notag \\
 	 &= \E\Big[\int_{\Pi_T}\int_{\R^d} \int_{E} \int_0^1 (1-\lambda) \beta^{\prime\prime}
 	 \Big(u(t,x)-v_{\eps}(t,y) + \lambda \big( \eta(u(t,x);z)- \widetilde{\eta}(v_\eps(t,y);z)\big)\Big) \notag \\
 	&  \hspace{4cm} \times \big( \eta(u(t,x);z)-\widetilde{\eta}(v_\eps(t,y);z)\big)^2
 	 \psi(t,x) \varrho_{\delta}(x-y) \,d\lambda\,\nu(dz)\,dx\,dy\,dt \Big].
\end{align*}
\end{lem}
We make use of the following lemma.
\begin{lem}
The following hold:
\begin{align*}
\mathcal{A}_2 & \le 2 \E \Big[\int_{\Pi_T}\int_{\R^d} \beta \big(u(t,x)-v_\eps(t,y)\big) \psi(t,x)\varrho_{\delta}(x-y)\,dx\,dy\,dt \Big] + \frac{\underset{k\ge 1}\sum \mathcal{E}_k(\sigma, \widetilde{\sigma})^2}{\xi}  \int_{0}^T ||\psi(t,\cdot)||_{L^\infty(\R^d)}\,dt  \\
\mathcal{A}_3 & \le  C \E\Big[\int_{\Pi_T}\int_{\R^d} 
 \beta\big(u(t,x)-v_\eps(t,y)\big)
   \psi(t,x) \varrho_{\delta}(x-y)\,dx\,dy\,dt \Big] 
   \\& \hspace{6cm}+ C\Big(\sqrt{\mathcal{D}(\eta, \widetilde{\eta})} 
  + \frac{\mathcal{D}(\eta,\widetilde{\eta}) }{\xi}\Big)  \int_0^T\|\psi(t,\cdot)\|_{L^\infty(\R^d)}\,dt,
\end{align*}
\end{lem}
where the first estimate comes from Assumptions \ref{A4} and \eqref{eq:approx to abosx} and their consequences and the second on Assumptions \ref{A5}, \eqref{eq:approx to abosx} and by arguments close to the ones developed in the first step of the proof of Theorem~\ref{thm:bv-viscous}. Technical details are given in \cite[(4.13) p.170 - (4.20) p.172]{Koley2}.

\noindent For the terms coming from the flux functions, following the arguments of the proof of \cite[Lemma 4.7]{Koley2}, we have the following lemma.
\begin{lem}\label{stochastic_lemma_4}
The following hold:
\begin{align*}
\lim_{l\goto 0}\lim_{\delta_0 \goto 0} (I_7+J_7) 
&\le \E\Big[\int_{t=0}^T \int_{\R^d}\int_{\R^d} f^\beta(u(t,x),v_\eps(t,y))\cdot \grad_x \psi(t,x) \,\varrho_\delta(x-y) \,dx\,dy\,dt\Big] \\
&+ \E\Big[|u_0|_{BV(\R^d)}\Big] \Big( M_2\,\xi\, ||f^{\prime\prime}||_{\infty}
  + ||f^\prime-g^\prime||_{\infty}\Big) \int_{t=0}^T ||\psi(t,\cdot)||_{L^\infty(\R^d)}\,dt. \\
|J_8| & \le C \frac{\eps}{\delta} \E\big[ |v_0|_{BV(\R^d)}\big].
\end{align*} 
\end{lem} 
\noindent Finally, we are left with fractional terms. To deal with these terms, we follow closely the uniqueness proof in Section~\ref{uniqueness}. In particular, following 
Step 2 of Lemma~\ref{fractional_lemma_two_1}, we conclude
\begin{lem}
\label{fractional_lemma_4}
The following hold:
\begin{align*}
\lim_{l\goto 0}\lim_{\delta_0\goto 0} & \big(I_9+J_{10} \big)
= -  \E \bigg[\int_{\R^d} \int_{\Pi_T}  \beta(u(t,x) -v_\eps(t,y) ) \,\mathcal{L}_{\lambda,r} \big[ \psi(t, \cdot)\varrho_{\delta} (\cdot -y)\big](x) \,dx\,dt \,dy \bigg] \notag\\ 
& - \E \bigg[\int_{\R^d} \int_{\Pi_T} \beta(v_\eps(t,y) - u(t,x)) \,\mathcal{L}_{\kappa,r} \big[ \varrho_{\delta} (x -\cdot)\big](y)\,\psi(t,x) \,dx\,dt \,dy \bigg]
 \underset{r \goto 0} \longrightarrow 0.
\end{align*} 
\end{lem} 
Finally, we are left with the last two terms. To deal with those terms, we make use of the Lemma~\ref{fractional_lemma_1} to conclude
\begin{align*}
\mathrm{M}:= & \lim_{l\to 0}\,\lim_{\delta_0 \to 0}  \big(I_8+J_9\big)\\ 
&= - \E \Big[\int_{\R^d} \int_{\Pi_T} \fr^r[u(t,\cdot)](x)\,\psi(x,t) \varrho_{\delta}(x-y)\, \beta'(u(t,x)-v_{\eps}(t,y)) \,dy\,dx\,dt \Big] \\
&\qquad - \E \Big[\int_{\R^d} \int_{\Pi_T} \mathcal{L}^r_{\kappa} [v_\eps(t,\cdot)](y)\,\psi(x,t) \varrho_{\delta}(x-y)\, \beta'(v_{\eps}(t,y)-u(t,x))  \,dy\,dx\,dt \Big] \\
& = \E \Big[\int_{\R^d} \int_{\Pi_T} \Big[ \int_{|z|>r} (u(t,x+z)- u(x)) \,d\mu_{\lambda}(z) - (v_\eps(t,y+z) - v_\eps(y)) \,d\mu_{\kappa}(z) \Big]\\
&\hspace{6cm} \times \psi(x,t) \varrho_{\delta}(x-y)\, \beta'(u(t,x)-v_{\eps}(t,y)) \,dy\,dx\,dt \Big]
\end{align*}
In order to proceed, we first state the following lemmas. 
\begin{lem}
\label{lemma_01}
The following holds:
\begin{align*}
&\E \Big[\int_{\R^d} \int_{\Pi_T} \Big[ \int_{|z|>r} (u(t,x+z)- u(x)) \,d\mu_{\lambda}(z) - (v_\eps(t,y+z) - v_\eps(y)) \,d\mu_{\lambda}(z) \Big]\\
&\hspace{6cm} \times \psi(x,t) \varrho_{\delta}(x-y)\, \beta' (u(t,x)-v_\eps(t,y)) \,dy\,dx\,dt \Big] \\
& \qquad \qquad \le - \E \Big[\int_{\R^d} \int_{\Pi_T} \beta (u(t,x)-v_\eps(t,y))  \fr^r[\psi(t,\cdot)](x) \, \varrho_{\delta}(x-y)\,dy\,dx\,dt \Big].
\end{align*} 
Note that the same inequality is satisfied with $\lambda_{\pm}$ in place of $\lambda$.
\end{lem} 
\begin{proof}
The proof of the above lemma is an easy adaptation of the calculations presented in Step 3, Lemma~\ref{fractional_lemma_1}. We leave the details to the reader.
\end{proof}
\begin{lem}
\label{lemma_02}
For any $k \in \R$, The following holds:
\begin{align*}
 \int_{\R^d} \beta'(k - u(t,x))\, \fr^r[u(t,\cdot)](x)\, \psi(t,x) \,dx  \le - \int_{\R^d} \beta(k-u(t,x)) \,\fr^r[\psi(t,\cdot)](x)\,dx
\end{align*} 
\end{lem} 
\begin{proof}
We can apply the same kind of trick: convexity inequalities and change of variables, as in Step 3, Lemma~\ref{fractional_lemma_1}, to conclude the proof.
\end{proof}
Now making use of Lemma~\ref{lemma_01} and notations \eqref{one} and \eqref{two}, we can rewrite 
\begin{align*}
\mathrm{M}& = \sum_{\pm} \E \Big[\int_{\R^d} \int_{\Pi_T} \Big( \mathcal{L}^r_{\kappa_{\pm}} [v_\eps(t,\cdot)](y) - \mathcal{L}^r_{\lambda_{\pm}} [u(t,\cdot)](x) \Big)\,\psi(x,t) \varrho_{\delta}(x-y)\, 
\beta'(u(t,x)-v_{\eps}(t,y)) \,dy\,dx\,dt \Big] \\
& \quad \le \E \Big[\int_{\R^d} \int_{\Pi_T} \Big( \mathcal{L}^r_{\kappa_{+}} [u(t,\cdot)](x) - \mathcal{L}^r_{\lambda_{+}} [u(t,\cdot)](x) \Big)\,\psi(x,t) \varrho_{\delta}(x-y)\, 
\beta'(u(t,x)-v_{\eps}(t,y)) \,dy\,dx\,dt \Big] \\
& \quad - \E \Big[\int_{\R^d} \int_{\Pi_T} \beta(u(t,x)-v_\eps(t,y)) \mathcal{L}^r_{\kappa_{+}} [\psi(t,\cdot)](x) \, \varrho_{\delta}(x-y)\,dy\,dx\,dt \Big] \\
& \quad + \E \Big[\int_{\R^d} \int_{\Pi_T} \Big( \mathcal{L}^r_{\kappa_{-}} [v_\eps(t,\cdot)](y) - \mathcal{L}^r_{\lambda_{-}} [v_\eps(t,\cdot)](y) \Big)\,\psi(x,t) \varrho_{\delta}(x-y)\, 
\beta'(u(t,x)-v_{\eps}(t,y)) \,dy\,dx\,dt \Big] \\
& \quad - \E \Big[\int_{\R^d} \int_{\Pi_T} \beta(u(t,x)-v_\eps(t,y))  \mathcal{L}^r_{\lambda_{-}} [\psi(t,\cdot)](x) \, \varrho_{\delta}(x-y)\,dy\,dx\,dt \Big] \\
& \quad = \E \Big[\int_{\R^d} \int_{\Pi_T} \big(\mathcal{L}^r_{\lambda_{+}} -\mathcal{L}^r_{\kappa_{+}} \big)[u(t,\cdot)](x)\,\psi(t,x) \varrho_{\delta}(x-y)\, \beta'(v_{\eps}(t,y) -u(t,x)) \,dy\,dx\,dt \Big] \\
& \quad + \E \Big[\int_{\R^d} \int_{\Pi_T} \big(-\mathcal{L}^r_{\lambda_{-}} +\mathcal{L}^r_{\kappa_{-}} \big)[v_\eps(t,\cdot)](y)\,\psi(x,t) \varrho_{\delta}(x-y)\, 
\beta'(u(t,x)-v_{\eps}(t,y)) \,dy\,dx\,dt \Big] \\
& \quad - \E \Big[\int_{\R^d} \int_{\Pi_T} \beta(u(t,x)-v_\eps(t,y)) \mathcal{L}^r_{\kappa_{+}} [\psi(t,\cdot)](x) \, \varrho_{\delta}(x-y)\,dy\,dx\,dt \Big] \\
& \quad - \E \Big[\int_{\R^d} \int_{\Pi_T} \beta(u(t,x)-v_\eps(t,y))  \mathcal{L}^r_{\lambda_{-}} [\psi(t,\cdot)](x) \, \varrho_{\delta}(x-y)\,dy\,dx\,dt \Big] \\[2mm]
&\quad := \mathrm{M}_1 + \mathrm{M}_2 + \mathrm{M}_3 + \mathrm{M}_4.
\end{align*}


Consider in the sequel $r_1 >r$. Using a proof similar to the one of Lemma~\ref{lemma_02}, we have
\begin{align*}
\mathrm{M}_1 & \le -\E \Big[\int_{\R^d} \int_{\Pi_T} \beta(v_{\eps}(t,y) -u(t,x))
\Big(\mathcal{L}^{r,1}_{\lambda_{+},r_1}-\mathcal{L}^{r,1}_{\kappa_{+},r_1}\Big) \big[\psi(\cdot,t) \varrho_{\delta}(\cdot-y) \big](x)\,dy\,dx\,dt \Big] \\
& + \E \Big[\int_{\R^d} \int_{\Pi_T} \big(\mathcal{L}^{r_1}_{\lambda_{+}}-\mathcal{L}^{r_1}_{\kappa_{+}}\big) [u(t,\cdot)](x)\,\psi(x,t) \varrho_{\delta}(x-y)\, \beta'(v_{\eps}(t,y) -u(t,x)) \,dy\,dx\,dt \Big] 
\end{align*}
where the notation $\mathcal{L}^{r,1}_{\lambda_{+},r_1}$ means that the nonlocal integration is understood in the set $\{r <|z|\leq r_1\}$ (resp. with $\kappa_{+}$).
To estimate the first term of the above inequality, we note that, by construction of the measures, the nonlocal domain of integration is always radial symmetric, the arguments of Appendix \ref{FractionalLaplace} hold and
\begin{align*}
& \E \Big[\int_{\R^d} \int_{\Pi_T}  \beta(v_\eps(t,y) - u(t,x)) \big(\mathcal{L}^{r,1}_{\lambda_{+}}-\mathcal{L}^{r,1}_{\kappa_{+}}\big) \big[ \underbrace{\psi(\cdot,t) \varrho_{\delta}(\cdot-y)}_{:=\Psi(\cdot,y,t)} \big](x) \,dx\,dt \,dy \Big] \\
& = \E \bigg[\int_0^1 \int_{\R^d} \int_{\Pi_T} \int_{r<|z|\le r_1} (1-\tau)\beta(v_\eps(t,y) - u(t,x)) \,           z^T.\text{Hess}_x \Psi(x+\tau z,y,t).z \,d(\lambda_{+}-\kappa_{+})(z)\,dx\,dt \,dy \,\d\tau \bigg]  \\
& \le \E \bigg[ \int_0^1  \int_{\Pi_T} \int_{r<|z|\le r_1}\int_{\R^d} \big|\nabla_x \Psi(x-\tau z,y,t).z\big|\, d\big|\nabla_x[\beta(v_\eps(t,\cdot) - u(t,x))].z\big|(x) \,d(\lambda_{+}-\kappa_{+})(z)\,dt \,dy \,\d\tau \bigg]  \\
& \le \E \bigg[ \int_0^1  \int_{\Pi_T} \int_{r<|z|\le r_1}\int_{\R^d} |z|^2\, \big| \nabla_x \Psi(x-\tau z,y,t) \big|\, d\big(|D u(t,\cdot)|\big)(x) \,d(\lambda_{+}-\kappa_{+})(z)\,dt \,dy \,\d\tau \bigg] \\
& \le |u_0|_{\mathrm{BV}} \int_0^T \Big( \norm{\nabla\psi(t, \cdot)}_{\infty} + \frac{C}{\delta} \norm{\psi(t,\cdot)}_{\infty} \Big)\, \int_{r<|z|\le r_1} |z|^{2} \,d(\lambda_{+}-\kappa_{+})(z) \,dt
\\
&\le |u_0|_{\mathrm{BV}} \int_0^T \Big( \norm{\nabla\psi(t, \cdot)}_{\infty} + \frac{C}{\delta} \norm{\psi(t,\cdot)}_{\infty} \Big)\, \int_{|z|\le r_1} |z|^{2} \,d(\lambda_{+}-\kappa_{+})(z) \,dt, 
\end{align*}
where we have used the fact that for any Lipschitz continuous function $\beta$ (with Lipschitz constant $1$), $|D\beta (u)| \le |Du|$. On the other hand, to handle the other term we proceed as follows:
\begin{align*}
& \E \Big[\int_{\R^d} \int_{\Pi_T} \big(\mathcal{L}^{r_1}_{\lambda_{+}}-\mathcal{L}^{r_1}_{\kappa_{+}}\big) [u(t,\cdot)](x)\, \psi(t,x) \varrho_{\delta}(x-y)\, \beta'(v_{\eps}(t,y) -u(t,x)) \,dy\,dx\,dt \Big] \\
& \qquad \qquad  \le C \int_0^T \norm{\psi(t,\cdot)}_{\infty}\, \int_{|z|> r_1} \norm{u_0(\cdot+z) -u_0}_{L^1(\R^d)} \,d(\lambda_{+}-\kappa_{+})(z) \,dt,
\end{align*}
thanks to the stability result of Theorem~\ref{uniqueness_new} and the fact that $u(\cdot,\cdot+z)$ is the solution associated with the initial condition $u_0(\cdot+z)$.

Observe that, exact same calculations will help us to estimate $\mathrm{M}_2$. Indeed, we have
\begin{align*}
\mathrm{M}_2 &  \le |v_0|_{\mathrm{BV}} \int_0^T \Big( \norm{\nabla\psi(t, \cdot)}_{\infty} + \frac{C}{\delta} \norm{\psi(t,\cdot)}_{\infty} \Big)\, \int_{|z|\le r_1} |z|^{2} \,d(\kappa_{-}-\lambda_{-})(z) \,dt \\
& \qquad \qquad + C \int_0^T \norm{\psi(t,\cdot)}_{\infty}\, \int_{|z|> r_1} \norm{v_0(\cdot+z) -v_0}_{L^1(\R^d)} \,d(\kappa_{-}-\lambda_{-})(z) \,dt,
\end{align*}
by assuming \textit{e.g.} that $v^\eps_0$, the regularisation of $v_0$, is obtained by convolution.

Consider a new parameter $r_2$ such that $0<r<r_2<r_1$. Then 
\begin{align*}
M_3=& -\E \Big[\int_{\R^d} \int_{\Pi_T} \int_{K^+\cap\{|z|>r\}} \frac{\psi(t,x)-\psi(t,x+z)}{|z|^{d+2\kappa}} \, dz\beta(u(t,x)-v_\eps(t,y)) \, \varrho_{\delta}(x-y)\,dy\,dx\,dt \Big]
\\
= &
- \E \Big[\int_{\R^d} \int_{\Pi_T} \beta(u(t,x)-v_\eps(t,y)) \mathcal{L}^{r_2}_{\kappa_{+}} [\psi(t,\cdot)](x) \, \varrho_{\delta}(x-y)\,dy\,dx\,dt \Big]
\\
&- \E \Big[\int_{\R^d} \int_{\Pi_T} \int_{K^+\cap\{r<|z|\leq r_2\}} \frac{\psi(t,x)-\psi(t,x+z)-\nabla \psi(t,x).z}{|z|^{d+2\kappa}} \, dz\beta(u(t,x)-v_\eps(t,y)) \, \varrho_{\delta}(x-y)\,dy\,dx\,dt \Big]
\\
\leq &
- \E \Big[\int_{\R^d} \int_{\Pi_T} \beta(u(t,x)-v_\eps(t,y)) \mathcal{L}^{r_2}_{\kappa_{+}} [\psi(t,\cdot)](x) \, \varrho_{\delta}(x-y)\,dy\,dx\,dt \Big]
\\
&+ \E \Big[\int_{\R^d} \int_{\Pi_T} \|D_x^2 \psi(t,\cdot)\|_\infty \int_{|z|\leq r_2} \frac{|z|^2}{|z|^{d+2\kappa}} \, dz\beta(u(t,x)-v_\eps(t,y)) \, \varrho_{\delta}(x-y)\,dy\,dx\,dt \Big]
\\
\leq &
- \E \Big[\int_{\R^d} \int_{\Pi_T} \beta(u(t,x)-v_\eps(t,y)) \mathcal{L}^{r_2}_{\kappa_{+}} [\psi(t,\cdot)](x) \, \varrho_{\delta}(x-y)\,dy\,dx\,dt \Big]
\\
&+ C_{\kappa}r_2^{2(1-\kappa)} \E \Big[\int_{\R^d} \int_{\Pi_T} \|D_x^2 \psi(t,\cdot)\|_\infty  \beta(u(t,x)-v_\eps(t,y)) \, \varrho_{\delta}(x-y)\,dy\,dx\,dt \Big],
\end{align*}
with, similarly 
\begin{align*}
M_4=&- \E \Big[\int_{\R^d} \int_{\Pi_T} \beta(u(t,x)-v_\eps(t,y))  \mathcal{L}^r_{\lambda_{-}} [\psi(t,\cdot)](x) \, \varrho_{\delta}(x-y)\,dy\,dx\,dt \Big]
\\
\leq &
- \E \Big[\int_{\R^d} \int_{\Pi_T} \beta(u(t,x)-v_\eps(t,y)) \mathcal{L}^{r_2}_{\lambda_{-}} [\psi(t,\cdot)](x) \, \varrho_{\delta}(x-y)\,dy\,dx\,dt \Big]
\\
&+ C_{\lambda}r_2^{2(1-\lambda)}\E \Big[\int_{\R^d} \int_{\Pi_T} \|D_x^2 \psi(t,\cdot)\|_\infty \beta(u(t,x)-v_\eps(t,y)) \, \varrho_{\delta}(x-y)\,dy\,dx\,dt \Big].
\end{align*}

We are now in a position to add \eqref{stochas_entropy_1-levy} and \eqref{stochas_entropy_3-levy}
and pass to the limits in $l$, $\delta_0$, $\eps$ and $r$. In what follows, invoking the above estimates and keeping in mind that $\{v_\eps\}_{\eps>0}$ converges in $L^p_{\mathrm{loc}} (\R^d; L^p((0,T) \times \Omega))$, for any $p \in [1,2)$, to the unique BV entropy solution $v$ of \eqref{eq:stoc_con_brown-1} with initial data $v_0$, we have 
\begin{align}
 0 & \le   \E \Big[\int_{\R^d} \int_{\R^d} \big|u_0(x)-v_0(y)\big| \psi(0,x)\varrho_{\delta} (x-y)\,dx\,dy\Big]  \label{stoc_entropy_4}  
 \\& + 
 \E \Big[\int_{\Pi_T}\int_{\R^d} \beta \big(u(t,x)-v(t,y)\big) \partial_t\psi(t,x)\varrho_\delta(x-y)\,dy\,dx\,dt\Big]\notag 
 \\& + 
 C\,\E \Big[\int_{\Pi_T}\int_{\R^d} \beta \big(u(t,x)-v(t,y)\big) \psi(t,x)\varrho_\delta(x-y)\,dy\,dx\,dt\Big]\notag 
 \\&
 + C \,\Bigg(\frac{\mathcal{E}(\sigma, \widetilde{\sigma})^2}{\xi} + C(|u_0|_{\mathrm{BV}})(\xi + ||f^\prime-g^\prime||_{\infty}) + \sqrt{\mathcal{D}(\eta, \widetilde{\eta})} + \frac{\mathcal{D}(\eta, \widetilde{\eta})}{\xi}  \Bigg) \int_{0}^T ||\psi(t,\cdot)||_{L^\infty(\R^d)}\,dt  \notag
\\&-\E\Big[\int_{\Pi_T}\int_{\R^d} f^{\beta} \big(u(t,x),v(t,y)\big) \cdot \grad_x \psi(t,x)\varrho_{\delta}(x-y)\,dx\,dy\,dt \Big]   \notag \\
& - \E \Big[\int_{\R^d} \int_{\Pi_T}  
\beta \big(u(t,x) - v(t,y)\big) \, (\mathcal{L}^{r_2}_{\kappa_{+}}+\mathcal{L}^{r_2}_{\lambda_{-}})[\psi(t,\cdot)](x)\, \varrho_{\delta} (x-y) \,dx\,dt \,dy \Big] \notag \\
&+ C_{\kappa,\lambda}[r_2^{2(1-\kappa)}+r_2^{2(1-\lambda)}] \E \Big[\int_{\R^d} \int_{\Pi_T}  \|D_x^2 \psi(t,\cdot)\|_\infty\beta(u(t,x)-v(t,y)) \, \varrho_{\delta}(x-y)\,dy\,dx\,dt \Big]
 \notag \\
& +\big[|u_0|_{\mathrm{BV}}+|v_0|_{\mathrm{BV}}\big] \int_0^T \Big( \norm{\nabla\psi(t, \cdot)}_{\infty} + \frac{C}{\delta} \norm{\psi(t,\cdot)}_{\infty} \Big)\, \int_{|z|\le r_1} |z|^{2} \,d |\mu_{\lambda} - \mu_{\kappa}| (z) \,dt \notag \\
& + C \int_0^T \norm{\psi(t,\cdot)}_{\infty}\, \int_{|z|> r_1} \big[\norm{u_0(\cdot+z) -u_0}_{L^1(\R^d)}+\norm{v_0(\cdot+z) -v_0}_{L^1(\R^d)}\big] \,d |\mu_{\lambda} - \mu_{\kappa}| (z) \,dt \notag
\end{align} 
where $\mathcal{E}(\sigma, \widetilde{\sigma})^2:= \underset{k\ge 1}\sum \mathcal{E}_k(\sigma, \widetilde{\sigma})^2$, and we have used the fact that $\sum_{\pm} \pm (\mu_{\lambda_{\pm}} -\mu_{\kappa_{\pm}}) = |\mu_{\lambda} - \mu_{\kappa}|$.


To proceed further, we make a special choice for the function $\psi(t,x)$. To this end, for each $h>0$ and fixed $t\ge 0$, we define $\psi_h^t$ and remind $\psi_R$:
 \begin{align}
 \psi_h^t(s)=\begin{cases} 1, &\quad \text{if}~ s\le t, \notag \\
 1-\frac{s-t}{h}, &\quad \text{if}~~t\le s\le t+h,\notag \\
 0, & \quad \text{if} ~ s \ge t+h.
 \end{cases}
 \quad 
 \psi_R(x) = \min \Bigg(1, \frac{R^a}{|x|^a} \Bigg).
 \end{align}
Furthermore, let $\rho$ be any nonnegative mollifier. Clearly, by truncation arguments, \eqref{stoc_entropy_4} holds with $\psi(s,x)=\psi_h^t(s) \, (\psi_R \star \rho)(x)$. 
 
With the above choice of test function in \eqref{stoc_entropy_4}, we first wish to pass to the limit as $R \to \infty$ and subsequently as $r_2 \to 0$ in \eqref{stoc_entropy_4}. Thanks to the a priori estimates in Appendix~\ref{sec:apriori+existence}, we recall that $u, v \in L^1(\Omega \times \Pi_T)$. Also note that by properties of $\psi_R$ and $\rho $, it follows that $\psi_R \star \rho \to 1$ pointwise as $R \to \infty$. 
Therefore, for any $x$ and $z$, $\psi_R \star \rho(x) -\psi_R \star \rho(x+z) \to 0$ and since it is bounded by $2$ which is integrable on the set $\{|z|>r_2\}$ with respect to $\mu_{\kappa_{+}}$, one concludes that $\mathcal{L}^{r_2}_{\kappa_{+}} [\psi(t,\cdot)](x) \to 0$. 
\\
As moreover $|\mathcal{L}^{r_2}_{\kappa_{+}} [\psi(t,\cdot)](x)| \leq 2 \int_{|z|>r_2} \frac{dz}{r^{d+2\kappa_{+}}}$, Lebesgue Theorem once again, yields 
\begin{align*}
& \E \Big[\int_{\R^d} \int_{\Pi_T} \beta(u(t,x)-v(t,y)) \mathcal{L}^{r_2}_{\kappa_{+}} [\psi(t,\cdot)](x) \, \varrho_{\delta}(x-y)\,dy\,dx\,dt \Big] \to 0\quad (R \to + \infty).
\end{align*}
Of course, the same holds with $\lambda_+$.
\\[0.1cm]
On the other hand, 
\begin{align*}
&C_{\kappa,\lambda}[r_2^{2(1-\kappa)}+r_2^{2(1-\lambda)}] \E \Big[\int_{\R^d} \int_{\Pi_T}  \|D_x^2 \psi(t,\cdot)\|_\infty\beta(u(t,x)-v(t,y)) \, \varrho_{\delta}(x-y)\,dy\,dx\,dt \Big]
\\
\leq&C_{\kappa,\lambda}[r_2^{2(1-\kappa)}+r_2^{2(1-\lambda)}] \|D_x^2 \rho\|_{L^1} \E \Big[\int_{\R^d} \int_{\Pi_T}  \beta(u(t,x)-v(t,y)) \, \varrho_{\delta}(x-y)\,dy\,dx\,dt \Big] \to 0 \quad (r_2 \to 0).
\end{align*}
Hence, a new simple application of dominated convergence theorem yields
\begin{align}
\label{stoc_entropy_4_02}
 0 & \le   \E \Big[\int_{\R^d} \int_{\D}\beta(u_0(x)-v_0(y)) \varrho_{\delta} (x-y) \,dx\,dy \Big] + 
 \E \Big[\int_{\Pi_T}\int_{\D} \beta \big(u(s,x)-v(s,y)\big) \partial_s \psi^t_h(s) \varrho_{\delta} (x-y) \,dx\,dy\,ds\Big]  
 \notag \\
 & \qquad + C\E \Big[\int_{\Pi_T} \int_{\D} \beta \big(u(s,x)-v(s,y)\big) \psi^t_h(s) \varrho_{\delta} (x-y)\,dx\,dy\,ds\Big]
 \\&\qquad
 + C \Bigg(\frac{\mathcal{E}(\sigma, \widetilde{\sigma})^2}{\xi} + C(|u_0|_{\mathrm{BV}})(\xi + ||f^\prime-g^\prime||_{\infty}) + \sqrt{\mathcal{D}(\eta, \widetilde{\eta})}+ \frac{\mathcal{D}(\eta, \widetilde{\eta})}{\xi}  \Bigg) \int_{0}^T \psi^t_h(s) \,ds  \notag \\
&\qquad + \frac{C}{\delta} (|u_0|_{\mathrm{BV}}+|v_0|_{\mathrm{BV}}) \int_{0}^T \psi^t_h(s) \,ds \, \int_{0<|z|\le r_1} |z|^{2} \,d |\mu_{\lambda} - \mu_{\kappa}| (z) \notag \\
&\qquad + C \int_{0}^T \psi^t_h(s) \,ds\, \int_{|z|> r_1} (\norm{u_0(\cdot+z) -u_0}_{L^1(\R^d)}+\norm{v_0(\cdot+z) -v_0}_{L^1(\R^d)}) \,d |\mu_{\lambda} - \mu_{\kappa}| (z). \notag
\end{align} 
Let $\mathbb{T}$ be the set all points $t$ in $[0, \infty)$ such that $t$ is a right
 Lebesgue point of 
 \begin{align*}
\mathcal{B}(t)= \E \Big[\int_{\R^d}\int_{\D} \beta \big(u(t,x)-v(t,x)\big) \varrho_{\delta} (x-y)\,dx\,dy \Big].
\end{align*}

\noindent Clearly, $\mathbb{T}^{\complement}$ has zero Lebesgue measure. Fix  $t\in \mathbb{T}$. Thus, passing to the limit as $h\goto 0$ in \eqref{stoc_entropy_4_02}, we obtain
 \begin{align*}
&\E \Big[\int_{\R^{d}}\int_{\D} \beta \big( u(t,x)-v(t,y)\big) \varrho_{\delta} (x-y)\,dx\,dy \Big] 
\le \E \Big[\int_{\R^{d}} \big|u_0(x)-v_0(y)\big| \varrho_{\delta} (x-y) \,dx\,dy\Big] \\
&+ C \E\Big[\int_0^{t}\int_{\R^{d}}  \beta \big( u(t,x)-v(t,y)\big) \varrho_{\delta} (x-y) \,dx\,dy\,ds \Big]   \notag
 \\&
 + C\,t\, \Bigg(\frac{\mathcal{E}(\sigma, \widetilde{\sigma})^2}{\xi} + C(|u_0|_{\mathrm{BV}})(\xi + ||f^\prime-g^\prime||_{\infty}) + \sqrt{\mathcal{D}(\eta, \widetilde{\eta})} + \frac{\mathcal{D}(\eta, \widetilde{\eta})}{\xi} \Bigg) \notag \\
& + \frac{Ct}{\delta} (|u_0|_{\mathrm{BV}}+|v_0|_{\mathrm{BV}}) \, \int_{0<|z|\le r_1} |z|^{2} \,d |\mu_{\lambda} - \mu_{\kappa}| (z)
 \notag \\
 & 
 + Ct \, \int_{|z|> r_1} \norm{(u_0(\cdot+z) -u_0}_{L^1(\R^d)}+\norm{v_0(\cdot+z) -v_0}_{L^1(\R^d)}) \,d |\mu_{\lambda} - \mu_{\kappa}| (z). \notag
\end{align*}
Therefore, a simple application of Gronwall argument reveals that
\begin{align} \label{inq:pre-final}
\E & \Big[\int_{\R^{d}}\int_{\D} \beta \big( u(t,x)-v(t,y)\big) \varrho_{\delta} (x-y)\,dx\,dy \Big]  \le e^{Ct}\E \Big[\int_{\R^{d}} \big|u_0(x)-v_0(y) \big| \varrho_{\delta} (x-y) \,dx\,dy\Big]
 \\& + Ce^{Ct} \Bigg(\frac{\mathcal{E}(\sigma, \widetilde{\sigma})^2}{\xi} + C(|u_0|_{\mathrm{BV}})(\xi + \|f^\prime-g^\prime\|_{\infty}) + \sqrt{\mathcal{D}(\eta, \widetilde{\eta})}+ \frac{\mathcal{D}(\eta, \widetilde{\eta})}{\xi}  \Bigg) t \notag \\
& + \frac{Cte^{Ct}}{\delta} (|u_0|_{\mathrm{BV}}+|v_0|_{\mathrm{BV}}) \, \int_{0<|z|\le r_1} |z|^{2} \,d |\mu_{\lambda} - \mu_{\kappa}| (z) \notag
\\& + Cte^{Ct} \, \int_{|z|> r_1} (\norm{u_0(\cdot+z) -u_0}_{L^1(\R^d)}+\norm{v_0(\cdot+z) -v_0}_{L^1(\R^d)}) \,d |\mu_{\lambda} - \mu_{\kappa}| (z). \notag
\end{align}
Let us consider now a bounded by $1$ weight-function $\varphi\in L^1(\R^d)$, non-negative (for example, a negative exponential of $|x|$). 
\\
 Again, in view of BV bound of the entropy solutions $u(t,x)$ and $v(t,y)$, and by using $|r|\le M_1 \xi + \beta_\xi(r)$, we have 
 \begin{align}
\E \Big[\int_{\R^d} & \big|v(t,y)-u(t,y)\big|\varphi(y)\,dy\Big]  \notag \\
  \le &   \E \Big[\int_{\R^d}\int_{\R^d}  \big| v(t,y)-u(t,x)\big| \varphi(y)
\varrho_\delta(x-y)\,dx\,dy \Big]
 + \E \Big[\int_{\R^d}\int_{\R^d} \big| u(t,x) -u(t,y)\big|\varphi(y)\varrho_\delta(x-y)\,dx\,dy \Big]\notag \\
\le & \E \Big[ \int_{\R^d}\int_{\R^d} \big|v(t,y) -u(t,x)\big|\varphi(y)\varrho_\delta(x-y)\,dx\,dy \Big]
  + \delta\,|u_0|_{\mathrm{BV}} \notag \\
\le &  \E  \Big[\int_{\R^{d}}\int_{\D} \beta \big( u(t,x)-v(t,y)\big) \varphi(y)\varrho_{\delta} (x-y)\,dx\,dy \Big] + C\xi \|\varphi\|_{L^1(\R^d)}  + \delta\,|u_0|_{\mathrm{BV}}\notag \\
\le &  \E  \Big[\int_{\R^{d}}\int_{\D} \beta \big( u(t,x)-v(t,y)\big) \varrho_{\delta} (x-y)\,dx\,dy \Big] + C\xi \|\varphi\|_{L^1(\R^d)}  + \delta\,|u_0|_{\mathrm{BV}}, \label{estimate:solu}
 \end{align} 
and 
 \begin{align}
   \E \Big[\int_{\R^d}\int_{ \R^d} \big| u_0(x) -v_0(y)\big|\varrho_\delta(x-y) \,dx\, dy\Big]
   \le  \E \Big[\int_{\R^d} \big|u_0(x)-v_0(x)\big| \,dx \Big]+ \delta\, |u_0|_{\mathrm{BV}}. \label{estimate:ini}
 \end{align}
So making use of \eqref{estimate:solu}, and \eqref{estimate:ini} in \eqref{inq:pre-final} yields
\begin{align}
  \E& \Big[\int_{\R^d} \int_{\D} \big| u(t,x)-v(t,x)\big|\varphi(x)\,dx \Big] 
 \label{inq:pre-final_001}
 \\
 \le& e^{Ct} \E\Big[\int_{ \R^d} \int_{\D} \big| u_0(x) -v_0(y)\big| \varrho_{\delta} (x-y)\,dx\,dy \Big]  
 + 2 \delta \,|u_0|_{\mathrm{BV}} + C\xi \|\varphi\|_{L^1(\R^d)} \notag \\
 & + Ce^{Ct} \Bigg(\frac{\mathcal{E}(\sigma, \widetilde{\sigma})^2}{\xi} + C(|u_0|_{\mathrm{BV}}(\xi + ||f^\prime-g^\prime||_{\infty}) + \sqrt{\mathcal{D}(\eta, \widetilde{\eta})} + \frac{\mathcal{D}(\eta, \widetilde{\eta})}{\xi} \Bigg) t.\notag \\
 & + \frac{Cte^{Ct}}{\delta} (|u_0|_{\mathrm{BV}}+|v_0|_{\mathrm{BV}}) \, \int_{0<|z|\le r_1} |z|^{2} \,d |\mu_{\lambda} - \mu_{\kappa}| (z) \notag \\
& + Cte^{Ct} \, \int_{|z|> r_1} (\norm{u_0(\cdot+z) -u_0}_{L^1(\R^d)}+\norm{v_0(\cdot+z) -v_0}_{L^1(\R^d)}) \,d |\mu_{\lambda} - \mu_{\kappa}| (z). \notag
 \end{align}
Now we optimize the terms involving $\delta$ in \eqref{inq:pre-final_001} by using the formula $\min_{\delta>0} \big(\delta a + \frac{b}{\delta}  \big) = 2 \sqrt{ab}$, for $a,b \ge 0$. This result is
\begin{align}
 \label{inq:pre-final_002}
  \E & \Big[\int_{\R^d} \int_{\D} \big| u(t,x)-v(t,x)\big|\varphi(x)\,dx \Big] 
 \le e^{Ct} \E\Big[\int_{ \R^d} \int_{\D} \big| u_0(x) -v_0(y) \varrho_{\delta} (x-y)\,dx\,dy \Big]   + C\xi \|\varphi\|_{L^1(\R^d)}\\
 & + Ce^{Ct} \Bigg(\frac{\mathcal{E}(\sigma, \widetilde{\sigma})^2}{\xi} + C(|u_0|_{\mathrm{BV}(\D)})(\xi + ||f^\prime-g^\prime||_{\infty}) + \sqrt{\mathcal{D}(\eta, \widetilde{\eta})} + \frac{\mathcal{D}(\eta, \widetilde{\eta})}{\xi} \Bigg) t.\notag \\
 & + C_T e^{Ct}\,(|u_0|_{\mathrm{BV}(\D)}+|v_0|_{\mathrm{BV}(\D)})  \,\sqrt{\int_{0<|z|\le r} |z|^2 \,d |\mu_{\lambda} - \mu_{\kappa}|(z)} \notag \\
& + C_Te^{Ct} \, \int_{|z|> r_1} (\norm{u_0(\cdot+z) -u_0}_{L^1(\R^d)}+\norm{v_0(\cdot+z) -v_0}_{L^1(\R^d)}) \,d |\mu_{\lambda} - \mu_{\kappa}| (z). \notag
 \end{align}
Again, by choosing $\xi= \max \bigg\{\mathcal{E}(\sigma, \widetilde{\sigma}), \sqrt{\mathcal{D}(\eta,\widetilde{\eta})}\bigg\}\sqrt{t}$ 
in \eqref{inq:pre-final_002}, we arrive at 
 \begin{align}
& \E \Big[\int_{\R^d}\big| u(t,x)-v(t,x)\big|\varphi(x)\,dx \Big] \label{esti:pre-final-2} \\
  \le& e^{Ct} \,\E\Big[\int_{\R^d}\big| u_0(x) -v_0(x)\big| \,dx\Big] + C(|u_0|_{\mathrm{BV}(\D)})e^{Ct}   ||f^\prime-g^\prime||_{\infty} t  \notag \\
&+C \max \bigg\{ \mathcal{E}(\sigma, \widetilde{\sigma}), \sqrt{\mathcal{D}(\eta,\widetilde{\eta})}\bigg\}    \Bigg(  \|\varphi\|_{L^1(\R^d)} + e^{Ct}\Big[1+C(|u_0|_{\mathrm{BV}(\D)})t + \sqrt{t}\Big]  \Bigg)\sqrt{t}
   \notag \\
  & +Ce^{Ct} \,(|u_0|_{\mathrm{BV}(\D)}+|v_0|_{\mathrm{BV}(\D)})  \,\sqrt{\int_{0<|z|\le r_1} |z|^2 \,d |\mu_{\lambda} - \mu_{\kappa}|(z)}    \notag \\
  & + Ce^{Ct} \, \int_{|z|> r_1} (\norm{u_0(\cdot+z) -u_0}_{L^1(\R^d)}+\norm{v_0(\cdot+z) -v_0}_{L^1(\R^d)}) \,d |\mu_{\lambda} - \mu_{\kappa}| (z) \,dt. \notag
  \end{align}
Hence, we conclude that for a.e. $t>0$, 
  \begin{align*}
     & \E \Big[\int_{\R^d}\big| u(t,x)-v(t,x)\big|\varphi(x)\,dx \Big] \notag \\
     & \hspace{0.5cm} \le C_Te^{Ct} \Bigg\{ \E\Big[\int_{\R^d}\big| u_0(x) -v_0(x)\big| \,dx\Big] 
     + \max \bigg\{ \mathcal{E}(\sigma, \widetilde{\sigma}),\sqrt{\mathcal{D}(\eta,\widetilde{\eta})}\bigg\} \sqrt{t} + t\,||f^\prime-g^\prime||_{L^{\infty}(\D)}  \\
& \qquad + \sqrt{\int_{0<|z|\le r_1} |z|^2 \,d |\mu_{\lambda} - \mu_{\kappa}|(z)} 
+ \int_{|z|> r_1} \Big(\norm{u_0(\cdot + z) -u_0}_{L^1(\D)}+\norm{v_0(\cdot + z) -v_0}_{L^1(\D)}\Big) \,d |\mu_{\lambda} - \mu_{\kappa}|(z)
\Bigg\},
\end{align*} for some constant $C_T$ depending on $T$, $|u_0|_{BV(\R^d)}, |v_0|_{BV(\R^d)}, \|f^{\prime\prime}\|_{\infty}, \|f^\prime\|_{\infty}$, and $\|\varphi\|_{L^1}$.
Finally, observe that
\begin{align*}
\int_{|z|> r_1} &\Big(\norm{u_0(\cdot + z) -u_0}_{L^1(\D)}+\norm{v_0(\cdot + z) -v_0}_{L^1(\D)}\Big) \,d |\mu_{\lambda} - \mu_{\kappa}|(z) \\
\leq & 2\big(\|u_0\|_{L^1(\R^d)}+ \|v_0\|_{L^1(\R^d)}\big)
\int_{|z|> r_1} \,d |\mu_{\lambda} - \mu_{\kappa}|(z).
\end{align*}
This essentially completes the proof.


\section{Proof of Corollary~\ref{rate-of-convergence}: Rate of Convergence}
\label{sec:cor}
We have already shown that the vanishing viscosity solutions $u_\eps(t,x)$ of the problem \eqref{eq:viscous-Brown} converge (in an appropriate sense) to the unique entropy solution $u(t,x)$ of the stochastic conservation law \eqref{eq:stoc_con_brown}. In this section, we wish to investigate the nature of such convergence described by a rate of convergence. Indeed, as a by product of the continuous dependence estimates (cf. Section~\ref{cont-depen-estimate}), we explicitly obtain the rate of convergence of vanishing viscosity solutions to the unique BV entropy solution of the underlying problem \eqref{eq:stoc_con_brown}. 

To that context, for $\eps>0$, let $u_\eps$ be the weak solution to the problem \eqref{eq:viscous-Brown} with data $(u_0, f,\sigma, \eta, \lambda)$ and $u(t,x)$ be the entropy solution. A similar argument (with $\lambda=\kappa$ leading to \eqref{stoc_entropy_4}), as in the proof of the Theorem~\ref{continuous-dependence}, yields, for any parameters $0<r_2$, 
\begin{align}
 0 \le &  \E \Big[\int_{\R^d} \int_{\R^d} \big|u_0(x)-u^{\eps}_{0}(y)\big|\psi(0,x)\varrho_{\delta} (x-y)\,dx\,dy\Big] \label{esti:final-02bis}  
 \\& + 
 \E \Big[\int_{\Pi_T}\int_{\R^d} \beta \big(u(t,x)-u_\eps(t,y)\big) \Big[\partial_t\psi(t,x) + C\, \psi(t,x) \Big]\varrho_\delta(x-y)\,dy\,dx\,dt\Big] + \frac{C \, \eps}{\delta} \notag  
\\&- \E \Big[\int_{\Pi_T}\int_{\R^d} f^{\beta} \big(u(t,x),u_\eps(t,y)\big) \cdot \grad_x \psi(t,x)\varrho_{\delta}(x-y)\,dx\,dy\,dt \Big]  + C(|u_0|_{BV})\, \xi\int^T_0 \| \psi(t,\cdot)\|_{L^\infty(\Rd)}\,dt  \notag \\
& - \E \Big[\int_{\R^d} \int_{\Pi_T}  
\beta \big(u(t,x) - u_\eps(t,y)\big) \, \mathcal{L}^{r_2}_{\lambda}[\psi(t,\cdot)](x)\, \varrho_{\delta} (x-y) \,dx\,dt \,dy \Big] \notag
\\
&+ C_{\kappa,\lambda}r_2^{2(1-\lambda)} \E \Big[\int_{\R^d} \int_{\Pi_T}  \|D_x^2 \psi(t,\cdot)\|_\infty\beta(u(t,x)-u_\eps(t,y)) \, \varrho_{\delta}(x-y)\,dy\,dx\,dt \Big] \notag.
\end{align}
As before we make a special choice of the test function $\psi(s,x)=\psi_h^t(s)\, (\psi_R \star \rho)(x)$, where $\psi_h^t$, $\psi_R$ are described previously, and then pass to the limit as $R \goto \infty$ and finally $r_2\to 0$. 
Then, similarly to \eqref{stoc_entropy_4_02}, 
\begin{align*}
-\E \Big[\int_{\Pi_T}\int_{\R^{d}}  & \beta \big(u(s,x)-u_\eps(s,y)\big)\,\partial_s\psi^t_h(s)\, \varrho_\delta(x-y)\,dy\,dx\,ds\Big]
\le  \E \Big[\int_{\R^d} \int_{\R^d} \big|u_0(x)-u^{\eps}_0(y)\big|\varrho_{\delta} (x-y)\,dx\,dy\Big]  \\
&+ C\, \E \Big[\int_{\Pi_T}\int_{\R^d} \beta \big(u(s,x)-u_\eps(s,y)\big) \, \psi^t_h(s)\,\varrho_\delta(x-y)\,dy\,dx\,ds\Big]
+  C  \xi   + \frac{C\eps}{\delta}.\notag
\end{align*}
Next, using the relation $\beta_{\xi}(r) \le |r|$ and letting $\xi \to 0$, we conclude
\begin{align*}
-\E \Big[\int_{\Pi_T}\int_{\R^{d}}  & \big|u(s,x)-u_\eps(s,y)\big|\,\partial_s\psi^t_h(s)\, \varrho_\delta(x-y)\,dy\,dx\,ds\Big]
\le  \E \Big[\int_{\R^d} \int_{\R^d} \big|u_0(x)-u^{\eps}_0(y)\big|\varrho_{\delta} (x-y)\,dx\,dy\Big]  \\
&+ C\, \E \Big[\int_{\Pi_T}\int_{\R^d} \big|u(s,x)-u_\eps(s,y)\big| \, \psi^t_h(s)\,\varrho_\delta(x-y)\,dy\,dx\,ds\Big]
+ \frac{C\eps}{\delta}.
\end{align*}
Next, we let the parameter $h\to0$ to get
\begin{align*}
\E \Big[\int_{\R^{2d}} & \big|u(t,x)-u_\eps(t,y)\big|\, \varrho_\delta(x-y)\,dy\,dx\,ds\Big]
\le  \E \Big[\int_{\R^d} \int_{\R^d} \big|u_0(x)-u^{\eps}_0(y)\big|\varrho_{\delta} (x-y)\,dx\,dy\Big]  \\
&+ C\, \E \Big[\int_0^t \int_{\R^{2d}} \big|u(s,x)-u_\eps(s,y)\big| \,\varrho_\delta(x-y)\,dy\,dx\,ds\Big]
+ \frac{C\eps}{\delta},
\end{align*}
and a Gronwall argument reveals that
\begin{align*}
 \E \Big[\int_{\R^{2d}}  \big|u(t,x)-u_\eps(t,y)\big| \varrho_\delta(x-y)\,dy\,dx\Big]  
\le 
e^{Ct}\Big( \E \Big[\int_{\R^d} \int_{\R^d} \big|u_0(x)-u^{\eps}_0(y)\big|\varrho_{\delta} (x-y)\,dx\,dy\Big] 
  + \frac{C\eps}{\delta}\Big).
\end{align*}
Again, since $u_\eps(t,y)$ and $u(t,x)$ satisfy spatial BV bounds, bounded by the BV norm of $u_0$, we obtain 
\begin{align}
  & \E \Big[\int_{\R^d}\big| u_\eps(t,x)-u(t,x)\big|\,dx\Big]\le Ce^{Ct}\Big(\eps^{1/2} + \delta +  \frac{\eps}{\delta}\Big) .\label{esti:final_001}
  \end{align}
Finally, choosing the optimal value of $\delta = \eps^{1/2}$ in \eqref{esti:final_001} yields:  for a.e. $t>0$, 
\begin{align*}
\E\Big[\|u_\eps(t,\cdot)-u(t,\cdot)\|_{L^1(\R^d)}\Big] \le C e^{Ct} \eps^{1/2},
\end{align*}
where $C>0$ is a constant depending only on  
$|u_0|_{BV(\R^d)}, \|f^{\prime\prime}\|_{\infty}$, and $ \|f^\prime\|_{\infty}$. 
This completes the proof.
  
 
\section{Proof of Theorem~\ref{vanishing-non-local}: Vanishing Non-local Regularization}
\label{sec:cor_01}
In this section, our aim is to consider a different (non-local) regularization \eqref{eq:stoc_regularization} of the stochastic conservation laws \eqref{eq:stoc_original}, and derive rate of convergence estimates between the solutions of the stochastic conservation law and the corresponding regularized problem. 

To achieve that, we first consider a regularization of \eqref{eq:stoc_regularization}: 
For $\gamma>0$, let $u_{\eps,\gamma}$ be the weak solution to the problem 
\begin{align}
du_{\eps,\gamma}(s,y) - \gamma \Delta u_{\eps,\gamma}(s,y)\,ds + \eps \mathcal{L}_{\lambda}[u_{\eps,\gamma}(s, \cdot)](y)\,ds  & - \mbox{div}_y f(u_{\eps,\gamma}(s,y)) \,ds  \label{eq:viscous_001} \\
& =\sigma(u_{\eps,\gamma}(s,y))\,dW(s) + \int_{E} \eta(u_{\eps,\gamma}(s,y);z)\,\widetilde{N}(dz,ds), \notag 
\end{align}
with a regular initial data $u_{\eps,\gamma}(0,\cdot) = u_0^{\eps, \gamma}\in H^1(\D)$ such that $\|u_0^{\eps, \gamma}\|_{W^{1,1}}$ is controlled by $C\|u_0^{\eps}\|_{BV}$. We also assume that $u_0^{\eps,\gamma} \underset{\gamma \to 0} \longrightarrow u^{\eps}_0$ in $L^2(\D)$.
In view of Theorem \ref{thm:existence-bv}, we conclude that $u_{\eps,\gamma}(s,y)$ converges to the unique BV-entropy solution $u_{\eps}(s, y)$ of \eqref{eq:stoc_regularization} with initial data $u_0$.

In what follows, let $u(t,x)$ be the unique entropy solution of the stochastic conservation laws \eqref{eq:stoc_original}. We now write the entropy inequality for $u(t,x)$, based on the entropy pair $(\beta(\cdot-k), f^\beta(\cdot, k))$, and then multiply by $\varsigma_l(u_{\eps,\gamma}(s,y)-k)$, take the expectation, and integrate with respect to $ s, y, k$. The result is
\begin{align}
0\le  & \E \Big[\int_{\Pi_T}\int_{\R^d}\int_{\R} \beta(u(0,x)-k)
\varphi_{\delta,\delta_0}(0,x,s,y) \varsigma_l(u_{\eps,\gamma}(s,y)-k)\,dk \,dx\,dy\,ds\Big] \notag \\
 &\qquad + \E \Big[\int_{\Pi_T} \int_{\Pi_T} \int_{\R} \beta(u(t,x)-k)\partial_t \varphi_{\delta,\delta_0}(t,x,s,y)
\varsigma_l(u_{\eps,\gamma}(s,y)-k)\,dk \,dx\,dt\,dy\,ds \Big]\notag \\ 
& + \E \Big[\sum_{k\ge 1}\int_{\Pi_T} \int_{\Pi_T} 
\int_{\R}  g_k(u(t,x)) \, \beta^\prime (u(t,x)-k)\, \varphi_{\delta,\delta_0}\, dx \,d\beta_k(t) \varsigma_l(u_{\eps,\gamma}(s,y)-k)\,dk\,dy\,ds \Big] \notag \\
 &+ \frac{1}{2}\, \E \Big[ \int_{\Pi_T} \int_{\Pi_T}  
\int_{\R} \mathbb{G}^2(u(t,x))\, \beta^{\prime\prime} (u(t,x) -k)\, \varphi_{\delta,\delta_0} \,dx\,dt 
\, \varsigma_l(u_{\eps,\gamma}(s,y)-k)\,dk\,dy\,ds \Big] \notag \\
 & \qquad +  \E \Big[ \int_{\Pi_T} \int_{\R}\int_{\Pi_T}\int_{|z|>0}\Big(\beta \big(u(t,x) +\eta(u(t,x);z)-k\big)-\beta(u(t,x)-k)\Big) \notag \\
& \hspace{6cm} \times \varphi_{\delta,\delta_0}(t,x,s,y)\,\varsigma_l(u_{\eps,\gamma}(s,y)-k) \,\tilde{N}(dz,dt) \,dx \,dk \,dy\,ds \Big] \notag\\
&\qquad +  \E \Big[\int_{\Pi_T} \int_{t=0}^T\int_{|z|>0}\int_{\R^d} 
\int_{\R} \Big(\beta \big(u(t,x) +\eta(u(t,x);z)-k\big)-\beta(u(t,x)-k) \notag \\
 & \hspace{6.5cm}-\eta(u(t,x);z) \beta^{\prime}(u(t,x)-k)\Big)
 \varphi_{\delta,\delta_0}(t,x;s,y) \notag \\
&\hspace{8cm}\times \varsigma_l(u_{\eps,\gamma}(s,y)-k)\,dk\,dx\,\nu(dz)\,dt\,dy\,ds\Big]\notag \\
& \qquad +  \E \Big[\int_{\Pi_T}\int_{\Pi_T} \int_{\R} 
 f^\beta(u(t,x),k) \cdot \grad_x \varphi_{\delta,\delta_0}(t,x,s,y) \, \varsigma_l(u_{\eps,\gamma}(s,y)-k)\,dk\,dx\,dt\,dy\,ds\Big] \notag \\[2mm]
& =:  I_1 + I_2 + I_3 +I_4 + I_5 + I_6 + I_7. \label{stochas_entropy_1-levy_001}
\end{align}

We now apply the It\^{o}-L\'{e}vy formula to \eqref{eq:viscous_001} and multiply with test 
function $\varphi_{\delta, \delta_0}(t,x,s,y)$ and $\varsigma_l(u(t,x)-k)$ and integrate. The result is
\begin{align}
 0\le  &\, \E \Big[\int_{\Pi_T}\int_{\R^d}\int_{\R} 
 \beta(u_{\eps,\gamma}(0,y)-k)\varphi_{\delta,\delta_0}(t,x,0,y) \varsigma_l(u(t,x)-k)\,dk \,dx\,dy\,dt\Big] \notag \\
   &  +   \E \Big[\int_{\Pi_T} \int_{\Pi_T} \int_{\R} 
 \beta(u_{\eps,\gamma}(s,y)-k)\partial_s \varphi_{\delta,\delta_0}(t,x,s,y)
 \varsigma_l(u(t,x)-k)\,dk \,dy\,ds\,dx\,dt\Big] \notag \\ 
 & + \E \Big[\sum_{k\ge 1}\int_{\Pi_T} \int_{\Pi_T}
\int_{\R} g_k(u_{\eps,\gamma}(s,y))\,\beta^\prime (u_{\eps,\gamma}(s,y)-k)\, \varphi_{\delta,\delta_0}\, dx \,d\beta_k(t) \varsigma_l(u(t,x)-k)\,dk\,dy\,ds \Big] \notag \\
 &+ \frac{1}{2}\, \E \Big[ \int_{\Pi_T} \int_{\Pi_T} 
\int_{\R} \mathbb{G}^2(u_{\eps,\gamma}(s,y)) \,\beta^{\prime\prime} (u_{\eps,\gamma}(s,y) -k)\, \varphi_{\delta,\delta_0} \,dx\,dt 
\, \varsigma_l(u(t,x)-k)\,dk\,dy\,ds \Big] \notag \\
  + &  \E \Big[\int_{\Pi_T} \int_{\Pi_T}\int_{|z|>0} \int_{\R} 
 \Big(\beta \big(u_{\eps,\gamma}(s,y) +\eta(u_{\eps,\gamma}(s,y);z)-k\big)
 -\beta(u_{\eps,\gamma}(s,y)-k)\Big) \notag \\
 & \hspace{6.5cm} \times \varphi_{\delta,\delta_0}(t,x,s,y)\varsigma_l(u(t,x)-k)\,dk \,\tilde{N}(dz,ds)\,dy\,dx\,dt \Big]\notag\\
  + &  \E \Big[\int_{\Pi_T} \int_{s=0}^T\int_{|z|>0}\int_{\R^d} 
 \int_{\R} \Big(\beta \big(u_{\eps,\gamma}(s,y) +\eta(u_{\eps,\gamma}(s,y);z)-k\big)
 -\beta(u_{\eps,\gamma}(s,y)-k) \notag \\
  & \hspace{6.0cm}-\eta(u_{\eps,\gamma}(s,y);z) \,\beta^{\prime}(u_{\eps,\gamma}(s,y)-k)\Big)
  \varphi_{\delta,\delta_0}(t,x;s,y) \notag \\
 &\hspace{8.7cm}\times \varsigma_l(u(t,x)-k)\,dk\,dy\,\nu(dz)\,ds\,dx\,dt\Big]\notag \\
  & \quad +  \E\Big[\int_{\Pi_T}\int_{\Pi_T} \int_{\R}  
 f^\beta(u_{\eps,\gamma}(s,y),k)\cdot \grad_y \varphi_{\delta,\delta_0}(t,x,s,y)\, \varsigma_l(u(t,x)-k)\,dk\,dx\,dt\,dy\,ds\Big] \notag \\
 &- \gamma \, \E\Big[ \int_{\Pi_T} \int_{\Pi_T} 
  \int_{\R} \beta^\prime(u_{\eps,\gamma}(s,y)-k)\grad_y u_{\eps,\gamma}(s,y)\cdot \grad_y  \varphi_{\delta,\delta_0}(t,x,s,y) \, \varsigma_l(u_(t,x)-k)\,dk\, \,dy\,ds\,dx\,dt\Big]  \notag \\ 
& - \eps\, \E \Big[\int_{\Pi_T} \int_{\Pi_T} 
\int_{\R} \mathcal{L}_{\lambda}[u_{\eps,\gamma}(s, \cdot)](y) \, \varphi_{\delta,\delta_0}(t,x,s,y)\, \beta^\prime (u_{\eps,\gamma}(s,y)-k)  \,dx\,dt \, \varsigma_l(u(t,x)-k)\,dk\,dy\,ds \Big]  \notag \\[2mm]
& =: J_1 + J_2 + J_3 + J_4 + J_5 + J_6 + J_7 + J_8 + J_9. \label{stochas_entropy_3-levy_001}
\end{align}
Again, as usual, our aim is to add \eqref{stochas_entropy_1-levy_001} and \eqref{stochas_entropy_3-levy_001}, 
and pass to the  limits with respect to the various parameters involved. 
  
First we remark that we use same strategies (cf. Section~\ref{uniqueness}) to deal with the terms $I_1, I_2, \cdots, I_7$ and $J_1, J_2, \cdots, J_8$. In fact adding all these terms and then passing to the limits as $\delta_0 \goto 0$, $l \goto 0$, and $\gamma\to0$ yields a majoration by
\begin{align}
&\quad \E \Big[\int_{\R^d} \int_{\R^d} \big|u_0(x)-u^{\eps}_0(y)\big|\psi(0,x)\varrho_{\delta} (x-y)\,dx\,dy\Big] + C  \xi  \int_{t=0}^T ||\psi(t,\cdot)||_{L^\infty(\R^d)}\,dt\label{esti:final-02bis_001}  
 \\& \quad + 
 \E \Big[\int_{\Pi_T}\int_{\R^d} \beta \big(u(t,x)-u_\eps(t,y)\big)  
[ \partial_t \psi(t,x) + C \psi(t,x)]
\varrho_\delta(x-y)\,dy\,dx\,dt\Big]\notag 
\\&
 \quad \quad -\E\Big[\int_{\Pi_T}\int_{\R^d} f^{\beta} \big(u(t,x),u_\eps(t,y)\big) \cdot \grad_x \psi(t,x)\varrho_{\delta}(x-y)\,dx\,dy\,dt \Big]  \notag 
\end{align}
On the other hand, regarding the non-local term, we invoke similar strategy, as in the uniqueness proof (cf. Section~\ref{uniqueness}), to conclude
\begin{lem}
\label{fractional_lemma_one}
It holds that 
\begin{align}
\limsup_{\gamma\to0}\, \lim_{l \goto 0}\, \lim_{\delta_0 \goto 0} J_{9} \le 
&- \eps \,\E \Big[\int_{\R^d} \int_{\Pi_T} 
\mathcal{L}^r_{\lambda} [u_{\eps}(t,\cdot)](y)\, \psi(t,x)\,\varrho_{\delta} (x-y)\, \beta'(u_\eps(t,y) - u(t,x))  \,dx\,dt 
\,dy \Big] \notag \\
& - \eps \,\E \Big[\int_{\R^d} \int_{\Pi_T}  \beta(u_\eps(t,y) - u(t,x) ) \,\mathcal{L}_{\lambda,r} \big[ \varrho_{\delta} (x -\cdot)\big](y)\,\psi(t,x) \,dx\,dt \,dy \Big]. \notag
\end{align}
\end{lem}
\noindent So, in view of \eqref{esti:final-02bis_001}, and Lemma~\ref{fractional_lemma_one}, we conclude  
 \begin{align}
 0 \le & \, \E \Big[\int_{\R^d} \int_{\R^d} \big|u_0(x)-u^{\eps}_0(y)\big| \psi(0,x)\varrho_{\delta} (x-y)\,dx\,dy\Big] + C  \,\xi \int^T_0 \| \psi(t,\cdot)\|_{L^\infty(\Rd)}\,dt\label{esti:final-02biss_002}  
 \\& \quad + 
 \E \Big[\int_{\Pi_T}\int_{\R^d} \beta \big(u(t,x)-u_\eps(t,y)\big)  
[ \partial_t \psi(t,x) + C \psi(t,x)]
\varrho_\delta(x-y)\,dy\,dx\,dt\Big]\notag 
\\&
 \quad  -\E\Big[\int_{\Pi_T}\int_{\R^d} f^\beta \big(u(t,x),u_\eps(t,y)\big) \cdot \grad_x \psi(t,x)\varrho_{\delta}(x-y)\,dx\,dy\,dt \Big]  \notag \\
 &\qquad - \eps \,\E \Big[\int_{\R^d} \int_{\Pi_T} 
\mathcal{L}^r_{\lambda} [u_{\eps}(t,\cdot)](y)\, \psi(t,x)\,\varrho_{\delta} (x-y)\, \beta'(u_\eps(t,y) - u(t,x))  \,dx\,dt 
\,dy \Big] \notag \\
& \qquad - \eps \,\E \Big[\int_{\R^d} \int_{\Pi_T}  \beta(u_\eps(t,y) - u(t,x) ) \,\mathcal{L}_{\lambda,r} \big[ \varrho_{\delta} (x -\cdot)\big](y)\,\psi(t,x) \,dx\,dt \,dy \Big] \notag
\end{align}
As before we use $\psi(s,x)=\psi_h^t(s)\, (\psi_R \star \rho)(x)$, where $\psi_h^t$, $\psi_R$, and $\rho$ are as before, and then pass to the limit as $R \goto \infty$ in \eqref{esti:final-02biss_002}. The resulting expression reads as
\begin{align}
-\E \Big[\int_{\Pi_T}\int_{\R^{d}}  & \beta \big(u(s,x)-v_\eps(s,y)\big)\,\partial_s\psi^t_h(s)\, \varrho_\delta(x-y)\,dy\,dx\,ds\Big]
\le  \E \Big[\int_{\R^d} \int_{\R^d} \big|u_0(x)-u^{\eps}_0(y)\big|\varrho_{\delta} (x-y)\,dx\,dy\Big]  \notag \\
&+ C\, \E \Big[\int_{\Pi_T}\int_{\R^d} \beta_\xi \big(u(s,x)-v_\eps(s,y)\big) \, \psi^t_h(s)\,\varrho_\delta(x-y)\,dy\,dx\,ds\Big]+  C \, \xi\int^T_0 \psi_h^t(s)\,ds   \notag \\
&- \eps \, \E \Big[\int_{\R^d} \int_{\Pi_T}  
\mathcal{L}^r_{\lambda} [u_{\eps}(t,\cdot)](y)\, \varrho_{\delta} (x-y)\, \beta'(u_\eps(t,y) - u(t,x))  \,dx\,dt 
\,dy \Big] \notag \\
& - \eps \, \E \Big[\int_{\R^d} \int_{\Pi_T}  \beta(u_\eps(t,y) - u(t,x) ) \,\mathcal{L}_{\lambda,r} \big[ \varrho_{\delta} (x -\cdot)\big](y)\,dx\,dt \,dy \Big]. \label{eq:imp_01_01}
\end{align}
Next, we intend to pass to the limit as $\xi \goto 0$ in \eqref{eq:imp_01_01}. In fact, a simple application of Lebesgue's dominated convergence theorem reveals that 
\begin{align}
-\E \Big[\int_{\Pi_T}\int_{\R^{d}}  & \big|u(s,x)-v_\eps(s,y)\big|\,\partial_s\psi^t_h(s)\, \varrho_\delta(x-y)\,dy\,dx\,ds\Big]
\le  \E \Big[\int_{\R^d} \int_{\R^d} \big|u_0(x)-u^{\eps}_0(y)\big|\varrho_{\delta} (x-y)\,dx\,dy\Big]  \notag \\
&+ C\, \E \Big[\int_{\Pi_T}\int_{\R^d} \big|u(s,x)-v_\eps(s,y)\big| \, \psi^t_h(s)\,\varrho_\delta(x-y)\,dy\,dx\,ds\Big] \notag \\
&- \eps \, \E \Big[\int_{\R^d} \int_{\Pi_T}  
\mathcal{L}^r_{\lambda} [u_{\eps}(t,\cdot)](y)\, \varrho_{\delta} (x-y)\, \sgn (u_\eps(t,y) - u(t,x))  \,dx\,dt 
\,dy \Big] \notag \\
& - \eps \, \E \Big[\int_{\R^d} \int_{\Pi_T}  \big|u_\eps(t,y) - u(t,x) \big| \,\mathcal{L}_{\lambda,r} \big[ \varrho_{\delta} (x -\cdot)\big](y)\,dx\,dt \,dy \Big]. \label{eq:imp_01}
\end{align}
We denote the last two terms of the above inequality \eqref{eq:imp_01} by $\mathcal{H}^r$ and $\mathcal{H}_r$ respectively. Now we want to estimate each of these terms separately. Following \cite{Alibaud}, first note that
\begin{align*}
|\mathcal{H}^r| &\le C \eps\, \E \Big[\int_{\R^d} \int_{\Pi_T} 
\int_{|z|> r} \frac{|u_{\eps}(t,y) -u_{\eps}(t,y+z)|}{|z|^{d + 2 \lambda}} \,dz\, \varrho_{\delta} (x-y)\,dx\,dt \,dy \Big] \\
& \le  C \eps\, \int_{0}^T \int_{|z|> r} |z|^{-d - 2 \lambda} \, \E \big[\norm{u_{\eps}(t,\cdot) -u_{\eps}(t,\cdot+z)}_{L^1(\R^d)} \big]\,dz\,dt.
\end{align*}
Next, observe that by Theorem \ref{thm:bv-viscous}
\begin{align*}
\E \big[\norm{u_{\eps}(t,\cdot) -u_{\eps}(t,\cdot+z)}_{L^1(\R^d)} \big] & \le C \norm{u_0}_{L^1(\R^d)}, \,\text{and}\\
\E \big[\norm{u_{\eps}(t,\cdot) -u_{\eps}(t,\cdot+z)}_{L^1(\R^d)} \big] & \le \E \big[|u_{\eps}(t, \cdot)|_{BV} \big]\, |z| \le |u_0|_{BV}\, |z|,
\, \text{for all}\, t>0.
\end{align*}
Now we cut the above integral term in two pieces according to $|z| > \sigma$ or not, where $\sigma$ is a positive parameter we will fix later. In what follows, we use both last estimates on each of one part, respectively. The result is 
\begin{align*}
|\mathcal{H}^r| &\le C |u_0|_{BV}\, \eps\, T\, \int_{|z|> \sigma} |z|^{-d - 2 \lambda}\,dz + C |u_0|_{BV}\, \eps\, T\, \int_{|z|\le \sigma} |z|^{-d - 2 \lambda +1}\,dz \\
& = C\, \eps\,\sigma^{-2 \lambda} + C\, \eps\, \int_{r}^{\sigma} \frac{d \tau}{\tau^{2\lambda}} 
\end{align*}

For the other term, we follow the same argument as in \cite{Alibaud} to conclude
\begin{align*}
|\mathcal{H}_r| &=  \Big| - \eps \,\E \Big[\int_{\R^d} \int_{\Pi_T}  |u_\eps(t,y) - u(t,x)| \,\mathcal{L}_{\lambda,r} \big[ \varrho_{\delta} (x -\cdot)\big](y)\,dx\,dt \,dy \Big] \Big| \\
& = \Big| -\eps \,\E \bigg[\int_0^1 \int_{\R^d} \int_{\Pi_T} \int_{|z|\le r} |z|^{-d -2\lambda} \, |u_\eps(t,y) - u(t,x)| \, D^2 \varrho_{\delta} (x-y-\tau z)z.z \,dz\,dx\,dt \,dy \,\d\tau \bigg]  \Big|\\
& = \Big| - \eps  \,\E \bigg[ \int_0^1 \int_{\R^d} \int_{\Pi_T} \int_{|z|\le r} |z|^{-d -2\lambda} \, \nabla \varrho_{\delta} (x-y-\tau z).z\, d\big(D(|u_\eps(t,\cdot) - u(t,x)|)\big)(y).z \,dz\,dx\,dt \,dy \,\d\tau \bigg] \Big| \\
& \le \eps\, \E \bigg[ \int_0^1 \int_{\R^d} \int_{\Pi_T} \int_{|z|\le r} |z|^{-d -2\lambda +2} \, \big|\nabla \varrho_{\delta} (x-y-\tau z) \big| \, d\big(|D u_\eps(t,\cdot)|\big)(y) \,dz\,dx\,dt \,dy \,\d\tau \bigg] \\
& \le C \, \eps\, \delta^{-1}\, \int_0^1 \int_0^T \int_{|z|\le r} |z|^{-d -2\lambda + 2} \,dz\,dt\,d\tau = C \, \eps\, \delta^{-1}\,r^{2 -2\lambda}.
\end{align*}
Keeping in mind the above estimates, we let $h \to 0$ in \eqref{eq:imp_01} to conclude, \textit{via} Gronwall's lemma
\begin{align}
 & \E \Big[\int_{\R^d}\big| u_\eps(t,x)-u(t,x)\big|\,dx\Big]\le Ce^{Ct}\Big(\delta + \eps \int_r^{\sigma} \frac{d\tau}{\tau^{2\lambda}} + \frac{\eps}{\delta} r^{2-2\lambda} + \eps \sigma^{-2\lambda}\Big), \label{esti:final}
\end{align}
where we have used the fact that $u_\eps(t,y)$ and $u(t,x)$ satisfy spatial BV bound, bounded by the BV norm of $u_0$. Finally, to optimize the right hand side of \eqref{esti:final}, we divide it into three cases:
\begin{itemize}
\item [(a)] For $0 < \lambda <1/2$, we first let $r\goto0$, then $\delta\goto 0$ and $\sigma=1$. 
\item [(b)] For $\lambda =1/2$, we take $r = \eps = \delta$, and $\sigma=1$.
\item [(c)] For $1/2 < \lambda <1$, we choose $\delta = r = \eps^{1/2\lambda}$, $\sigma\to\infty$.
\end{itemize}
With these above choices, we conclude that, for a.e. $t>0$
\begin{align*}
\E\Big[\|u_\eps(t,\cdot)-u(t,\cdot)\|_{L^1(\R^d)}\Big] =Ce^{Ct}
\begin{cases}
\mathcal{O}(\eps), & \quad \text{if}\,\, 0 < \lambda <1/2, \\
\mathcal{O}(\eps \,|\log(\eps)|), & \quad \text{if}\,\, \lambda =1/2, \\
\mathcal{O}(\eps^{1/2\lambda}), & \quad \text{if}\,\, 1/2 < \lambda <1,
\end{cases}
\end{align*}
where $C>0$ is a constant depending only on  
$|u_0|_{BV(\R^d)}, \|f^{\prime\prime}\|_{\infty}$, and $ \|f^\prime\|_{\infty}$. 
This completes the proof.


\section{Proof of Theorem~\ref{prop:vanishing viscosity-solution}: Existence of Viscous Solution}
\label{viscous}

Just like the deterministic counterpart, here also we study the corresponding regularized problem by adding a 
small diffusion operator and derive some \textit{a priori} bounds. We remark that, for the case $1/2 < \lambda <1$, 
$\mathcal{L}_{\lambda}[u]$ is the dominant term and hence the equation \eqref{eq:stoc_con_brown} is a stochastic
parabolic equation. Therefore the existence of the solution to \eqref{eq:stoc_con_brown} can be obtained by a fixed
point or contraction mapping argument. However, we present a proof of existence of solution which works for all 
$\lambda \in (0,1)$. In what follows, we consider the following regularized equation
\begin{align}
 du_\eps(t,x)  + \fr[u_\eps(t,x)]\,dt & - \mbox{div}_x f(u_\eps(t,x)) \,dt = \notag \\
& \eps \Delta u_{\eps}(t,x) \,dt  + \sigma(u_\eps(t,x))\,dW(t) + \int_{E} \eta(u_\eps(t,x);z)\widetilde{N}(dz,dt),
\label{eq:viscous-Brown_one} 
\end{align}
with a regular initial data $u_\epsilon(0,\cdot) = u_0^{\eps}\in H^1(\D)$ such that $u_0^{\eps} \underset{\eps\to0} \longrightarrow u_0$ in $L^2(\D)$ and $\eps \|\nabla u^\eps_0\|^2_{L^2(\R^d)} \leq \|u_0\|_{L^2(\R^d)}$\footnote{\textit{e.g.} $u_0^\eps$ solution to $u-\eps \Delta u=u_0$}.

As alluded to before, our aim is to prove the existence and uniqueness of $u_\eps$, as an element of $L^\infty(0,T;L^2(\Omega,H^1(\R^d)))\cap L^2(\Omega\times(0,T);H^2(\R^d))$ such that $\partial_t(u_\eps - \int^t_0 \sigma(u_\eps) \,dW -\int^t_0 \int_E \eta( u_\eps;z) \widetilde{N}(\,dz,\,dt)) \in L^2((0,T) \times \Omega ; L^2(\Rd))$.

\subsection{Existence}
We propose a proof of the existence of solutions to \eqref{eq:viscous-Brown_one} based on a semi-implicit time discretization. The description of the implicit scheme reads as follows:

For a given small positive parameter $\Delta t$ and $u_n$ in $L^2(\Omega, H^1(\R^d))$, $\mathcal{F}_{n\Delta t}$ measurable, find $u_{n+1}$ in $L^2(\Omega, H^1(\R^d))$, $\mathcal{F}_{(n+1)\Delta t}$ measurable, such that $\mathbb{P}$-a.s and for any $v \in H^1(\R^d)$
\begin{align}
&\intrd \Bigg((u_{n+1}-u_n)\,v + \Delta t \,\bigg\{ \ep \nabla u_{n+1} \cdot \nabla v + f(u_{n+1})\cdot \nabla v \bigg\} \Bigg) \,dx+\Delta t \, \Bigl< \fr[u_{n+1}] , v \Bigr> \notag \\
&\qquad \qquad = (W_{n+1} -W_n) \intrd \sigma(u_n) \,v \,dx + \intrd \int_{t_n}^{t_{n+1}} \int_{E} \eta(u_n;z) \,v \,\widetilde{N}(dz, ds)\,dx, \label{eq:discrete}
\end{align}
where $t_n = n\,\Delta t$, for $n=0,1,2, \cdots, N$ and $W_n := W(n\Delta t)$.


\begin{prop}
If $\Delta t < \frac{2\ep}{\| f'\|^2_{\infty}}$, such a sequence $(u_n)$, defined by \eqref{eq:discrete}, exists.
\end{prop}

\begin{proof}
A simple application of Lax-Milgram theorem, contraction mapping theorem along with simple properties of fractional operators yield the existence and uniqueness of $u^{n+1}$, the stochastic measurability is given by the stability of the solution with respect to the right-hand side. For more details, consult \cite{BaVaWit_2012}.
\end{proof}

\noindent Setting the test function $u_{n+1}$ in \eqref{eq:discrete} reveals, thanks in particular to the property of independence of the noises, that
\begin{align}
&\frac{1}{2}\E \Big[\intrd \big[ | u_{n+1}|^2-|u_n|^2+ |u_{n+1}-u_n|^2 \big] \,dx \Big]+ \Delta t \,\ep \,\E \Big[\intrd | \nabla u_{n+1}|^2 \,dx \Big] + \Delta t \,\E \Big[\intrd f(u_{n+1}) \cdot \nabla u_{n+1} \,dx \Big]\notag\\
&+ \Delta t \, \E \Big[\big< \fr[u_{n+1}], u_{n+1} \big> \Big] =\E \Big[ \intrd \Big\{ (W_{n+1}-W_n)\sigma(u_n) +\int_{t_n}^{t_{n+1}} \int_{E} \eta(u_n;z)  \,\widetilde{N}(dz, ds)\Big\}\,u_n \,dx \Big]\notag\\
&+\E \Big[ \intrd \Big\{ (W_{n+1}-W_n)\sigma(u_n) +\int_{t_n}^{t_{n+1}} \int_{E} \eta(u_n;z)  \,\widetilde{N}(dz, ds)\Big\}\,(u_{n+1}-u_n) \,dx \Big]\notag\\
&\le   \frac{\alpha}{2} \E \Big[ | u_{n+1}-u_n|^2 \Big] + \frac{1}{2\alpha}
\E \Big[ \big\| (W_{n+1}-W_n)\sigma(u_n) +\int_{t_n}^{t_{n+1}} \int_{E} \eta(u_n;z)  \,\widetilde{N}(dz, ds)\big\|^2 \Big]
 \notag\\
&= \frac{\alpha}{2} \E \Big[| u_{n+1} -u_n |^2 \Big] + \frac{\Delta t}{2 \alpha} \E\big[\|\sigma(u_n)\|^2 \big]+ \frac{1 }{2 \alpha} \E\Big[\intrd \Big(\int_{t_n}^{t_{n+1}} \int_{|z|>0} \eta(u_n;z) \,\widetilde{N}(dz, ds) \Big)^2\,dx \Big] \notag\\
&= \frac{\alpha}{2} \E \Big[| u_{n+1} -u_n |^2 \Big] + \frac{\Delta t}{2 \alpha} \E\big[\|\sigma(u_n)\|^2 \big]+ \frac{1 }{2 \alpha} \E\Big[\intrd \int_{t_n}^{t_{n+1}} \int_{|z|>0} |\eta(u_n;z)|^2 \,m(dz) ds \,dx \Big].\label{test}
\end{align}
We now sum both sides over $n$, and choose $\alpha >0$ so that 
\begin{align*}
\E \big[\| u_n\|^2\big] + \sum\limits^{n-1}_{k=0} \E \big[\| u_{k+1}-u_k \|^2 \big] + \eps
\Delta t \sum\limits^{n-1}_{k=0} \E \big[\| \nabla u_{k+1}\|^2 \big]+ & \Delta t\sum\limits^{n-1}_{k=0} \E \Big[\big< \fr[u_{n+1}], u_{n+1} \big> \Big] \\
& \qquad \qquad \le C + \widetilde{C} \Delta t \sum\limits^{n-1}_{k=0} \E \big[\| u_k\|^2 \big]
\end{align*}
Applying Discrete Gronwall's lemma we get 
\begin{align}
\E \big[\| u_n\|^2\big] + \sum\limits^{n-1}_{k=0} \E \big[\| u_{k+1}-u_k \|^2 \big] + \eps
\Delta t \sum\limits^{n-1}_{k=0} \E \big[\| \nabla u_{k+1}\|^2 \big] +\Delta t\sum\limits^{n-1}_{k=0} \E \Big[\big\| \frt[u_{k+1}]\big\|^2 \Big]\le C\label{Gronwall}
\end{align}
To proceed further, we need to introduce the following classical notations:
\begin{align*}
&u^{\Delta t}= \sum\limits^N_{k=1} u_k \mathds{1}_{[(k-1) \Delta t, k \Delta t)},\\
&\widetilde{u}^{\Delta t} = \sum\limits^N_{k=1} \bigg[ \frac{u_k -u_{k-1}}{\Delta t}[t-(k-1)\Delta t] +u_{k-1} \bigg] \mathds{1}_{[(k-1) \Delta t, k \Delta t)},
\end{align*}
with $u^{\Delta t}(t)= u_0$ for $t<0$. 
Then from \eqref{Gronwall}, it is clear that $u^{\Delta t} , \widetilde{u}^{\Delta t}$ are bounded sequences in $L^{\infty}((0,T); L^2(\Omega \times \Rd)) $, $\sqrt{\ep} u^{\Delta t}$ is a bounded sequence in $L^2((0,T); L^2( \Omega ; H^1(\Rd)))$ and $u^{\Delta t}$ is a bounded sequence in $L^2((0,T); L^2( \Omega ; H^{\lambda}(\Rd)))$.

We now set the test function $v = u_{n+1} -u_n -(W_{n+1}-W_n)\sigma(u_n)- \int^{t_{n+1}}_{t_n}\int_E \eta(u_n ;z) \widetilde{N}( dz,ds)$ in \eqref{eq:discrete} to get 
\begin{align}
\nonumber
&\E \Big[\big\| u_{n+1} -u_n -(W_{n+1}-W_n) \sigma(u_n)- \int^{t_{n+1}}_{t_n}\int_E \eta(u_n ;z) \widetilde{N}( dz,ds) \big\|^2_{L^2(\Rd)} \Big]\\
&\qquad+\eps \Delta t \E\Big[\intrd \nabla u_{n+1} \cdot \nabla\Big(u_{n+1}-u_n -(W_{n+1}-W_n)\sigma(u_n)-\int^{t_{n+1}}_{t_n}\int_E \eta(u_n ;z) \widetilde{N}( dz,ds)\Big)\,dx \Big]\label{testone}\\
&\qquad+\Delta t  \E\Big[\Big\langle \fr u_{n+1}, u_{n+1} -u_n -(W_{n+1}-W_n)\sigma(u_n)-\int^{t_{n+1}}_{t_n}\int_E \eta(u_n ;z) \widetilde{N}( dz,ds) \Big\rangle \Big]\label{testtwo}\\
&\qquad= \Delta t \E\Big[\intrd\Big( u_{n+1}-u_n-(W_{n+1}-W_n)\sigma(u_n)-\int^{t_{n+1}}_{t_n}\int_E \eta(u_n ;z) \widetilde{N}( dz,ds)\Big) f'(u_{n+1})\cdot \nabla_{n+1} \,dx \Big]\nonumber\\\nonumber
& \le \frac{1}{2} \E \Big[\big\| u_{n+1} -u_n -(W_{n+1}-W_n) \sigma(u_n)-\int^{t_{n+1}}_{t_n}\int_E \eta(u_n ;z) \widetilde{N}( dz,ds) \big\|^2 \Big]+ C(f')(\Delta t)^2 \E\Big[\| \nabla u_{n+1}\|^2 \Big]
\end{align}
Next, our aim is to estimate the terms \eqref{testone} and \eqref{testtwo}. In fact, using elementary properties of
It\^{o}-L\'evy integral, we can rewrite \eqref{testone} as
\begin{align*}
& \eps \Delta t \E\Big[\intrd \nabla u_{n+1} \cdot \nabla\Big(u_{n+1}-u_n -(W_{n+1}-W_n)\sigma(u_n)-\int^{t_{n+1}}_{t_n}\int_E \eta(u_n ;z) \widetilde{N}( dz,ds)\Big)\,dx \Big]\\
&\qquad =\eps \Delta t  \E \Big[\intrd \Big( \nabla u_{n+1} \cdot \nabla ( u_{n+1}-u_n) -\nabla(u_{n+1}-u_n)\cdot \nabla \sigma(u_n)(W_{n+1}-W_n)\\
& \hspace{6cm} -\nabla(u_{n+1}-u_n) \cdot\int^{t_{n+1}}_{t_n}\int_E \nabla \eta(u_n ;z) \widetilde{N}( dz,ds) \Big) \,dx \Big]\\
&\qquad = \frac12 \eps \Delta t \E \bigg[ \| \nabla u_{n+1} \|^2_{L^2(\Rd)^d}+ \| \nabla u_{n+1} -\nabla u_n\|^2_{L^2(\Rd)^d}-\| \nabla u_n \|^2_{L^2(\Rd)^d}\bigg]\\
&\hspace{2cm} - \Delta t \ep \E \bigg[  \intrd (W_{n+1}-W_n)\nabla( u_{n+1}-u_n) \cdot \nabla \sigma(u_n) \,dx\bigg] 
\\ &\hspace{2.5cm} -\Delta t \ep \E \bigg[ \intrd \nabla(u_{n+1}-u_n) \cdot\int^{t_{n+1}}_{t_n}\int_E \nabla \eta(u_n ;z) \widetilde{N}( dz,ds) \bigg] \,dx\\
\end{align*}
Notice that 
\begin{align*}
&- \eps \Delta t  \E \Big[  \intrd (W_{n+1}-W_n)\nabla( u_{n+1}-u_n) \cdot \nabla \sigma(u_n) \,dx\Big]\\
&\hspace{3cm} \ge -\frac{\eps \Delta t }{8} \E \Big[\| \nabla( u_{n+1} -u_n) \|^2 \Big] -2 \eps (\Delta t)^2 \E \Big[\| \sigma(\nabla u_n)\|^2 \Big]\\
&\hspace{5cm} \ge -\frac{\eps \Delta t }{8} \E \Big[\| \nabla( u_{n+1} -u_n) \|^2 \Big] -2K \eps (\Delta t)^2 \E \Big[\| \nabla u_n\|^2 \Big],
\end{align*}
and
\begin{align*}
&-\eps \Delta t  \E \bigg[ \intrd \nabla(u_{n+1}-u_n) \cdot\int^{t_{n+1}}_{t_n}\int_E \nabla \eta(u_n ;z) \widetilde{N}( dz,ds) \bigg] \,dx\\
& \qquad \ge -\frac{\eps \Delta t }{8} \E \Big[\| \nabla( u_{n+1} -u_n) \|^2 \Big]   -2 \ep \Delta t \E \Big[\big\| \int^{t_{n+1}}_{t_n}\int_E \nabla \eta(u_n ;z) \widetilde{N}( dz,ds) \big\|^2 \Big]\\
&\qquad= -\frac{\eps \Delta t }{8} \E \Big[ \| \nabla( u_{n+1} -u_n) \|^2 \Big]  -2 \ep \Delta t \E \Big[\intrd \int^{t_{n+1}}_{t_n}\int_E |\nabla \eta(u_n ;z)|^2 m(dz) \,ds \,dx \Big]
\\
&\qquad \ge -\frac{\eps \Delta t }{8} \E \Big[ \| \nabla( u_{n+1} -u_n) \|^2 \Big]  -2 (\lambda^*)^2 \|h\|^2_{L^2(E)}\ep (\Delta t)^2  \E \Big[\| \nabla u_n\|^2 \Big] 
\end{align*}

Thus, keeping the above estimates in mind, \eqref{testone} can be estimated, for a positive constant $C$, as
\begin{align*}
&\Delta t \ep \E \Big[\intrd \nabla u_{n+1} \cdot \nabla\Big(u_{n+1}-u_n -(W_{n+1}-W_n)\sigma(u_n)-\int^{t_{n+1}}_{t_n}\int_E \eta(u_n ;z) \widetilde{N}( dz,ds)\Big)\,dx \Big]\\
&\qquad \ge \frac{\Delta \ep}{2}\E \bigg[ \| \nabla u_{n+1} \|^2 + \frac{1}{2} \| \nabla( u_{n+1} -u_n) \|^2 -\| \nabla u_n \|^2 - C \Delta t \| \nabla u_n \|^2  \bigg]
\end{align*}
We now consider \eqref{testtwo} and recast it as follows:
\begin{align*}
&\Delta t \E \Big[\Big\langle \fr u_{n+1}, u_{n+1} -u_n -(W_{n+1}-W_n)\sigma(u_n)-\int^{t_{n+1}}_{t_n}\int_E \eta(u_n ;z) \widetilde{N}( dz,ds) \Big\rangle \Big]\\
&\qquad= \Delta t \E \Big[\big\langle \fr u_{n+1}, u_{n+1} -u_n \big\rangle - \big\langle \fr( u_{n+1} -u_n),(W_{n+1} -W_n) \sigma(u_n) \big\rangle \\
&\hspace{5cm}-\Big\langle \fr(u_{n+1}-u_n) , \int^{t_{n+1}}_{t_n} \int_E \eta( u_n; z) \widetilde{N}(dz ,dx) \Big\rangle \Big]\\
&\qquad = \frac{\Delta t}{2} \E \Big[ \| \frt u_{n+1} \|^2+ \| \frt( u_{n+1}-u_n) \|^2 - \| \frt u_n\|^2 \Big] \\
&\hspace{5cm}- \Delta t \E \Big[ \big\langle \fr( u_{n+1} -u_n) , (W_{n+1} -W_n) \sigma(u_n) \big\rangle \Big]\\
& \hspace{6cm}- \Delta t \E \Big[\Big\langle \fr(u_{n+1}-u_n) , \int^{t_{n+1}}_{t_n} \int_E \eta( u_n; z) \widetilde{N}(dz ,dx) \Big\rangle \Big]
\end{align*}
Notice that similarly
\begin{align*}
&- \Delta t \E \Big[ \big \langle \fr( u_{n+1} -u_n) , (W_{n+1} -W_n) \sigma(u_n) \big \rangle \Big]\\
&\hspace{4cm} \ge -\frac{\Delta t}{8} \E \Big[\| \frt ( u_{n+1} -u_n) \|^2 \Big] -2K( \Delta t)^2 \E \Big[\| \frt u_n \|^2 \Big]
\end{align*}
and
\begin{align*}
&- \Delta t \E \Big[\Big\langle \fr(u_{n+1}-u_n) , \int^{t_{n+1}}_{t_n} \int_E \eta( u_n; z) \widetilde{N}(dz ,dx) \Big\rangle \Big]\\
& \qquad \ge -\frac{\Delta t}{8} \E \Big[\| \frt ( u_{n+1} -u_n) \|^2 \Big]-2\Delta t \E \Big[ \Big\| \frt \int^{t_{n+1}}_{t_n} \int_E \eta(u_n; z) \widetilde{N}(dz ,ds) \Big\|^2 \Big]\\
&\qquad \quad \ge -\frac{\Delta t}{8} \E \Big[\| \frt ( u_{n+1} -u_n) \|^2 \Big]-2\Delta t \E \Big[\intrd \int^{t_{n+1}}_{t_n} \int_E | \frt \eta( u_n ;z)|^2 m(dz) \,ds \,dx \Big]
\\
&\qquad \quad \ge -\frac{\Delta t}{8} \E \Big[\| \frt ( u_{n+1} -u_n) \|^2 \Big]-2(\lambda^*)^2 \|h\|^2_{L^2(E)}(\Delta t)^2 \E \Big[\| \frt u_n \|^2  \Big].
\end{align*}

Thus \eqref{testtwo} can be estimated, for a given positive constant $C$, as  
\begin{align*}
&\E\Delta t \Big[\Big\langle \fr u_{n+1}, u_{n+1} -u_n -(W_{n+1}-W_n)\sigma(u_n)-\int^{t_{n+1}}_{t_n}\int_E \eta(u_n ;z) \widetilde{N}( dz,ds) \Big\rangle \Big]\\
&\qquad \qquad \ge \frac{\Delta t}{2} \E \Big[ \| \frt u_{n+1} \|^2+ \| \frt( u_{n+1}-u_n) \|^2 - \| \frt u_n\|^2 -C \Delta t \| \frt u_n \|^2\Big]
\end{align*}
Putting back these estimates we get 
\begin{align*}
&\E \Big[\big\| u_{n+1} -u_n -(W_{n+1}-W_n) \sigma(u_n) -\int^{t_{n+1}}_{t_n}\int_E \eta(u_n ;z) \widetilde{N}( dz,ds) \big\|^2\Big]\\
&\qquad + \eps \Delta t  \E \Big[ \| \nabla u_{n+1}\|^2 - \| \nabla u_n \|^2 + \frac{1}{2}\| \nabla u_{n+1} -\nabla u_n\|^2 \Big]\\
&\qquad \quad +\Delta t\E \Big[ \big\|\frt u_{n+1} \big\|^2  - \big\|\frt u_n\big\|^2 +\frac{1}{2}\| \frt (u_{n+1}-u_n)\big\|^2 \Big]\\[1.5mm]
&\le \big[C(f')+C\ep \big](\Delta t)^2 \E \Big[\| \nabla u_{n+1}\|^2\Big]+ C(\Delta t)^2 \E \Big[\| \frt u_n\|^2\Big].
\end{align*}
Divide both sides by $\Delta t$ and sum over $n$ to yield
\begin{align*}
&\Delta t \sum\limits^k_{n=0}  \E \bigg[\Big\| \frac{u_{n+1}-u_n-(W_{n+1}-W_n)\sigma(u_n)-\int^{t_{n+1}}_{t_n}\int_E \eta(u_n ;z) \widetilde{N}( dz,ds)}{\Delta t}\Big\|^2 \bigg]+ \ep \E \Big[\| \nabla u_{k+1}\|^2 \Big]\\
&\qquad +\frac{\ep}{2} \E \Big[\sum\limits^k_{n=0} \| \nabla( u_{n+1}-u_n) \|^2\Big]+\E \Big[\big\|\frt[ u_{k+1}] \big\|^2\Big] +\frac{1}{2}\sum\limits^k_{n=0} \E \Big[\big\|\frt[(u_{k+1}-u_k)]\big\|^2\Big]\\
&\le C\Delta t \sum\limits^{k+1}_{n=0} \E \Big[\| \nabla u_{n}\|^2\Big]+C\Delta t \sum\limits^{k+1}_{n=0}\E \Big[\big\|\frt[ u_n]\big\|^2\Big]+ \ep \E \Big[\|\nabla u_0\|^2\Big]+\E \Big[\big\|\frt[ u_0]\big\|^2\Big].
\end{align*}

Then, using \eqref{Gronwall}, we get 
\begin{align}
&\sum\limits^k_{n=0} \Delta t \E \bigg[ \Big\| \frac{u_{n+1}-u_n-(W_{n+1}-W_n)\sigma(u_n)-\int^{t_{n+1}}_{t_n}\int_E \eta(u_n ;z) \widetilde{N}( dz,ds)}{\Delta t}\Big\|^2\bigg] + \ep \E \Big[\| \nabla u_{k+1}\|^2\Big] \notag\\
&\quad +\frac{\ep}{2} \sum\limits^k_{n=0} \E \Big[\| \nabla( u_{n+1}-u_n) \|^2\Big]+\E \Big[\big\|\frt[ u_{k+1}] \big\|^2\Big]+\frac{1}{2}\sum\limits^k_{n=0} \E \Big[\big\|\frt[(u_{k+1}-u_k)]\big\|^2\Big]  \le C. \label{bounds}
\end{align}
To pass to the limit, we define 
\begin{align*}
&B_n := B^{(1)}_n+B^{(2)}_n:= \sum\limits^{n-1}_{k=0} (W_{k+1}-W_k)\sigma(u_k)+ \sum\limits^{n-1}_{k=0}\int^{t_{n+1}}_{t_n}\int_E \eta(u_n ;z) \widetilde{N}( dz,ds)\\
&= \int^{n \Delta t}_0 \sigma(u^{\Delta t}(s- \Delta t)) \,dW(s)+ \int^{n \Delta t}_0 \int_E \eta( u^{\Delta t}(s-\Delta t);z) \widetilde{N}( dz,ds) 
\end{align*}
and 
\begin{align*}
&\widetilde{B}^{\Delta t}_i := \sum\limits^N_{k=1} \bigg[ \frac{B^{(i)}_k -B^{(i)}_{k-1}}{\Delta t}[t-(k-1)\Delta t] +B^{(i)}_{k-1} \bigg] \mathds{1}_{[(k-1) \Delta t, k \Delta t)},\\[1.5mm]
& \widetilde{B}^{\Delta t} := \tilde{B}^{\Delta t}_1 +\tilde{B}^{\Delta t}_2.
\end{align*}
Thanks to \eqref{bounds}, we conclude that $\partial_t[\tu-\bu]$ is bounded in $L^2((0,T); L^2( \Omega ; L^2(\Rd)))$,
$u^{\Delta t}$ and $\tu$ are bounded in $L^{\infty}(0,T; L^2( \Omega ; H^1(\Rd)))$, and $\tu - u^{\Delta t}$ converges 
to zero in $L^2((0,T); L^2( \Omega ; H^1(\Rd)))$. This implies the existence of a limit function 
$u \in L^{\infty}(0,T; L^2( \Omega ; H^1(\Rd)))$ such that 
\begin{align*}
u^{\Delta t} \ws u \text{ in }  L^{\infty}(0,T; L^2( \Omega ; H^1(\Rd))), \\
 \tu \ws u \text{ in }  L^{\infty}(0,T; L^2( \Omega ; H^1(\Rd))).
\end{align*}
Moreover, by assumptions \eqref{A3}, \eqref{A4} and \eqref{A5}, there exist 
$\sigma_u \in L^2(0,T; L^2( \Omega ; HS(H^1)))$, $f_u \in L^{2}(0,T; L^2( \Omega ; H^1(\Rd)))$ and $\eta_u \in L^2((0,T)\times E; L^2( \Omega ; L^2(\Rd)))$ such that
\begin{align*}
&\sigma(u^{\Delta t}) \wc \sigma_u \text{ in }  L^2(0,T; L^2( \Omega ; HS(H^1))),\\
&f(u^{\Delta t}) \wc f_u \text{ in } L^2(0,T; L^2( \Omega ; H^1(\Rd))),\\
&\eta(u^{\Delta t}) \wc \eta_u \text{  in } L^2((0,T)\times E; L^2( \Omega ; L^2(\Rd))).
\end{align*}
Since $\tu -\bu$ converges weakly in $L^2( \Omega ; W(0,T))$ where 
$W(0,T) \equiv \{ w \in L^2(0,T; H^1(\Rd)) | \partial_t w \in L^2(0,T; H^{-1}(\Rd)) \}$ one gets that  $\tu -\bu $ converges weakly
in $L^2(\Omega; C([0,T]; L^2(\Rd)))$. Thus for any $t \in [0,T]$, $(\tu -\bu)(t)$ converges weakly in
$L^2(\Omega;L^2(\Rd))$.

Next, we want to find an upper bound for $\bu$. To that end, we have the following proposition whose proof can be found in \cite{BaVaWit_2012} and \cite{BisMajVal}. 
\begin{prop}
\label{prop01}
$\bu$ is a bounded sequence in $L^2(\Omega \times \Pi_T)$, and
\begin{align*}
&\Big\|\bu_1 (\cdot)- \int^{\cdot}_0 \sigma( u^{\Delta t}(s- \Delta t)) \,dW(s)\Big\|^2_{L^2(\Omega \times \Rd)} \le C \Delta t\\
& \Big\| \bu_2(\cdot) - \int^{\cdot}_0 \int_E \eta( u^{\Delta t}(s-\Delta t);z)   \widetilde{N}( dz,ds)\Big\|^2_{L^2(\Omega \times \Rd)} \le C\Delta t
\end{align*}
\end{prop}
\noindent Thus, in view of the above Proposition~\ref{prop01}, we conclude 
\begin{align*}
\| \bu - B_n\|^2_{L^2(\Omega \times \Rd)} & \le 2 \bigg(\Big\|\bu_1 (\cdot)- \int^{\cdot}_0 \sigma( u^{\Delta t}(s- \Delta t)) \,dW(s)\Big\|^2_{L^2(\Omega \times \Rd)}\\
& \qquad \qquad+ \Big\| \bu_2(\cdot) - \int^{\cdot}_0 \int_E \eta( u^{\Delta t}(s-\Delta t);z) \widetilde{N}( dz,ds)\Big\|^2_{L^2(\Omega \times \Rd)}\bigg) \le C \Delta t
\end{align*}
Now we shall identify the weak limit of $B_n$.
As $\sigma( u^{\Delta t}) \wc \sigma_u$ in the space $L^2(0,T;L^2(\Omega; HS(L^2)))$, using the fact that the It\^{o} integral is a
linear continuous operator between the spaces $L^2((0,T); L^2(\Omega \times \Rd))$ and $L^2(\Omega ; C([0,T],L^2(\Rd))$, it
preserves the weak convergence and thus $\int_0^t \sigma(u^{\Delta t}(s-\Delta t) \,dW(s) \wc \int^t_0 \sigma_u\, dW(s)$ in
$L^2(\Omega ; L^2(\Rd))$ for any $t$.

Similarly $\int^t_0 \int_E \eta( u^{\Delta t}(s-\Delta t);z) \widetilde{N}( dz,ds) \wc \int^t_0 \int_E \eta_u \widetilde
{N}(dz,ds)$ in $L^2(\Omega ; L^2(\Rd))$. Thus we get that $\tu(t) -\bu(t)$ converges weakly towards 
$u(t) - \int^t_0 \sigma_u\,dW(s) -\int^t_0 \int_E \eta_u \widetilde{N}(dz,ds)$, and $\tu(t)$ converges weakly to $u(t)$ in
$L^2( \Omega ; L^2(\Rd))$. This also implies that
$\partial_t(\tu(t) -\bu(t)) \wc \partial_t(u(t) - \int^t_0 \sigma_u \,dW(s)-\int^t_0 \int_E \eta_u \widetilde{N}( dz,ds))$ in
the space $ L^2( \Omega ; L^2(\Rd))$.
  
\noindent Moreover, for any $v \in H^1(\Rd)$, from \eqref{eq:discrete}, we have 
\begin{align*}
\intrd \partial_t(\tu(t) -\bu(t))v \,dx + \ep \intrd \nabla u^{\Delta t} \nabla v \,dx + \intrd f( u^{\Delta t}) \nabla v \,dx + \langle \fr[u^{\Delta t}] ,v \rangle =0,
\end{align*}
and at the limit as $\Delta t \to 0$, we get
\begin{align*}
\intrd\partial_t \bigg( u - \int^t_0 \sigma_u \,dW(s)-\int^t_0 \int_E \eta_u \widetilde{N}( dz,ds)\bigg) v \,dx + \ep \intrd \nabla u\nabla v \,dx+ \intrd f_u \nabla v \,dx + \big\langle \fr[u],v \big\rangle =0.
\end{align*}
Thus, 
\begin{align*}
 du-\ep  \Delta u\, dx - \Div  f_u  \,dx + \fr[u]\,dx =\sigma_u \,dW + \int_E \eta_u \widetilde{N}( dz,ds)
\end{align*}
and, applying the It\^{o}-L\'evy formula to the function $g(t,x)=e^{-ct} x^2$, then integrating over $\Rd$ and taking expectation, we get
\begin{align}
&e^{-ct} \E \big[\|u(t)\|^2\big] + 2\ep \int^t_0 e^{-cs} \E \Big[\| \nabla u \|^2 \Big]\,ds + 2 \int^t_0 \E \bigg[\intrd e^{-cs} f_u \nabla u \,dx \, \bigg]ds + 2 \int^t_0 e^{-cs} \E \Big[\langle \fr[u],u \rangle \Big]\,ds\notag\\
&\quad = \|u_0\|^2 -c \int^t_0 e^{-cs} \E \big[\| u(s) \|^2\big] \,ds + \int^t_0 e^{-cs} \E \big[\| \sigma_u\|^2\big] \,ds + \E \bigg[\int^t_0 \intrd \int_E e^{-cs} \eta^2_u \,\,m(dz) \,dx \,ds \bigg] \label{discrete1}
\end{align}
Choosing $\alpha =1$ in \eqref{test}, we notice that for any positive $c$ and $n \in \mathbb{N}$ we have  
\begin{align*}
&\E \bigg[\intrd \big[ e^{-cn\Delta t} | u_{n+1}|^2 -e^{c(n-1) \Delta t} | u_n|^2 \big] \,dx \bigg]+ 2 \Delta t e^{-nc \Delta t} \Bigg[\ep\,\E \bigg[\intrd | \nabla u_{n+1}|^2 \,dx \bigg]+ \Delta t \E \Big[\langle \fr[u_{n+1}],u_{n+1}\rangle \Big]\Bigg]\\
&\quad\le \Delta t e^{-cn \Delta t} \E\Big[ \|\sigma(u_n) \|^2\Big] + \Big[ e^{-cn\Delta t} -e^{-c(n-1)\Delta t}\Big] \E \Big[ |u_n|^2 \Big] 
+ e^{-cn \Delta t} \E\intrd \int_{t_n}^{t_{n+1}} \int_{|z|>0} |\eta(u_n;z)|^2 \,dz\, ds\,dx.
\end{align*}
Summing the index $n$ from $0$ to $k$ both sides, we get 
\begin{align*}
&e^{-ck \Delta t} \E \Big[\| u_{k+1}\|^2\Big] + 2\Delta t \ep \sum^k_{n=0} e^{-c n\Delta t} \E \Big[\| \nabla u_{n+1}\|^2\Big] + 2 \Delta t \sum^k_{n=0} e^{-cn \Delta t} \E \Big[\| \frt[u_{n+1}]\|^2 \Big] \\
\le& \|u_0\|^2 + \Delta t \sum^k_{n=0} e^{-cn\Delta t} \E\Big[ \| \sigma(u_n) \|^2 \Big]+ \sum^k_{n=0}e^{-cn \Delta t} \E\intrd \int_{t_n}^{t_{n+1}} \int_{|z|>0} |\eta(u_n;z)|^2 \,dz\, ds\,dx  
\\ &- c\Delta t \sum^k_{n=1} e^{-cn\Delta t} \E \Big[\|u_n\|^2\Big].
\end{align*}
Rewriting in terms of $u^{\Delta t}$ yields
\begin{align*}
&e^{-ck \Delta t} \E \Big[\| u_{k+1}\|^2 \Big]+ 2\ep \int^{(k+1)\Delta t}_0 e^{-cs} \E\Big[ \| \nabla u^{\Delta t} \|^2 \Big] \,ds+ 2 \int^{(k+1)\Delta t}_0 e^{-cs} \E \Big[\| \frt[u^{\Delta}] \|^2 \Big]\,ds \\
& \qquad \le \|u_0\|^2 +\Delta t \| \sigma(u_0)\|^2+\E\intrd \int_{0}^{\Delta t} \int_{|z|>0} |\eta(u_0;z)|^2 \,dz\, ds\,dx + \int^{k \Delta t }_0 e^{-cs} \E \Big[\| \sigma(u^{\Delta t})\|^2 \Big]\,ds \\
& \qquad +\E\intrd \int_{0}^{k\Delta t} e^{-cs}\int_{|z|>0} |\eta(u^{\Delta t};z)|^2 \,dz\, ds\,dx -ce^{-c \Delta t} \int^{kt}_0 e^{-cs} \E\Big[ \| u^{\Delta t}\|^2\Big] \,ds.
\end{align*}
Let $t \in [ k\Delta t, (k+1)\Delta t)$, then 
\begin{align*}
&e^{-ct} \E \Big[\| u^{\Delta t}(t)\|^2\Big] + 2\ep \int^{t}_0 e^{-cs} \E\Big[ \| \nabla u^{\Delta t} \|^2\Big] \,ds + 2 \int^{t}_0 e^{-cs} \E\Big[ \| \frt[u^{\Delta t}] \|^2 \Big]\,ds \\
&\le \|u_0\|^2 +\widetilde{C}\Delta t \|u_0\|^2+ \int^{t}_0 e^{-cs} \E \Big[\| \sigma(u^{\Delta t})\|^2\Big] \,ds\\
& \qquad +\E\intrd \int_{0}^{t} e^{-cs}\int_{|z|>0} |\eta(u^{\Delta t};z)|^2 \,dz\, ds\,dx -ce^{-c \Delta t} \int^{(t-\Delta t)^+}_0 e^{-cs} \E\Big[ \| u^{\Delta t}\|^2\Big] \,ds\\
&\le \|u_0\|^2 +C\Delta t + \int^{t}_0 e^{-cs} \E \Big[\| \sigma(u^{\Delta t})\|^2\Big] \,ds\\
& \qquad +\E\intrd \int_{0}^{t} e^{-cs}\int_{|z|>0} |\eta(u^{\Delta t};z)|^2 \,dz\, ds\,dx -ce^{-c \Delta t} \int^{t}_0 e^{-cs} \E\Big[ \| u^{\Delta t}\|^2\Big] \,ds.
\end{align*}
Replacing $u^{\Delta t}$ by $(u^{\Delta t}-u)$ in few terms on both sides  and subtracting extra terms we get,
\begin{align*}
&e^{-ct} \E \Big[\| u^{\Delta t}(t)\|^2\Big] + 2\ep \int^{t}_0 e^{-cs} \E \Big[\| \nabla (u^{\Delta t}-u) \|^2\Big] \,ds + 2 \int^{t}_0 e^{-cs} \E \Big[\big\| \frt[u^{\Delta t}-u] \big\|^2 \Big]\,ds\\
&\qquad+2 \int^t_0 e^{-cs} \E \bigg[\intrd ( f(u^{\Delta t})-f(u)) \nabla( u^{\Delta t}-u) \,dx \bigg]\,ds \\
&\qquad+4 \ep \int^t_0 e^{-cs} \E \bigg[\intrd \nabla u^{\Delta t} \nabla u \,dx\bigg] \,ds -2 \ep \int^t_0 e^{-cs} \E \Big[\| \nabla u \|^2\Big] \,ds\\
&\qquad +4 \int^t_0 e^{-cs} \E\bigg[ \intrd \frt u^{\Delta t} \frt u \,dx \bigg]\,ds -2 \int^t_0 e^{-cs} \E\Big[ \| \frt u\|^2\Big] \,ds \\
&\le \| u_0\|^2 +C \Delta t -2 \int^t_0 e^{-cs} \E\bigg[ \intrd f(u^{\Delta t} ) u \,dx \bigg]\,ds  -2 \int^t_0 e^{-cs} \E \bigg[\intrd f(u) \nabla u^{\Delta t} \,dx \bigg]\,ds \\
&\qquad- \int^t_0 e^{-cs} \E \Big[\| \sigma(u)\|^2\Big] \,ds
+\int^t_0 e^{-cs} \E \Big[\| \sigma(u^{\Delta t})-\sigma(u)\|^2\Big] \,ds + 2 \int^t_0 e^{-cs} \E \Big[\left( \sigma(u^{\Delta t}), \sigma(u) \right)\Big] \,ds \\
&\qquad -\E\intrd \int_{0}^{t} e^{-cs}\int_{|z|>0} |\eta(u;z)|^2 \,dz\, ds\,dx + \E\intrd \int_{0}^{t} e^{-cs}\int_{|z|>0} |\eta(u^{\Delta t};z)-\eta(u;z)|^2 \,dz\, ds\,dx  \\
&\qquad + 2\E\intrd \int_{0}^{t} e^{-cs}\int_{|z|>0} \eta(u^{\Delta t};z)\eta(u;z) \,dz\, ds\,dx
\\
&\qquad -ce^{-c \Delta t}\Big( \int^t_0 e^{-cs} \E \Big[\| u^{\Delta t}-u\|^2\Big] \,ds +2 \int^t_0 e^{-cs}\E \bigg[\intrd u^{\Delta t} u\,dx \bigg]\,ds -  \int^t_0 e^{-cs} \E \Big[\|u \|^2\Big] \,ds\Big).
\end{align*}
Next, observe that
\begin{align*}
& -2\ep \int^t_0 e^{-cs} \E \Big[\| \nabla (u^{\Delta t}-u) \|^2\Big] \,ds- 2 \int^t_0 e^{-cs} \E \bigg[\intrd ( f(u^{\Delta t})-f(u)) \nabla( u^{\Delta t}-u) \,dx\bigg] \,ds \\
&\qquad+\int^t_0 e^{-cs} \E \Big[\| \sigma(u^{\Delta t})-\sigma(u)\|^2\Big] \,ds- ce^{-c \Delta t} \int^t_0 e^{-cs} \E \Big[\| u^{\Delta t}-u\|^2\Big] \,ds\\
&+\E\intrd \int_{0}^{t} e^{-cs}\int_{|z|>0} |\eta(u^{\Delta t};z)-\eta(u;z)|^2 \,dz\, ds\,dx
\\
&\le -\ep\int^t_0 e^{-cs} \E\Big[ \| \nabla (u^{\Delta t}-u) \|^2\Big] \,ds +1/\ep \int^t_0 e^{-cs} \E\Big[\| f(u^{\Delta t}) -f(u) \|^2\Big] \,ds\\
&\qquad+ [(\lambda^*)^2\|h\|^2_{L^2(E)}+K- ce^{-c \Delta t}] \int^t_0 e^{-cs} \E \Big[\| u^{\Delta t}-u\|^2\Big] \,ds \\
& \qquad \le -\ep\int^t_0 e^{-cs} \E \Big[\| \nabla (u^{\Delta t}-u) \|^2\Big] \,ds + [(\lambda^*)^2\|h\|^2_{L^2(E)}+\frac{c(f)}{\epsilon} + K- ce^{-c \Delta t}] \int^t_0 e^{-cs} \E \Big[\| u^{\Delta t}-u\|^2\Big] \,ds.
\end{align*}
Putting back these terms with $c \geq (\lambda^*)^2\|h\|^2_{L^2(E)}+ \frac{c(f)}{\epsilon}+ K$ and integrating over $[0,T]$, we get 
\begin{align*}
&\int^T_0e^{-ct} \E \Big[\| u^{\Delta t}\|^2 \Big]\,dt + \ep\int^T_0 \int^{t}_0 e^{-cs} \E \Big[\| \nabla (u^{\Delta t}-u) \|^2 \Big]\,ds\,dt + 2\int^T_0 \int^{t}_0 e^{-cs} \E \Big[\big\| \frt[u^{\Delta t}-u] \big\|^2 \Big]\,ds \,dt\\
&\le T\| u_0\|^2 +C\Delta t -2 \int^T_0\int^t_0 e^{-cs} \E \bigg[\intrd f(u^{\Delta t}) \nabla u \,dx \bigg] \,ds \,dt -2\int^T_0\int^t_0 e^{-cs} \E \bigg[\intrd f(u) \nabla u^{\Delta t} \,dx\bigg] \,ds\,dt\\
&\qquad +2 \int_0^T\int^t_0 e^{-cs} \E \bigg[\left( \sigma(u^{\Delta t}), \sigma(u) \right)\bigg] \,ds\,dt -\E \int^T_0 \int^t_0 e^{-cs} \E \Big[\| \sigma(u) \|^2 \Big]\,ds \,dt \\
& \qquad 
+2\E\int_0^T\intrd \int_{0}^{t} e^{-cs}\int_{|z|>0} \eta(u^{\Delta t};z)\eta(u;z) \,dz\, ds\,dx\,dt-\E\int_0^T\intrd \int_{0}^{t} e^{-cs}\int_{|z|>0} |\eta(u;z)|^2 \,dz\, ds\,dx\,dt\\
& \qquad -  c e^{-c \Delta t}\Big(2\int^T_0 \int^t_0 e^{-cs} \E \bigg[\intrd u^{\Delta t} u \,dx\bigg] \,ds \,dt - \int^T_0 \int^t_0 e^{-cs} \E \Big[\| u\|^2\Big] \,ds \,dt \Big)
 +   \\
 &\qquad -4 \ep \int^T_0\int^t_0 e^{-cs} \E \bigg[\intrd \nabla u^{\Delta t} \nabla u \,dx\bigg] \,ds +2 \ep\int^T_0 \int^t_0 e^{-cs} \E \Big[\| \nabla u \|^2 \Big]\,ds\\
 &\qquad-4 \int^T_0\int^t_0 e^{-cs} \E \bigg[\intrd \frt u^{\Delta t} \frt u \,dx\bigg] \,ds +2 \int^T_0 \int^t_0 e^{-cs} \E \Big[\| \frt u\|^2 \Big]\,ds 
\end{align*}
Taking $\limsup_{\Delta t \to 0}$ both sides and using the fact that $\frt$ is a continuous linear operator from $H^1$ to
$L^2$, we get
\begin{align*}
&\limsup \int^T_0 e^{-ct} \E \Big[\| u^{\Delta t}(t)\|^2\Big] \,dt\\
&~~~~~\le \int^T_0 \bigg[ \| u_0\|^2  -2 \int^t_0 e^{-cs} \E\bigg[ \intrd f_u \nabla u \,dx\bigg] \,ds -2 \ep \int^t_0 e^{-cs} \E \Big[\| \nabla u\|^2 \Big]\,ds -2 \int^t_0 e^{-cs} \E \Big[\| \frt u\|^2\Big] \,ds \\
&~~~~~   -c \int_0^T\int^t_0  e^{-cs} \E \Big[\| u\|^2\Big] \,ds  \,dt
+2 \int^T_0\int^t_0 e^{-cs} \E \bigg[ \left( \sigma_u, \sigma(u) \right)\bigg] \,ds\,dt -\E \int^T_0 \int^t_0 e^{-cs} \E \Big[\| \sigma(u) \|^2 \Big] \,ds \,dt \\
&~~~~~ +2\E\intrd \int_0^T\int_{0}^{t} e^{-cs}\int_{|z|>0} \eta_u\eta(u;z) \,dz\, ds\,dx\,dt-\E\int_0^T\intrd \int_{0}^{t} e^{-cs}\int_{|z|>0} |\eta(u;z)|^2 \,dz\, ds\,dx\,dt.
\end{align*}
For terms inside the bracket we use \eqref{discrete1} to get 
\begin{align*}
&\limsup \int^T_0 e^{-ct} \E \Big[\| u^{\Delta t} \|^2 \Big]\,dt \\
& \leq \int^T_0 \bigg( e^{-ct} \E \Big[\| u(t) \|^2\Big] - \int^t_0 e^{-cs} \E \Big[\| \sigma_u\|^2\Big] \,ds  -\int^t_0 \E \Big[\intrd \int_E e^{-cs} \eta^2_u \,m(dz) \,dx \,ds\Big]\bigg) \,dt \\
&+2 \int^T_0\int^t_0 e^{-cs} \E \bigg[\left( \sigma_u, \sigma(u) \right) \bigg]\,ds\,dt  -\E \int^T_0 \int^t_0 e^{-cs} \E \Big[\| \sigma(u) \|^2\Big]  \,ds \,dt \\
& +2\E\int_0^T\intrd \int_{0}^{t} e^{-cs}\int_{|z|>0} \eta_u\eta(u;z) \,dz\, ds\,dx\,dt -\E\int_0^T\intrd \int_{0}^{t} e^{-cs}\int_{|z|>0} |\eta(u;z)|^2 \,dz\, ds\,dx\,dt, 
\end{align*}
and
\begin{align*}
&\limsup \int^T_0 e^{-ct} \E \Big[\| u^{\Delta t} \|^2\Big] \,dt + \int^T_0 \int^t_0 e^{-cs} \E \Big[\|\sigma_u -\sigma(u)\|^2\Big] \,ds \,dt \\
&\qquad + \E\intrd \int_0^T\int_{0}^{t} e^{-cs}\int_{|z|>0} |\eta_u-\eta(u;z)|^2 \,dz\, ds\,dt\,dx  
 \le \int^T_0 e^{-ct} \E \Big[\| u(t)\|^2\Big] \,dt,
\end{align*}
which in fact implies that $\sigma_u=\sigma(u)$, $\eta_u=\eta(u,\cdot)$ and $u^{\Delta t} \to u$ in $L^2((0,T) \times \Omega \times \Rd)$ and finally that $f_u =f(u)$.
Thus we conclude that $u$ is a solution.  We denote it by $u_\eps$.

\subsection{Uniqueness} 
For uniqueness, consider two solutions $u_1$ and $u_2$, denote by $u=u_1-u_2$, apply to $du$ the It\^{o}-L\'evy formula to the function $g(t,x)= x^2$ and take expectation to get 
\begin{align*}
&\E\| u(t) \|^2 + 2\E\int^t_0 \intrd  \big[ \ep | \nabla u |^2 + [f(u_1)-f(u_2)] \nabla u\big] \,dx \,ds+2 \E\int^t_0 \big\| \frt u \big\|^2 \,ds \\
& \qquad= \E\int^t_0 \intrd \|\sigma(u_1)-\sigma(u_2)\|^2 \,dx \,ds + \E\intrd \int^t_0 \int_E |\eta(u_1;z)-\eta(u_2;z)|^2 \,m(dz)\,ds\,dx
\end{align*}
Using Young's inequality for the flux-term and the Lipschitz information on the nonlinear functions, one has 
\begin{align*}
&\E \big[\| u(t) \|^2 \big]+ 2\E \Big[\int^t_0 \big[ \ep \| \nabla u \|^2 + \big\| \frt u \big\|^2 \big] \,ds\Big] \leq C \int_0^t \E \big[\| u(s) \|^2 \big]\,ds,
\end{align*} 
and the uniqueness of the solution by applying Gronwall's lemma. 

\subsection{Additional Regularity} 

\begin{cor}
The solution $u_\eps$ satisfies moreover $\Delta u_\eps \in L^2((0,T) \times \Omega ; L^2(\Rd))$.
\end{cor}

\begin{proof}
Since $ -\ep \Delta u_\eps + \fr[u_\eps] = \Div f(u_\eps) -\partial_t(u_\eps - \int^t_0 \sigma(u_\eps) \,dW -\int^t_0 \int_E \eta( u_\eps;z) \widetilde{N}(\,dz,\,dt))$, it is an element of $L^2((0,T) \times \Omega ; L^2(\Rd))$.
\\
Note that since $u-\eps \Delta u + \fr u =g \in L^2(\R^d)$, by taking Fourier transform, we get  
\begin{align*}
(1+\eps |\xi|^2+ |\xi|^{2\lambda})\hat{u} \in L^2(\R^d) \Longrightarrow C_\eps (1+|\xi|^2) \hat{u} \in L^2 (\R^d) \Longrightarrow u \in H^2(\R^d).
\end{align*}
\end{proof}


\subsection{A Priori Estimates}
\label{regularity}

Having established existence of solution, we are now ready to derive \textit{a priori} estimates of the solution.
First note that since for any $\eps>0$, the unique weak solution $u_\eps \in N_w^2(0,T,H^1(\R^d))$, so $\partial_t \big(u_\eps - \int_0^t \sigma(u_\eps(s,\cdot))\, dW(s) - \int_0^t \int_{E} \eta(u_\eps(s,\cdot);z)\widetilde{N}(dz,ds)\big)
\in L^2(\Omega\times(0,T),L^2(\R^d))$.

Next, denote by $\beta$ a non-negative, convex regular function with quadratic growth at infinity. Then, It\^o-L\'{e}vy formula applied to the viscous solution \eqref{eq:viscous-Brown} yields, for any bounded $C^1_b(Q)$ function $\psi$,    
\begin{align*}
0=&\ \int_\D\beta(u_\epsilon)\,\psi(t,x)\,dx - \int_\D\beta(u_0^\epsilon)\,\psi(0,x)\,dx - \int_0^t\int_\D \beta(u_\epsilon)\partial_t \psi \,dx\,ds + \epsilon \int_0^t\int_\D \beta^{\prime\prime}(u_\epsilon)|\nabla u_\epsilon|^2\psi\,dx\,ds \\
&+ \epsilon \int_0^t\int_\D \beta^\prime(u_\epsilon)\nabla u_\epsilon\nabla\psi\,dx\,ds 
+ \int_0^t\int_\D \psi\,\beta^{\prime\prime}(u_\epsilon) f(u_\epsilon)\nabla u_\epsilon \,dx\,ds +  \int_0^t\int_\D \beta^\prime(u_\epsilon) f(u_\epsilon)\nabla\psi\,dx\,ds \\
& + \sum_{k\ge 1}\int_0^t \int_{\D} g_k(u_\eps(t,x))\beta^\prime (u_\eps(t,x))\psi(t,x)\,d\beta_k(t)\,dx
 + \frac{1}{2}\int_0^t \int_{\D}\mathbb{G}^2(u_\eps(t,x))\beta^{\prime\prime} (u_\eps(t,x))\psi(t,x)\,dx\,dt \notag \\
  &  + \int_0^t \int_{\D}\int_{E} \int_0^1 \eta(u_\eps(t,x);z)\beta^\prime \big(u_\eps(t,x) + \lambda\,\eta(u_\eps(t,x);z)\big)\psi(t,x)\,d\lambda\,\widetilde{N}(dz,dt)\,dx  \notag \\
 & +\int_0^t \int_{\D} \int_{E}  \int_0^1  (1-\lambda)\eta^2(u_\eps(t,x);z)\beta^{\prime\prime} \big(u_\eps(t,x) + \lambda\,\eta(u_\eps(t,x);z)\big)
 \psi(t,x)\,d\lambda\,m(dz)\,dx\,dt \notag \\
&+ \int_0^t\int_{\D}\int_{\D} \frac{\Big(u_\eps(s,x) - u_\eps(s,y)\Big) \Big(\psi \,\beta^\prime(u_\eps)(s,x)-(\psi\,\beta^\prime(u_\eps))(s,y) \Big)}{|x-y|^{d+2\lambda}}\,dx\,dy\,ds
\end{align*}

We first assume that $\psi=1$, so that $\int_\D \beta^{\prime\prime}(u_\epsilon)\,f(u_\epsilon)\nabla u_\epsilon \,dx=0$ by classical arguments. Moreover, we assume that $\beta(u)=\frac12 u^2$, then we get 
 \begin{align*}
& \int_\D u_\epsilon^2(t,x)\,dx + 2 \eps \int_0^t \int_\D |\nabla u_\eps(s,x)|^2\,dx\,ds
+2 \int_0^t \int_{\D} \int_{\D} \frac{\big(u_\eps(s,x)- u_\eps(s,y) \big)^2}{|x-y|^{d+2\lambda}}\,dx\,dy\,ds \\
& \qquad = \int_\D  (u_0^\eps(x))^2\,dx+  2 \sum_{k\ge 1} \int_0^t\int_{\D} g_k(u_\eps(s,x)) \,u_\eps(s,x) \,dx\,dW(s) + \frac12 \int_0^t\int_{\D} \mathbb{G}^2(u_\eps(s,x))  \,dx\,ds \\
& \qquad \qquad + 2 \int_0^t \int_{\D}\int_{E} \int_0^1 \eta(u_\eps(s,x);z)\, \big(u_\eps(s,x) + \lambda\,\eta(u_\eps(s,x);z)\big)\,d\lambda\,\widetilde{N}(dz,ds)\,dx  \\
& \qquad \qquad  \qquad \qquad +\int_0^t \int_{\D} \int_{E}   \eta^2(u_\eps(s,x);z)\,m(dz)\,dx\,ds.
\end{align*}
After taking expectation in the above equality, we obtain
\begin{align*}
\E\Big[\big\|u_\eps(t)\big\|_{L^2(\R^d)}^2\Big]  &+ \eps \int_0^t \E\Big[\big\|\grad u_\eps(s)\big\|_{L^2(\R^d)}^2\Big]\,ds + \int_0^t \E\Big[\big\| u_\eps(s)\|_{H^{\lambda}(\R^d)}^2\Big]\,ds \\
&\le \E\Big[\big\|u_0^\eps\big\|_{L^2(\R^d)}^2\Big] + C \int_0^t \E\Big[\big\|u_\eps(s)\big\|_{L^2(\R^d)}^2\Big]\,ds.
\end{align*} 
Finally, an application of the Gronwall's inequality yields
\begin{align*}
\sup_{0\le t\le T} \E\Big[\big\|u_\eps(t)\big\|_{L^2(\R^d)}^2\Big]  + \eps \int_0^T \E\Big[\big\|\grad u_\eps(s)\big\|_{L^2(\R^d)}^2\Big]\,ds + \int_0^T \E\Big[\big\| u_\eps(s)\|_{H^{\lambda}(\R^d)}^2\Big]\,ds \le C.
\end{align*}


\appendix
\section{Bounded Variation Estimates}
\label{sec:apriori+existence}

Here we aim to derive uniform spatial BV bound for the solutions of fractional stochastic balance laws driven by L\'{e}vy noise \eqref{eq:stoc_con_brown} under
the assumptions $\ref{A1}$-$\ref{A5}$. Like its deterministic counter part, we first secure uniform spatial BV bound for the viscous solutions, i.e., solutions of \eqref{eq:viscous-Brown}. Regarding this, we have the following theorem. 
\begin{thm}
\label{thm:bv-viscous}
Let the assumptions $\ref{A1}$-$\ref{A5}$ hold. For $\eps>0$, let $u_\eps(t,x)$  be a solution to the Cauchy problem \eqref{eq:viscous-Brown}.
Then there exists a constant $C>0$, independent of $\eps$, such that for any time $t>0$, 
\begin{align*}
\sup_{\eps>0} \E\Big[ \|u_\eps(t)\|_{L^1(\R^d)}\Big] \le  C \, \|u_0\|_{L^1(\R^d)}, \quad
&\sup_{\eps>0} \E\Big[ TV_x(u_\eps(t))\Big] \le TV_x(u_0).
 \end{align*}
\end{thm}

\begin{rem}
In view of the lower semi-continuity property and the positivity of the total variation $TV_x$, we point out that $u \mapsto \E[TV_x(u)]$ makes sense for 
 any $u\in L^1(\Omega \times \R^d)$ as a real-extended lsc convex function.
\end{rem}

\begin{proof} We shall divide the proof of the above theorem in two parts: \\[1mm]
{\bf Step-I.}  As we have already seen that under natural assumptions on initial data, flux 
functions, and noise coefficients, viscous equation \eqref{eq:viscous-Brown} has weak solutions $u_\eps$ and moreover \eqref{bounds:a-priori-viscous-solution} holds. To that context, under additional assumption on the   
initial data, $u_0 \in L^1(\Omega\times \R^d)$, we show that for fixed $\eps>0$, $u_\eps \in L^1(\Omega\times \Pi_T)$. 
To do this, we proceed as follows: let  us consider a convex, even, approximation of the absolute-value function $\beta_\xi$.

Then, by applying It\^{o}-L\'{e}vy formula to $\int_{\R^d}\beta_\xi(u_\eps(t,x))\,dx$ and taking expectation, we conclude 
\begin{align}
 & \E\Big[ \int_{\R^d} \beta_{\xi}(u_\eps(t,x))\,dx\Big] + \eps\,\E\Big[ \int_0^t \int_{\R^d} \beta_\xi^{\prime\prime}(u_\eps(s,x))\, |\grad u_\eps(s,x)|^2  \,dx \Big] \notag \\
 &+\frac12 \E\Bigg[\int_0^t\int_{\D}\int_{\D} \frac{\Big(u_\eps(s,x) - u_\eps(s,y)\Big) \Big(\beta^\prime(u_\eps)(s,x)-(\beta^\prime(u_\eps))(s,y) \Big)}{|x-y|^{d+2\lambda}}\,dx\,dy\,ds \Bigg] \notag \\
 & \le  \E\Big[ \int_{\R^d} \beta_\xi(u_0(x))\,dx\Big]
 - \E \Big[  \int_0^t \int_{\R^d}  \beta_\xi^{\prime\prime}(u_\eps(s,x))  f(u_\eps(s,x))\cdot \grad u_\eps(s,x) \,dx\,ds\Big] \notag \\
  & \quad  + \E\Big[ \int_0^t \int_{\R^d} \int_{E}\int_{0}^1 (1-\lambda) \eta^2(u_\eps(s,x);z)
  \beta_\xi^{\prime\prime}\big(u_\eps(s,x)+ \lambda \eta(u_\eps(s,x);z)\big)\,d\lambda\,m(dz)\,dx\,ds\Big] \notag \\
    &  \hspace{2cm} + \frac{1}{2} \E\Big[ \int_0^t \int_{\R^d} \mathbb{G}^2(u_\eps(s,x)) \beta_\xi^{\prime\prime}(u_\eps(s,x))
    \,dx\,ds\Big]. \label{esti:l1-ness-0}
\end{align}
Since $\beta_\xi$  is a convex function, we have from \eqref{esti:l1-ness-0}
\begin{align}
 & \E\Big[ \int_{\R^d} \beta_{\xi}(u_\eps(t,x))\,dx\Big]- \E\Big[ \int_{\R^d} \beta_\xi(u_0(x))\,dx\Big] \notag \\
 & \quad \le - \E \Big[  \int_0^t \int_{\R^d}  \beta_\xi^{\prime\prime}(u_\eps(s,x))  f(u_\eps(s,x))\cdot \grad u_\eps(s,x) \,dx\,ds\Big] \notag \\
  & \quad  + \E\Big[ \int_0^t \int_{\R^d} \int_{E}\int_{0}^1 (1-\lambda) \eta^2(u_\eps(s,x);z)
  \beta_\xi^{\prime\prime}\big(u_\eps(s,x)+ \lambda \eta(u_\eps(s,x);z)\big)\,d\lambda\,m(dz)\,dx\,ds\Big] \notag \\
    &  \hspace{2cm} + \frac{1}{2} \E\Big[ \int_0^t \int_{\R^d} \mathbb{G}^2(u_\eps(s,x)) \beta_\xi^{\prime\prime}(u_\eps(s,x))
    \,dx\,ds\Big] \notag \\
    & := \mathcal{A}_1(\eps,\xi) + \mathcal{A}_2(\eps,\xi) + \mathcal{A}_3(\eps,\xi).\label{esti:l1-ness-1}
\end{align}
Next, we estimate each of the above terms separately. 
Let us first remark that a simple application of chain-rule implies that $\mathcal{A}_1(\eps,\xi)=0$.
We now consider the term $\mathcal{A}_2(\eps,\xi)$. In fact, in view of Assumption \ref{A5} and \eqref{eq:approx to abosx}, we have
 \begin{align*}
  0 \leq &\eta^2(u_\eps(s,x);z)
  \beta_\xi^{\prime\prime}\big(u_\eps(s,x)+ \lambda \eta(x,u_\eps(s,x);z)\big)
  \\
  \leq & (\lambda^*)^2 h^2(z) | u_\eps(s,x)|^2
  \beta_\xi^{\prime\prime}\big(u_\eps(s,x)+ \lambda \eta(u_\eps(s,x);z)\big)
  \\
  \leq & \frac{(\lambda^*)^2 h^2(z)}{(1-\lambda\lambda^*)^2} | u_\eps(s,x)+ \lambda \eta(u_\eps(s,x);z)|^2
  \beta_\xi^{\prime\prime}\big(u_\eps(s,x)+ \lambda \eta(u_\eps(s,x);z)\big)
  \\
  \leq & C\frac{(\lambda^*)^2 h^2(z)}{(1-\lambda\lambda^*)^2}  \beta_\xi\big(u_\eps(s,x)+ \lambda \eta(u_\eps(s,x);z)\big)
  \leq C\frac{(\lambda^*)^2 h^2(z)}{(1-\lambda\lambda^*)^2}  \beta_\xi\big(|u_\eps(s,x)|+ |\eta(u_\eps(s,x);z)|\big)
\\
  \leq & C\frac{(\lambda^*)^2 h^2(z)}{(1-\lambda\lambda^*)^2}  \beta_\xi\big((1+\lambda^*)|u_\eps(s,x)|\big)
    \leq  C h^2(z)\frac{(\lambda^*)^2(1+\lambda^*)^2}{(1-\lambda\lambda^*)^2}  \beta_\xi\big(u_\eps(s,x)\big),
\end{align*}
and this implies that 
\begin{align}
\big|\mathcal{A}_2(\eps,\xi)\big|  
  \le&
C\E\Big[ \int_0^t \int_{\R^d} \int_{E} h^2(z)\,m(dz)  \beta_\xi\big(u_\eps(s,x)\big)\,dx\,ds\Big]
\leq 
C\E\Big[ \int_0^t \int_{\R^d}  \beta_\xi\big(u_\eps(s,x)\big)\,dx\,ds\Big]. \label{esti:a3-l1-ness} 
\end{align}
Again, we use assumption \ref{A4} to conclude 
\begin{align}
  \big|\mathcal{A}_3(\eps,\xi)\big|  \leq 
C\E\Big[ \int_0^t \int_{\R^d}  \beta_\xi\big(u_\eps(s,x)\big)\,dx\,ds\Big]. \label{esti:a4-l1-ness} 
\end{align}
Thus, combining all the above estimates \eqref{esti:a3-l1-ness}-\eqref{esti:a4-l1-ness} in \eqref{esti:l1-ness-1}, we arrive at
\begin{align*}
 \E\Big[ \int_{\R^d} \beta_\xi(u_\eps(t,x))\,dx\Big] 
 \leq C\int_0^t\E\Big[  \int_{\R^d} \beta_\xi \big(u_\eps(s,x)\big)\,dx\Big]\,ds + \E\Big[ \int_{\R^d} \beta_\xi(u^{\eps}_0(x))\,dx\Big], 
\end{align*}
and this implies 
\begin{align}
\E\Big[  \int_{\R^d}  \beta_\xi \big(u_\eps(t,x)\big)\,dx\Big] \leq C
\E\Big[ \int_{\R^d} \beta_\xi (u^{\eps}_0(x))\,dx\Big].
 \label{eq:final-Bv-estimate}
\end{align}
Passing to the limit with respect to $\xi$ yields (thanks to Fatou's lemma for the lefthand side and Lebesgue theorem for the righthand one)
\begin{align}
\label{L1-bound}
\E \Big[\int_{\R^d} \big|u_\eps(t,x)\big|\,dx\Big] \le C \E \Big[\int_{\R^d} \big|u^{\eps}_0(x)\big|\,dx\Big] \le C \E \Big[\int_{\R^d} \big|u_0(x)\big|\,dx  \Big].
 \end{align}
This implies that, $u_\eps \in L^1(\Omega \times \Pi_T)$, for every fixed $\eps>0$. 

 \noindent {\bf Step-II.} For the second part, we proceed as follows: 
Set $\eps>0$ and let $u_\eps$ be the strong solution to the problem \eqref{eq:viscous-Brown} and $v_\eps$ be a strong solution to the stochastic equation 
\begin{align}
dv_\eps(t,x) -\Delta v_\eps(t,x)\,dt  + \mathcal{L}_{\lambda}[v_\eps(t, \cdot)](x)\,dt  &- \mbox{div}_x f(v_\eps(t,x)) \,dt \notag \\
& = \sigma(v_\eps(t,x))\,dW(t) + \int_{E} \eta(v_\eps(t,x);z)\widetilde{N}(dz,dt),\notag 
\end{align}
with $v_\eps(0,x)=v_0(x)$. Then, it is evident that $u_\eps -v_\eps$ is a stochastic weak solution to the problem 
\begin{align*}
 \begin{cases}
  d(u_\eps(t,x)-v_{\eps}(t,x))- \Delta (u_\eps(t,x))-v_\eps(t,x)\,dt + \mathcal{L}_{\lambda}[u_\eps(t, \cdot)- v_\eps(t, \cdot)](x)\,dt \\   \hspace{2cm} - \mbox{div}_x\big( f(u_\eps(t,x))- f(v_\eps(t,x))\big)\,dt
 = \big(\sigma(u_\eps(t,x))-\sigma(v_\eps(t,x))\big)\,dW(t) \\
 \hspace{4cm}+ \int_{E}\big(\eta(u_\eps(t,x);z)-\eta(v_\eps(t,x);z)\big)\,\widetilde{N}(dz,dt),\\
  u_\eps -v_\eps\big|_{(t=0,x)} = u_0(x)-v_0(x).
 \end{cases}
 \end{align*}
Next, we apply It\^{o}-L\'{e}vy formula (as proposed in Fellah \cite{martina,fellah} and Biswas et al. \cite{BisMajVal}) to $\int_{\R^d}\beta_\xi(u_\eps-v_\eps)dx$, 
where for technical reasons, one assumes that $\exists a>0, \forall r \in [-a,a], \beta^{\prime\prime} (r)\geq a$ and $\exists A>0, \forall r \in \R, \beta^{\prime\prime} (r)\leq A$. Thus, there exists a constant $C>0$ such that $r^2\beta_\xi^{\prime\prime} (r) \leq C\beta_\xi(r)$ and $\beta_\xi(\alpha r) \leq \alpha^2 \beta_\xi(r)$ for any $\alpha \geq  1$.
A classical example is given by $\beta'(r)=\max[-1,\min(r,1)]$.

We then take expectation, and the result is
\begin{align}
& \E\Big[ \int_{\R^d} \beta_\xi \big(u_\eps(t,x)-v_\eps(t,x)\big)\,dx\Big] \notag \\
&= \E\Big[ \int_{\R^d} \beta_\xi \big(u_0(x)-v_0(x)\big)\,dx\Big] \notag \\
& - \E\Big[ \int_{\R^d}
\int_0^t \beta_\xi^{\prime\prime} \big( u_\eps(s,x)-v_\eps(s,x)\big) \grad \big(u_\eps(s,x) - v_\eps(s,x)\big)\cdot \grad \big(u_\eps(s,x)-v_\eps(s,x)\big)  \,ds\,dx\Big] \notag \\
& - \E\Big[ \int_{\R^d}
\int_0^t \Big<\beta_\xi^{\prime} \big( u_\eps(s,x)-v_\eps(s,x)\big), \,\mathcal{L}_{\lambda}[u_\eps(t, \cdot)- v_\eps(t, \cdot)](x)  \Big> \,ds\,dx\Big] \notag \\
&\quad - \E\Big[ \int_{\R^d}
\int_0^t \beta_\xi^{\prime\prime} \big( u_\eps(s,x)-v_\eps(s,x)\big) \big(f(u_\eps(s,x))-f(v_\eps(s,x))\big)\cdot \grad \big(u_\eps(s,x)-v_\eps(s,x)\big)  \,ds\,dx\Big] \notag \\
& \qquad  + \frac{1}{2}  \E\Big[ \sum_{k\ge 1} \int_{\R^d}
\int_0^t \beta_\xi^{\prime\prime} \big( u_\eps(s,x)-v_\eps(s,x)\big) \big(g_k(u_\eps(s,x))-g_k(v_\eps(s,x))\big)^2 \,ds\,dx\Big] \notag \\
&\quad \quad  + \E\Big[ \int_{\R^d}
\int_0^t \int_{E} \int_0^1 (1-\lambda) \beta_\xi^{\prime\prime} \Big( u_\eps(s,x)-v_\eps(s,x) + \lambda \big(\eta(u_\eps(s,x);z)-\eta(v_\eps(s,x);z)\big)\Big)\notag \\
 & \hspace{6cm} \times \big(\eta(u_\eps(s,x);z)-\eta(v_\eps(s,x);z)\big)^2 \,d\lambda\,m(dz)\,ds\,dx\Big] \notag \\
& := \mathcal{A} + \mathcal{B}+ \mathcal{C} + \mathcal{D} +\mathcal{E} +\mathcal{F}.
\end{align}
Our aim is to estimate each of the above terms separately. First observe that the term $\mathcal{B}$ is non-positive and
denoting $w_\eps(s,x) =\big( u_\eps(s,x)-v_\eps(s,x)\big)$, we have 
\begin{align*}
\Big<\beta_\xi^{\prime} \big( u_\eps(s,x)-v_\eps(s,x)\big),& \,\mathcal{L}_{\lambda}[u_\eps(t, \cdot)- v_\eps(t, \cdot)](x)  \Big> \\
& = \frac12\int_{\D}\int_{\D} \frac{\Big[\beta_\xi^{\prime}( w_\eps(s,x)) - \beta_\xi^{\prime} (w_\eps(s,y)) \Big] \Big[ \big( w_\eps(s,x)-w_\eps(s,y)\big)\Big]}{|x-y|^{d+2\lambda}}\,dx\,dy.
\end{align*}
Since $\beta_\xi^{\prime}$ is a non-increasing function, we conclude that $\mathcal{C}$ is also non-positive.

Next we move on to estimate the flux term $\mathcal{D}$. In view of the Young's inequality, one has
\begin{align}
\mathcal{D} \le& \frac{\eps}{4}  \E\Big[ \int_{\R^d}
\int_0^t \beta_\xi^{\prime\prime} \big( u_\eps(s,x)-v_\eps(s,x)\big) \Big|\grad \big(u_\eps(s,x)-v_\eps(s,x)\big)\Big|^2  \,ds\,dx\Big] \notag \\ 
& \qquad \qquad + C(\eps) \, \E\Big[ \int_{\R^d}
\int_0^t \beta_\xi^{\prime\prime} \big( u_\eps(s,x)-v_\eps(s,x)\big) \big| f(u_\eps(s,x))-f(v_\eps(s,x))\big|^2  \,ds\,dx\Big] \notag \\ 
&:= \mathcal{D}_1 + \mathcal{D}_2. \notag
\end{align}
In view of the Lipschitz continuity of $f$ and \eqref{eq:approx to abosx}, it is easy to see that 
$\mathcal{D}_2 \le \eps(\xi)$ where $\eps(\xi) \goto 0$ as $\xi \goto 0$. 
Moreover, observe that $\mathcal{B} + \mathcal{D}_1 $ is non-positive. 

In view of the assumption \ref{A4}, one has 
\begin{align*}
\mathcal{E} \le K \E\Big[\int_{\R^d}\int_0^t \beta_\xi^{\prime\prime} \big( u_\eps(s,x)-v_\eps(s,x)\big) \big|u_\eps(s,x)-v_\eps(s,x)\big|^2\,ds\,dx\Big].
\end{align*}
A similar calculation to the one with $\mathcal{D}_2$ reveals that $\mathcal{E}\le \eps(\xi)$  where $\eps(\xi) \goto 0$ as $\xi \goto 0$.

Now we move on to estimate $\mathcal{F}$. The estimate of $\mathcal{F}$ is similar to the one of $\mathcal{E}$, following the calculations proposed in \cite{BisKoleyMaj}, and we can  conclude that  $\mathcal{F} \goto 0$, as $\xi \goto 0$.

Combining all the above estimates, we arrive at  
\begin{align}
\E\Big[ \int_{\R^d} \beta_\xi \big(u_\eps(t,x)-v_\eps(t,x)\big)\,dx\Big]
 \le  \eps(\xi) + \E\Big[ \int_{\R^d} \beta_\xi \big(u^{\eps}_0(x)-v^{\eps}_0(x)\big)\,dx\Big].
  \label{esti:final-bv}
\end{align}
Keeping $\eps >0$ fixed, we pass to the limit $\xi\goto 0$ in \eqref{esti:final-bv} and the resulting expressions reads as 
\begin{align*}
\E\Big[ \int_{\R^d} \big|u_\eps(t,x)-v_\eps(t,x)\big|\,dx\Big] \le \E\Big[ \int_{\R^d} \big|u^{\eps}_0(x)-v^{\eps}_0(x)\big|\,dx\Big]\le\E\Big[ \int_{\R^d} \big|u_0(x)-v_0(x)\big|\,dx\Big] .
\end{align*} 
Assume that $v_0(x)=u_0(x+c)$ for fixed $c\in \R^d$. Then, since $\sigma$ and $\eta$ do not depend on $x$ explicitly, by uniqueness of the weak solution, one can conclude that $v_\eps(t,x)=u_\eps(t,x+c)$ and hence
 \begin{align*}
 \E\Big[ \int_{\R^d}  \frac{ \big|u_\eps(t,x)-v_\eps(t,x)\big|}{|c|}\,dx\Big]& \le \E\Big[ \int_{\R^d} \frac{\big|u_0(x)-u_0(x+c)\big|}{|c|}\,dx\Big]\le C, 
 \end{align*}
independent of $c$, if $u_0 \in BV(\R^d)$. This implies that, for any $t>0$, since $v_\eps(t,x)=u_\eps(t,x+c)$
 \begin{align}
\sup_{\eps>0} \E\Big[TV_x(u_\eps(t))] \le  \E\Big[TV_x(u_0)\Big] . \label{bv-bound}
 \end{align}
 This completes the proof.
 \end{proof}
 
In view of the well-posedness results (cf. Section~\ref{uniqueness}), we conclude that under the assumptions \ref{A1}-\ref{A5}, the family $\{u_\eps(t,x)\}_{\eps>0}$ converges to the unique entropy solution $u(t,x)$ of the underlying problem \eqref{eq:stoc_con_brown}. Now, our aim is to show that $u(t,x)$ is actually a spatial BV solution of \eqref{eq:stoc_con_brown} provided the initial function $u_0$ lies in $L^2 \cap BV(\R^d)$. 
Since $u_\eps$ converges to $u$ weakly in $L^2(\Omega\times(0,T)\times\R^d)$, we have 
\begin{align*}
\E\Big[\int_{\Pi_T}|u| \,dx\,dt\Big] \leq \lim\inf_\eps \E\Big[\int_{\Pi_T}|u_\eps| \,dx\,dt\Big] \leq M,
\end{align*}
thanks to \eqref{L1-bound} and $u \in L^1(\Omega\times(0,T)\times\R^d)$.

In view of the lower semi-continuity property of $TV_x$ and 
Fatou's lemma, we have, for a.e. $t>0$,
\begin{align*}
 \E\Big[ TV_x(u(t))\Big] \le \liminf_{\eps \goto 0} \E\Big[ TV_x(u_\eps(t))\Big] \le \E \Big[ TV_x(u_0)\Big],
\end{align*}
where the last inequality follows from Theorem~\ref{thm:bv-viscous}. Thus, $u(t,x)$ is a function of bounded variation in spatial variable.
In other words, we have existence of a BV entropy solution for the problem \eqref{eq:stoc_con_brown} given by the following theorem.
 \begin{thm} [BV entropy solution]
 \label{thm:existence-bv}
 Suppose that the assumptions \ref{A1}-\ref{A5} hold. Then there exists a constant $C>0$, and an unique BV entropy solution
 of \eqref{eq:stoc_con_brown} such that for a.e. $t>0$
\begin{align*}
   \E \Big[|u(t,\cdot)|_{BV(\R^d)} \Big] \le C\E \Big[|u_0|_{BV(\R^d)} \Big].
\end{align*}	
\end{thm}


\section{On the Fractional Laplace Operator}
\label{FractionalLaplace}

Let $\fr[\varphi]$ denotes the fractional Laplace operator $(-\Delta)^{\lambda}[\varphi]$ of order $\lambda \in (0,1)$. Depending on the regularity of $\varphi$, several definitions can be proposed and we recapitulate the ones pertaining to this manuscript. 
\\[0.2cm]
A first definition is given by the Fourier-transform: assume that $\varphi \in L^2(\R^d)$ and that $|\cdot|^{2\lambda}\hat \varphi \in L^2(\R^d)$ too, then, $\fr[\varphi] \in L^2(\R^d)$ is given by $\widehat{\fr[\varphi]} = |\cdot|^{2\lambda}\hat \varphi$. Note that this corresponds to $\varphi$ element to the fractional Sobolev space $H^{2\lambda}(\R^d)$.
\\[0.2cm]
A second definition is a pointwise one:
\begin{align*}
\fr[\varphi](x) := c_{\lambda}\, \text{P.V.}\, \int_{|z|>0} \frac{\varphi(x) -\varphi(x+z)}{|z|^{d + 2 \lambda}} \,dz=c_{\lambda}\, \lim_{\epsilon \to 0^+}\, \int_{|z|>\epsilon} \frac{\varphi(x) -\varphi(x+z)}{|z|^{d + 2 \lambda}} \,dz,
\end{align*}
for some constants $c_{\lambda}=\frac{4^\lambda \Gamma(\lambda + \frac d2)}{\pi^{\frac d2}|\Gamma(-\lambda)|}$, for a measurable $\varphi$ such that the integral and the limit exist.
\\[0.2cm]
Since, for any positive $\eps$, $z \mapsto \frac{1_{|z|\geq \eps}}{|z|^{d + 2 \lambda}} \in L^p(\R^d)$ for any $p \in [1,+\infty]$,  if $\varphi \in L^2(\R^d)$ and $\eps$ is small, the following integral $\int_{|z|>\eps} \frac{\varphi(x) -\varphi(x+z)}{|z|^{d + 2 \lambda}} \,dz$ exists. 
\\
Denote by $\vec \chi(x)$ a given vector and note that 
\begin{align*}
\int_{|z|>\epsilon} \frac{\varphi(x) -\varphi(x+z)}{|z|^{d + 2 \lambda}} \,dz = \int_{|z|>1} \frac{\varphi(x) -\varphi(x+z)}{|z|^{d + 2 \lambda}} \,dz + \int_{1>|z|>\epsilon} \frac{\varphi(x) -\varphi(x+z)+z.\vec \chi(x)}{|z|^{d + 2 \lambda}} \,dz
\end{align*}
and, the existence of the principal value of the above integral is related to the existence of the limit when $\epsilon \to 0^+$ of $\int_{1>|z|>\epsilon} \frac{\varphi(x) -\varphi(x+z)+z.\vec \chi(x) }{|z|^{d + 2 \lambda}} \,dz$.
\\[0.2cm]
Assume in a first step that $\varphi \in \mathcal{S}(\R^d)$, so that if $\vec \chi(x) = \nabla \varphi(x)$, 
by Taylor's expansion : 
\begin{align*}
\varphi(x+z) -\varphi(x)-z.\nabla \varphi(x) = \int_0^1(1-t)D^2\varphi(x+tz).(z,z)dt
\end{align*}
one gets that $|\varphi(x+z) -\varphi(x)-z.\nabla \varphi(x)| \leq C(D^2 \varphi)|z|^2$. Then, the limit exists by Lebesgue's theorem and 
\begin{align*}
\fr[\varphi](x) = c_\lambda\int_{\R^d} \frac{\varphi(x) -\varphi(x+z)+z.\nabla \varphi(x) 1_{|z|<1}}{|z|^{d + 2 \lambda}} \,dz.
\end{align*}
Thus, for any $p \in [1,+\infty]$, there exists $c_{p,\lambda},c_{d,\lambda}\geq 0$ such that 
$$|\fr[\varphi](x)| \leq c_{p,\lambda}\Big[|\varphi(x)| + \|\varphi\|_{L^p(\R^d)}\Big] + c_{d,\lambda} \|D^2 \varphi\|_{\infty} \leq C\Big[\|\varphi\|_{L^p(\R^d)}+\|\varphi\|_{W^{2,\infty}(\R^d)}\Big].$$
Moreover, making use of Cauchy-Schwarz inequality, we have
\begin{align*}
&\int_{1>|z|>\epsilon} \Big|\frac{\varphi(x) -\varphi(x+z)+z.\nabla \varphi(x)}{|z|^{d + 2 \lambda}}\Big| \,dz =
\int_{1>|z|>\epsilon} \Big|\frac{\int_0^1(1-t)D^2\varphi(x+tz).(z,z)dt}{|z|^{d + 2 \lambda}}\Big| \,dz  \\  
& \quad \leq \int_{1>|z|>\epsilon} \frac{\int_0^1|D^2\varphi(x+tz)|dt}{|z|^{d + 2 (\lambda-1)}} \,dz \leq 
C_{\lambda,d}\int_0^1\Bigg[\int_{1>|z|} \frac{|D^2\varphi(x+tz)|^2}{|z|^{d + 2 (\lambda-1)}} \,dz\Bigg]^{1/2} dt, 
\end{align*}
and, by a density argument, for almost all $x$ (indep. of $\eps$), the same inequality holds if $\varphi \in H^2(\R^d)$.
\\
Since $\int_{\R^d}|\int_{|z|>1} \frac{\varphi(x) -\varphi(x+z)}{|z|^{d + 2 \lambda}} \,dz|^2dx < +\infty$, Fubini and monotone convergence theorems yield, for almost all $x$,  the integrability of the above left-hand side and $\fr[\varphi](x)$ exists, $x$ a.e. Moreover, 
\begin{align*}
\int_{\R^d}|\fr[\varphi](x)|^2dx \leq C \int_{\R^d}\int_0^1 \int_{1>|z|} \frac{|D^2\varphi(x+tz)|^2}{|z|^{d + 2 (\lambda-1)}} \,dz dt dx + C\int_{\R^d}\Big|\int_{|z|>1} \frac{\varphi(x) -\varphi(x+z)}{|z|^{d + 2 \lambda}} \,dz\Big|^2dx
\end{align*}
and $\fr[\varphi] \in L^2(\R^d)$.
\\[0.2cm]
Our last definition is a variational one and concerns $\varphi \in H^{\lambda}(\R^d)$, where $\fr[\varphi]$ is defined by the duality: 
\begin{align*}
\forall \psi \in H^{\lambda}(\R^d),\quad \langle \fr[\varphi] , \psi \rangle = \frac{c_\lambda}{2}\int_{\R^{2d}}\frac{[\varphi(x)-\varphi(y)][\psi(x)-\psi(y)]}{|x-y|^{d+2\lambda}}dxdy.
\end{align*}
Note that if $\varphi \in H^{\lambda}(\R^d)$ then $\mathcal{L}_{\lambda/2}[\varphi] \in L^{2}(\R^d)$ and $\langle \fr[\varphi] , \psi \rangle = \int_{\R^d} \mathcal{L}_{\lambda/2}[\varphi] \mathcal{L}_{\lambda/2}[\psi] \,dx$.
\\[0.2cm]
We close this section by recalling (\cite{Kwasnicki}) that if $\varphi \in L^2(\R^d)$ is such that $\fr[\varphi]$ exists and is in $L^2(\R^d)$ for one of the above definitions, the same holds for all the other definitions. 

\section{Solutions to $\fr[u]=\varphi$}
\label{SolPositive}
Assume that $\lambda \in (0,d/2)$ and that $0\neq \varphi \in \mathcal{D}(\D)$ is non negative. Then there is a solution in $H^\lambda(\D)$ to $\fr[u]=\varphi$ in $\R^d$, in the variational sense. Moreover, $u > 0$ too.

Indeed, following \cite{Stein,Kwasnicki}, we know that in the region $ 0< \lambda <d/2$, the solution is given by the Riesz potential formula i.e., there exists a positive constant $c$ such that
\begin{align*}
u(x) = I_{2\lambda}(\varphi)(x) = c\int_{\R^d}\frac{\varphi(y)}{|x-y|^{d-2\lambda}} dy.
\end{align*}
As a consequence, $u>0$ in $\D$. Moreover, there exists a constant $C$ exists such that 
$$\|u\|_{L^2(\R^d)} \leq C\, \|\varphi\|_{L^{\frac{2d}{d+2\alpha}}(\R^d)}.$$
Since $u(x) = c\displaystyle\int_{\R^d}\frac{\varphi(z+x)}{|z|^{d-\alpha}}dz$, by differentiation, for any multi-index $\beta$, 
$$D^\beta u(x) = c\int_{\R^d}\frac{D^\beta \varphi(z+x)}{|z|^{d-\alpha}}dz = I_\alpha (D^\beta \varphi)(x).$$
By density,  for any $m$, $\varphi \in W^{m,\frac{2d}{d+2\alpha}}(\R^d) \mapsto u \in H^m(\R^d)$ is linear and continuous.


\end{document}